\documentclass[10pt,a4paper]{article}

\usepackage{amsopn}

\usepackage{graphics}
\usepackage{graphicx}
\usepackage{epsfig}

\usepackage{amssymb}
\usepackage{amsthm}
\usepackage{amsmath}
\usepackage{amsfonts}
\usepackage{lipsum}
\usepackage[left=2.5cm,right=2.5cm,bottom=3.5cm,top=3.5cm]{geometry}
\usepackage{empheq}
\usepackage{epstopdf}
\usepackage{changepage}
\usepackage{rotating}
\usepackage{adjustbox}
\usepackage{graphbox}
\usepackage{color}
\usepackage{xcolor}
\usepackage{dsfont}
\usepackage{dutchcal}
\usepackage{hyperref} % referencias cruzadas 
\usepackage{float}
\usepackage{resizegather}
\usepackage{multicol}
\usepackage{booktabs} % mejora la calidad de las tablas y proporciona nuevas opciones
\usepackage{colortbl} % coloreado en tablas
\usepackage{multirow} % tablas con celdas de varias lineas
\usepackage{longtable,tabu} % tablas
\usepackage{hhline}   % mejora las lineas horizontales en las tablas
\usepackage{anyfontsize} % tamanos de fuente personalizados
\usepackage{everysel}
\usepackage{psfrag}
\usepackage{titlesec}
\usepackage{bbm}
\usepackage{bm}
\usepackage{pgf,tikz}

\definecolor{darkgreen}{rgb}{0.1,0.6,0.1}

\newcommand{\q}{\mathbf{q}}
\newcommand{\p}{\mathbf{p}}
\newcommand{\f}{\mathbf{f}}
\newcommand{\g}{\mathbf{g}}

\renewcommand{\v}{\mathbf{v}}

\newcommand{\AAA}{\mathbf{A}}
\newcommand{\err}{\mathcal{r}}

\newcommand{\normal}{\mathbf{n}} 
\newcommand{\halb}{\frac{1}{2}}

\allowdisplaybreaks

\newtheorem{theorem}{Theorem}

\hypersetup{
	pdftitle={On thermodynamically compatible finite volume schemes for continuum mechanics},    % title
	pdfauthor={S. Busto. M Dumbser, I. Peshkov, E. Romenski},     % author
	pdfcreator={S. Busto. M Dumbser, I. Peshkov, E. Romenski},   % creator of the document
	colorlinks=true,       % false: boxed links; true: colored links
}

\title{On thermodynamically compatible finite volume schemes for continuum mechanics}

\begin{document}

\begin{center}
	\textbf{ \Large{On thermodynamically compatible finite volume schemes for continuum mechanics} }
	
	\vspace{0.5cm}
	{S. Busto\footnote{saray.busto@uvigo.es}, M. Dumbser\footnote{michael.dumbser@unitn.it}, I. Peshkov\footnote{ilya.peshkov@unitn.it}, E. Romenski\footnote{evrom@math.nsc.ru}}
	
	\vspace{0.2cm}
	{\small
		\textit{$^{(1)}$ Department of Applied Mathematics I, Universidade de Vigo, Campus As Lagoas, 36310 Vigo, Spain}
		
		\textit{$^{(2,3)}$ Department of Civil, Environmental and Mechanical Engineering, University of Trento, Via Mesiano 77, 38123 Trento, Italy }
		
		\textit{$^{(4)}$ Sobolev Institute of Mathematics, 4 Acad. Koptyug Avenue, 630090 Novosibirsk, Russia}
	}
\end{center}

% % % % % % % % % % % % % % % % % % % % % % % % % % % % % %
%                   Abstract                              %
% % % % % % % % % % % % % % % % % % % % % % % % % % % % % %
\hrule
\vspace{0.4cm}

\begin{center}
	\textbf{Abstract}
\end{center} 

\vspace{0.1cm}

In this paper we present a new family of semi-discrete and fully-discrete finite volume schemes for overdetermined, hyperbolic and thermodynamically compatible PDE systems. In the following we will denote these methods as HTC schemes. 
In particular, we consider the Euler equations of compressible gasdynamics, as well as the more complex  Godunov-Peshkov-Romenski (GPR) model of continuum mechanics, which, at the aid of suitable relaxation source terms, is able to describe nonlinear elasto-plastic solids at large deformations as well as viscous fluids as two special cases of a more general first order hyperbolic model of continuum mechanics.
 The main novelty of the schemes presented in this paper lies in the fact that we solve the \textit{entropy inequality} as a primary evolution equation rather than the usual total energy conservation law. Instead, total energy conservation is achieved as a mere consequence of a thermodynamically compatible discretization of all the other equations. For this, we first construct a discrete framework for the compressible Euler equations that mimics the continuous framework of Godunov's seminal paper \textit{An interesting class of quasilinear systems} of 1961 \textit{exactly} at the discrete level. All other terms in the governing equations of the more general GPR model, including non-conservative products, are judiciously discretized in order to achieve discrete thermodynamic compatibility, with the exact conservation of total energy density as a direct consequence of all the other equations. As a result, 
the HTC schemes proposed in this paper are provably marginally stable in the energy norm and satisfy a discrete entropy inequality by construction. 
We show some computational results obtained with HTC schemes in one and two space dimensions,  considering both the fluid limit as well as the solid limit of the governing partial differential equations.

% % % % % % % % % % % % % % % % % % % % % % % % % % % % % %
%                   Keywords                              %
% % % % % % % % % % % % % % % % % % % % % % % % % % % % % %
\vspace{0.2cm}
\noindent \textit{Keywords:} 
	thermodynamically compatible finite volume schemes; 
	semi-discrete and fully-discrete Godunov formalism; 
	vanishing viscosity limit; 
	entropy inequality; 
	hyperbolic thermodynamically compatible PDE systems;
	overdetermined hyperbolic PDE systems;  
	unified GPR model for solid mechanics and fluid mechanics.	

\vspace{0.4cm}

\hrule

%  % See http://www.ams.org/msc/  
%  % Partial differential equations
%  35L40, % First-order hyperbolic systems  
%  % Numerical analysis
%  65M08. % Finite volume methods for initial value and initial-boundary value problems involving PDEs
%  % Relativity and gravitational theory
%  % Astronomy and astrophysics

\section{Introduction}

In his groundbreaking work \textit{An interesting class of quasilinear systems} \cite{God1961}   
published 60 years ago in 1961 Godunov discovered the connection between symmetric hyperbolicity in 
the sense of Friedrichs \cite{FriedrichsSymm} and thermodynamic compatibility, 10 years before the 
work of Friedrichs \& Lax on the same subject \cite{FriedrichsLax}. In subsequent work by Godunov 
\& Romenski and collaborators, the theory of symmetric hyperbolic and thermodynamic compatible 
(SHTC) systems was extended to a wide class of mathematical models in continuum physics, 
ranging from the magnetohydrodynamics (MHD) equations over nonlinear hyperelasticity to 
compressible multi-phase flows and relativistic gasdynamics, see e.g. 
\cite{God1972MHD,GodunovRomenski72,Godunov:1995a,GodRom2003,Rom1998,RomenskiTwoPhase2010,Godunov2012,GRGPR}.
 All SHTC systems can be rigorously derived from an underlying variational principle. A connection 
between SHTC systems and Hamiltonian mechanics was established in \cite{SHTC-GENERIC-CMAT}, 
emphasizing a peculiar role of the energy potential (Hamiltonian), while an extension to continuum 
mechanics with torsion was provided in \cite{PRD-Torsion2018}. 

Notwithstanding the mathematical elegance of the SHTC framework, to the best knowledge of the 
authors it was up to now never directly carried over to the discrete level. Most existing papers on 
thermodynamically compatible schemes are based on the ideas of the seminal work of Tadmor 
\cite{Tadmor1}, in which a discrete compatibility with the \textit{entropy equation} is sought, 
rather than a discrete compatibility with the \textit{total energy} conservation, as suggested by the 
SHTC framework. A fully discrete entropy-stable scheme has been recently forwarded in 
\cite{Ranocha2020}, while the convergence of entropy-stable schemes was proven in 
\cite{Chatterjee2020}.  For high order entropy-compatible schemes the reader is referred to 
\cite{GassnerSWE,ShuEntropyMHD1,ShuEntropyMHD2,GassnerEntropyGLM,Hennemann2021} and references 
therein. 
In \cite{Fjordholm2012,AbgrallBT2018} entropy compatible schemes were applied to non-conservative hyperbolic equations. Last, but not least, we also would like to mention the general framework for the construction of numerical methods that satisfy additional extra conservation laws recently introduced by Abgrall in \cite{Abgrall2018}. A first attempt to achieve discrete energy conservation as a consequence of all other equations was made in \cite{SWETurbulence} for a novel hyperbolic model of unsteady turbulent shallow water flows.  
\textcolor{black}{For compatible schemes in the context of Lagrangian hydrodynamics, where total energy conservation is obtained as a consequence of the discrete mass, momentum and internal energy equations, see the interesting papers \cite{CaramanaCompatible1,CaramanaCompatible2}, while a fully-discrete compatible kinetic energy preserving scheme was forwarded in \cite{Compatiblekin}. However, all aforementioned schemes address only the compressible Euler equations and not the full GPR model of continuum mechanics. } 

\textcolor{black}{The main contribution of this paper is thus a new thermodynamically compatible finite volume scheme for the GPR model of continuum mechanics \cite{Rom1998,PeshRom2014,GPRmodel} in which the discrete energy conservation law is obtained as a consequence of a compatible discretization of all the 
other equations. To the best knowledge of the authors, this is the first time that such a provably 
thermodynamically compatible scheme is proposed for the PDE system  \eqref{eqn.GPR}, which is able 
to describe solid mechanics and fluid mechanics at the same time. We stress that the main objective 
of this paper is \textit{not} to introduce a better or more efficient numerical scheme compared to 
existing methods, but to introduce a \textit{radically new concept}: direct discretization of the 
entropy inequality in order to obtain the discrete total energy conservation law as a consequence. 
We also would like to clearly indicate the three main shortcomings of the new method introduced in 
this paper:
	\begin{itemize}
		\item[i)] the numerical fluxes are only known \textit{implicitly} via path integrals of the physical flux in phase space; however, also other numerical methods are based on path integrals, like the Osher-Solomon flux \cite{osherandsolomon}, the entropy-consistent scheme of Tadmor \cite{Tadmor1} and the family of path-conservative schemes of Castro and Par\'es \cite{Castro2006,Pares2006}; 		
		\item[ii)] currently, in our new framework, a numerical scheme that provably satisfies total energy conservation at the fully-discrete level can only be achieved at the aid of a special implicit time integrator;
	    \item[iii)] in the case of the semi-discrete scheme, total energy conservation is in general lost at the fully-discrete level once a standard, nonsymplectic Runge-Kutta time discretization is employed. 
    \end{itemize}
}
The rest of this paper is organized as follows. In Section \ref{sec.model} we present the unified first order hyperbolic model of continuum mechanics (GPR model) under consideration. In Sections  \ref{sec.scheme.1d} and  \ref{sec.scheme.fd} the construction of thermodynamically compatible semi-discrete and fully-discrete finite volume schemes is explained for the one-dimensional case. In Section \ref{sec.scheme.2d} an extension to the general multi-dimensional case is presented, together with a  proof of nonlinear stability in the energy norm and a proof of the entropy inequality satisfied by the scheme. Numerical results are shown in Section \ref{sec.results} for the fluid and  the solid limits of the governing PDE system. The paper closes with some concluding remarks and an outlook to future work in Section \ref{sec.Conclusions}. 
 
%\clearpage

\section{Mathematical model and its structure}
\label{sec.model}

We consider the following first order hyperbolic model of continuum mechanics regularized with 
vanishing viscosity terms and which goes back to the work of Godunov 
\cite{God1961}, Godunov \& Romenski \cite{GodunovRomenski72,Rom1998, GodRom2003} and Peshkov \& 
Romenski, see \cite{PeshRom2014,GPRmodel}:   
\begin{subequations}\label{eqn.GPR}
\begin{eqnarray}
	&& \,\, \qquad \frac{\partial \rho}{\partial t}+\frac{\partial (\rho v_k)}{\partial x_k} 
	\textcolor{blue}{- \frac{\partial}{\partial x_m} \!\! \left( \epsilon \frac{\partial \rho}{\partial x_m} \right)} = 0, \label{eqn.conti} \\ 
	&& \,\, \qquad \frac{\partial \rho v_i}{\partial t}+\frac{\partial \left(\rho v_i v_k + p \, \delta_{ik}  
		\textcolor{red}{\, + \, \sigma_{ik} + \omega_{ik}} \right)}{\partial x_k}  
	\textcolor{blue}{-\frac{\partial}{\partial x_m} \!\! \left( \epsilon \frac{\partial \rho v_i}{\partial x_m} \right)} = 0,  \label{eqn.momentum} \\ 
	&& \,\, \qquad	\frac{\partial \rho S}{\partial t}+\frac{\partial \left( \rho S v_k  \textcolor{red}{\, + \, \beta_k}  
	\right)}{\partial x_k}  
\textcolor{blue}{ - \frac{\partial}{\partial x_m} \!\! \left( \epsilon \frac{\partial \rho S}{\partial x_m} \right) } = \textcolor{blue}{\Pi} \,   \textcolor{darkgreen}{+ \dfrac{\alpha_{ik} \alpha_{ik} }{\theta_1(\tau_1) T}  + 
	\dfrac{\beta_i \beta_i}{\theta_2(\tau_2) T} } \geq 0,   \label{eqn.entropy} \\		
	&& \,\, \qquad \textcolor{red}{\frac{\partial A_{i k}}{\partial t}+\frac{\partial (A_{im} v_m)}{\partial x_k} + 
	v_m \left(\frac{\partial A_{ik}}{\partial x_m}-\frac{\partial A_{im}}{\partial x_k}\right) }  \textcolor{blue}{-\frac{\partial}{\partial x_m} \!\! \left( \epsilon \frac{\partial A_{ik}}{\partial x_m} \right) }
	=  \textcolor{darkgreen}{-\dfrac{ \alpha_{ik} }{\theta_1(\tau_1)} },\label{eqn.deformation} \\
	&& \,\, \qquad \textcolor{red}{\frac{\partial J_k}{\partial t}+\frac{\partial \left( J_m v_m + T \right)}{\partial x_k} + 
	v_m \left(\frac{\partial J_{k}}{\partial x_m}-\frac{\partial J_{m}}{\partial x_k}\right) }  \textcolor{blue}{-\frac{\partial}{\partial x_m} \!\! \left( \epsilon \frac{\partial J_{k}}{\partial x_m} \right) } = 	 
	\textcolor{darkgreen}{-\dfrac{\beta_k}{\theta_2(\tau_2)}}, \label{eqn.heatflux} \\ 
	&& \,\, \qquad \frac{\partial \mathcal{E}}{\partial t}+\frac{\partial \left( v_k \left( \mathcal{E}_1 \!+\! \mathcal{E}_2 \textcolor{red}{+ \mathcal{E}_3 + \mathcal{E}_4}  \right) + v_i \, (p \, \delta_{ik} 
		 \textcolor{red}{+ \sigma_{ik}+\omega_{ik}})  \textcolor{red}{ +  h_k} \right)}{\partial x_k}  
	 \textcolor{blue}{-\frac{\partial}{\partial x_m} \!\! \left( \epsilon \frac{\partial 
	 \mathcal{E}}{\partial x_m} \right)} = 0. \label{eqn.energy} 
\end{eqnarray}
\end{subequations}
In the overdetermined system above $\mathbf{q}=\{q_i\} = (\rho, \rho v_i, \rho S, 
\textcolor{red}{A_{ik}}, \textcolor{red}{J_k})^T$ denotes the state vector, the total energy 
potential is $\mathcal{E} = \rho E =\mathcal{E}_1  +\mathcal{E}_2 \textcolor{red}{\, + \, 
\mathcal{E}_3 + \mathcal{E}_4} $ with $\mathcal{E}_i = \rho E_i$, $\epsilon >0$ is a vanishing 
viscosity and the nonnegative entropy production term due to the viscous terms is given by 
\begin{equation}
   \textcolor{blue}{\Pi = \frac{\epsilon}{T} \, \partial_{x_m} q_i \,\, \partial^2_{q_i q_j} 
   \mathcal{E} \,  \partial_{x_m} q_j \geq 0}, 
\end{equation}
since $\epsilon>0$ and we assume that the temperature $T > 0$ and that the Hessian of the total 
energy potential is at least positive semi-definite, $\mathcal{H}_{ij}:=\partial^2_{q_i q_j} \mathcal{E} \geq 0$. 
\textcolor{black}{Throughout this paper, we use the notations $\partial_p = \partial / \partial 
p$ 
and $\partial^2_{pq} = \partial^2 / (\partial p \partial q)$ for the first and second partial derivatives w.r.t. generic coordinates or quantities $p$ and $q$, which may also be vectors or components of a vector. Furthermore, 
we make use of the Einstein summation convention over repeated indices. Last but not least, in some occasions we also use bold face symbols in order to denote vectors and matrices, e.g. $\mathbf{q}=\{q_i\}$ and $\mathbf{A}=\{A_{ik}\}$, and so on.} 
In the above model the four contributions to the total energy density are 
\begin{equation}
	\mathcal{E}_1 = \frac{\rho^\gamma}{\gamma-1} e^{S/c_v}, \quad 
	\mathcal{E}_2 = \halb \rho v_i v_i, 	
	\quad 
	\mathcal{E}_3 = \frac{1}{4} \rho c_s^2 \mathring{G}_{ij} \mathring{G}_{ij}, 	
	\quad 
	\mathcal{E}_4 = \halb c_h^2 \rho J_i J_i,  	
\end{equation}
with the metric tensor $\mathbf{G}$ components and its trace-free part $\mathring{\mathbf{G}}$ 
given by 
$
	{G}_{ik} = A_{ji} A_{jk} $, and 
	$\mathring{G}_{ik} = {G}_{ik} - \frac{1}{3} \, G_{mm} \delta_{ik}
$.
The vector of thermodynamic dual variables reads $\mathbf{p} = \partial_{\q} \mathcal{E}=\{p_i\} = \left( r, 
v_i, T, 
\textcolor{red}{\alpha_{ik}}, \textcolor{red}{\beta_k} \right)^T$ with  
\begin{equation}
	r = \partial_{\rho} \mathcal{E}, 
	\qquad 
	v_i = \partial_{\rho v_i} \mathcal{E}, 
	\qquad 
	T = \partial_{\rho S} \mathcal{E},  	
	\qquad 
	\alpha_{ik} = \partial_{A_{ik}} \mathcal{E}, 
	\qquad 
	\beta_{k} = \partial_{J_{k}} \mathcal{E}. 
\end{equation}
The pressure is defined as
$
   p = \rho \, \partial_\rho \mathcal{E} + \rho v_i \, \partial_{\rho v_i} \mathcal{E} + 
   \rho S \, \partial_{\rho S} \mathcal{E} - \mathcal{E} = \rho^2 \partial_\rho E,
$
the stress tensors due to shear stress and thermal stress are, respectively,  
\begin{equation}
\sigma_{ik} = A_{ji} \partial_{A_{jk}} \mathcal{E} = 
A_{ji} \alpha_{jk} = \rho c_s^2 G_{ij} \mathring{G}_{jk}, 
\qquad 
\omega_{ik} = J_i \partial_{J_{k}} \mathcal{E} = J_i \beta_{k} = \rho c_h^2 J_i J_k,
\end{equation}
while the heat flux vector is given by 
\begin{equation}
	h_k = \partial_{\rho S} \mathcal{E} \, \partial_{J_k} \mathcal{E} = T \beta_k = \rho c_h^2 T J_k. 
\end{equation}
Note that for our convenience, we use the opposite sign in the definition of the stress tensor compared to the generally accepted notation.	 
Furthermore, $\theta_1(\tau_1)>0$ and $\theta_2(\tau_2)>0$ are two algebraic functions of the state vector $\mathbf{q}$ and the positive relaxation times $\tau_1>0$ and $\tau_2>0$:  
\begin{equation}
  \theta_1 = \frac{1}{3}  \rho z_1 \tau_1 \, c_s^2 \, \left| \mathbf{A} \right|^{-\frac{5}{3}},
  \qquad 
  \theta_2 = \rho z_2 \tau_2 \, c_h^2,
  \qquad
  z_1 = \frac{\rho_0}{\rho},
  \qquad 
  z_2 = \frac{\rho_0 T_0}{\rho \, T},
\end{equation}
with $ \rho_0 $ and $ T_0 $ being some reference density and temperature. 
It is easy to check that \eqref{eqn.energy} is a consequence of \eqref{eqn.conti}-\eqref{eqn.heatflux}, i.e.  
\begin{equation}\label{eqn.sum}
\eqref{eqn.energy} = 
 r \cdot \eqref{eqn.conti} + 
 v_i \cdot \eqref{eqn.momentum} +
 T \cdot \eqref{eqn.entropy} + 
 \alpha_{ik} \cdot \eqref{eqn.deformation} + 
 \beta_{k} \cdot \eqref{eqn.heatflux}. 
\end{equation} 
In \cite{GPRmodel} a formal asymptotic analysis of the model \eqref{eqn.conti}-\eqref{eqn.energy} was carried out, revealing that in the stiff limit the stress tensor $\sigma_{ik}$ and the heat flux $h_k$ tend to 
\begin{equation}
	\sigma_{ik} = -\frac{1}{6} \rho_0 c_s^2 \tau_1 \left( \partial_k v_i + \partial_i v_k
	- \frac{2}{3} \left( \partial_m v_m\right) \delta_{ik} \right), 
	\qquad 
	 h_k = -  \rho_0 T_0 c_h^2 \tau_2 \partial_k T,
	\label{eqn.asymptoticlimit}
\end{equation}
i.e. when the relaxation times $\tau_1, \tau_2\to 0$, the Navier-Stokes-Fourier equations are 
retrieved with effective shear viscosity $ \mu = \frac{1}{6} \rho_0 c_s^2 \tau_1 $ and heat 
conductivity \mbox{$ \kappa = \rho_0 T_0 c_h^2 \tau_2 $}.

%%%%%%%%%%%%%%%%%%%%%%%%%%%%%%%%%%%%%%%%%%%%%%%%%%%%%%%%%%%%%%%%%%
\section{Thermodynamically compatible semi-discrete finite volume scheme for the complete model in one space dimension}
\label{sec.scheme.1d}
In this section, we derive the thermodynamically compatible semi-discrete finite volume scheme for model \eqref{eqn.GPR}  in one space dimension. To this end, we start analysing the black terms on the system which enter into the original Godunov formalism \cite{God1961}. Once compatibility of these terms is established for the Euler subsystem, we can
include dissipative terms which require the consideration of the non-negative entropy production term, coloured in blue. The third step is the study of the red terms of  \eqref{eqn.momentum}-\eqref{eqn.energy} corresponding to the discretization of the distortion field and the thermal impulse. Finally, also the relaxation terms, in green, are addressed. 

Throughout the discretization, we will employ lower case subscripts, $i, \, j, \, k$, for tensor 
indices while lower case superscripts, $\ell$, refer to the spatial discretization index. 
Accordingly, we denote by $\Omega^{\ell}=[x^{\ell-\halb},x^{\ell+\halb}]$
a spatial control volume in one space dimension. 
\textcolor{black}{The Godunov form \cite{God1961} of the Euler subsystem (black terms in \eqref{eqn.GPR}) reads}  
\begin{gather} 
	\label{eqn.par0} 
	\left( \partial_{\p} L \right)_t + \partial_x \left( \partial_{\p} (v_1 L) \right) = 0, \\[2mm]
	\q = \partial_{\p} L, \qquad \p = \partial_{\q} \mathcal{E}, \qquad \f = \partial_{\p} (v_1 L), \qquad F = \p \cdot \f - v_1 L,
	\label{eqn.par02} 
\end{gather}
\textcolor{black}{with the generating potential $L = \p \cdot \q - \mathcal{E}$, which is the Legendre transform of the total energy potential $\mathcal{E}$.   
The  semi-discrete finite volume discretization of \eqref{eqn.par0} reads } 
\begin{equation}
	\label{eqn.god_sd1d}
	\frac{d}{d t} \q^{\ell} = -\frac{\f^{\ell+\halb} - \f^{\ell-\halb}}{\Delta x} 
	= -\frac{\left( \f^{\ell+\halb} - \f^{\ell}\right) -\left( \f^{\ell-\halb} - \f^{\ell}\right) }{\Delta x}
\end{equation}
with $\f^{\ell}=\f(\q^{\ell})$ and $\f(\q) = (\rho v_1, \rho v_i v_1 + p \delta_{i1}, 
\rho S v_1, \mathbf{0}, \mathbf{0} )^T$, containing only the fluxes of the Euler subsystem, i.e. 
the black terms in \eqref{eqn.GPR}, and $ F $ being the corresponding energy flux.
\textcolor{black}{Just like} on the continuous level, our first objective is to get a discrete form of the energy conservation 
as consequence of the discrete form of equations \eqref{eqn.conti}-\eqref{eqn.entropy}, see 
\eqref{eqn.sum}. Therefore we proceed alike we would do on the continuous level and we perform the dot product of the discrete dual variables, $\p^{\ell}= \partial_\q \mathcal{E}(\q^{\ell})$, with the discrete equations, obtaining 
\begin{equation} 
\label{eqn.fluct.pt_sd1d} 
\p^\ell \cdot \frac{d}{dt} \q^\ell = \frac{d}{dt}  \mathcal{E}^\ell = - \p^\ell \cdot \frac{ (\f^{\ell+\halb} - \f^{\ell}) + ( \f^\ell - \f^{\ell-\halb} ) }{\Delta x}.
\end{equation}
We now introduce the fluctuations
$D_{\mathcal{E}}^{\ell+\halb,-} = \p^{\ell} \cdot ( \f^{\ell+\halb} - \f^{\ell} )$, 
$D_{\mathcal{E}}^{\ell-\halb,+} = \p^{\ell} \cdot (\f^{\ell} - \f^{\ell-\halb} ) $.
In order to achieve a flux conservative expression for the discrete formulation of \eqref{eqn.energy}, we must be able to rewrite the fluctuations related to an interface as a flux difference 
\begin{equation}
\label{eqn.F.cond} 
D_{\mathcal{E}}^{\ell+\halb,-} + D_{\mathcal{E}}^{\ell+\halb,+} = F^{\ell+1} - F^{\ell}, 
\end{equation} 
with $F^{\ell}$ a consistent approximation of the 
total energy flux $F$. \textcolor{black}{This condition \eqref{eqn.F.cond} is mandatory in order to guarantee total energy conservation for vanishing energy flux at the boundary via the telescopic-sum property
\begin{equation}
   \sum \limits_{\ell}  \Delta x \frac{d}{dt}  \mathcal{E}^\ell = 
   - \sum \limits_{\ell} \left( D_{\mathcal{E}}^{\ell+\halb,-} + D_{\mathcal{E}}^{\ell-\halb,+} \right) =   
   - \sum \limits_{\ell} \left( F^{\ell+1} - F^{\ell} \right) = 0.   
\end{equation}
}
Substitution of the fluctuations in the former definition gives
\begin{eqnarray}
\p^{\ell} \cdot ( \f^{\ell+\halb}  - \f^{\ell} ) + \p^{\ell+1}  \cdot (\f^{\ell+1} - \f^{\ell+\halb} ) && \nonumber \\ 
=-\f^{\ell+\halb}  \cdot \left( \p^{\ell+1} - \p^{\ell} \right) + \p^{\ell+1}  \cdot \f^{\ell+1} - \p^{\ell}  \cdot \f^{\ell}  
& = & F^{\ell+1} - F^{\ell} 
\end{eqnarray} 
and, taking into account definition \eqref{eqn.par02} for $ \f $ and $ F $, we conclude
\begin{eqnarray} 
- \partial_{\p} (v_1 L)^{\ell+\halb} \cdot \left( \p^{\ell+1} - \p^{\ell} \right)  + \p^{\ell+1} \cdot \f^{\ell+1} - \p^{\ell} \cdot \f^{\ell} &=& \nonumber  \\ 
\p^{\ell+1} \cdot \f^{\ell+1} - (v_1 L)^{\ell+1} - \p^{\ell} \cdot \f^{\ell} + (v_1 L)^{\ell},  && 
\end{eqnarray} 
where $F^{\ell} = \p^{\ell} \cdot \f^{\ell} - (v_1 L)^{\ell}$. 
Accordingly, the numerical flux $\f^{\ell+\halb} = \partial_{\p} (v_1 L)^{\ell+\halb}$ must verify the 
Roe-type property, 
\begin{equation}
\label{eqn.fluxcondition_sd1d}
\f^{\ell+\halb} \cdot \left( \p^{\ell+1} - \p^{\ell} \right)  =  \partial_{\p} (v_1 L)^{\ell+\halb} \cdot \left( \p^{\ell+1} - \p^{\ell} \right) = (v_1 L)^{\ell+1} - (v_1 L)^{\ell}. 
\end{equation}
Next, we make use of the key idea on which path conservative schemes are based, see \cite{Castro2006,Pares2006}, and construct a path integral in phase-space by recalling the fundamental theorem of calculus
\begin{equation}
\label{eqn.pathint_sd1d}
(v_1 L)^{\ell+1} - (v_1 L)^{\ell} = \int \limits_{\p^{\ell}}^{\p^{\ell+1}} \partial_{\p} (v_1 L) \cdot d\p = \int \limits_{0}^{1} \partial_{\p} (v_1 L) \cdot \frac{\partial \boldsymbol{\psi}}{\partial s} ds.
\end{equation} 
Note that a similar methodology has already been successfully used in the construction of entropy-conservative fluxes \cite{Tadmor1}.
Since the path, $\boldsymbol{\psi}(s), \, s\in[0,1]$, can be freely chosen, we can select any parametrization convenient for our purposes. 
As a path connecting $\p^{\ell}$ and $\p^{\ell+1}$ we choose the simple straight line segment path 
in $\p$ variables: 
\begin{equation} 
\label{eqn.path1_sd1d} 
\boldsymbol{\psi}(s) = \p^{\ell} + s \left( \p^{\ell+1} - \p^{\ell} \right), \qquad \frac{\partial \boldsymbol{\psi}}{\partial s} = \p^{\ell+1} - \p^{\ell}, \qquad 0 \leq s \leq 1.
\end{equation}  
So \eqref{eqn.pathint_sd1d} together with \eqref{eqn.path1_sd1d} leads to 
\begin{equation}
\label{eqn.pathint.seg1_sd1d}
(v_1 L)^{\ell+1} - (v_1 L)^{\ell}  = \left( \int \limits_{0}^{1} \f(\boldsymbol{\psi}(s)) 
ds \right) \cdot \left( \p^{\ell+1} - \p^{\ell} \right). 
\end{equation} 
Therefore, the corresponding thermodynamically compatible numerical flux, 
\begin{equation}
\label{eqn.p.scheme_sd1d}
\f^{\ell+\halb}_\p = \int \limits_{0}^{1} \f(\boldsymbol{\psi}(s)) ds = 
\left( f^{\ell+\halb}_\rho, \f^{\ell+\halb}_{\rho \v}, f^{\ell+\halb}_{\rho S}, \mathbf{0}, \mathbf{0} \right)^T,  
\end{equation} 
guarantees \eqref{eqn.fluxcondition_sd1d} by construction. The subscript $\p$ refers to 
the segment path in $\p$ variables ($ \p $-scheme). For a different choice of path in 
terms of $\q$ variables ($ \q $-scheme) \textcolor{black}{the reader is referred to} \cite{SWETurbulence}.  
From the numerical point of view all path integrals appearing in this paper are approximated using a sufficiently accurate numerical quadrature rule, see e.g. \cite{OsherNC}. 
\textcolor{black}{If not stated otherwise, throughout this paper, we use a standard Gauss-Legendre quadrature rule with $n_{GP}=3$ points in order to compute the path integral appearing in \eqref{eqn.p.scheme_sd1d}. For a quantitative study of the influence of the quadrature rule on total energy conservation, see Section \ref{sec.results}.  } 
		
\subsection{Compatible scheme with dissipation terms}
So far we have presented a compatible discretization for the black terms in
\eqref{eqn.GPR}. To derive a dissipative scheme, we still need to include a compatible numerical dissipation. Let us enlarge \eqref{eqn.god_sd1d} with an additional dissipative flux and 
corresponding production terms:
\begin{equation}
\label{eqn.goddis_sd1d}
\frac{d}{dt} \q^{\ell} + \frac{\f^{\ell+\halb} - \f^{\ell-\halb}}{\Delta x} = \textcolor{blue}{\frac{\g^{\ell+\halb} - \g^{\ell-\halb}}{\Delta x} + \mathbf{P}^{\ell}}.
\end{equation}
We first focus on the numerical flux
\begin{equation}
\g^{\ell+\halb} = \epsilon^{\ell+\halb} \frac{\Delta \q^{\ell+\halb}}{\Delta x}, \qquad \Delta \q^{\ell+\halb}= \q^{\ell+1}-\q^{\ell}, 
\end{equation}
whose scalar numerical dissipation is \textcolor{black}{either chosen to be constant, $\epsilon^{\ell+\halb}=\epsilon$}, or taken of the form 
\begin{equation}
\epsilon^{\ell+\halb} = \halb \left( 1 - \phi^{\ell+\halb} \right) \Delta x \, s_{\max}^{\ell+\halb} \geq 0. 
\label{eqn.viscosity} 
\end{equation}
In the former expression, we have denoted by $s_{\max}^{\ell+\halb}$ the maximum signal speed at the cell interface and introduced $\phi^{\ell+\halb}$ which allows the use of a flux limiter, hence a reduction of the numerical dissipation in smooth regions. In particular, we consider the minbee flux limiter given by
\begin{equation}
\phi^{\ell+\halb} = \min \left( \phi^{\ell+\halb}_{-}, \phi^{\ell+\halb}_+  \right), \quad \textnormal{with} \quad 
\phi^{\ell+\halb}_{\pm} = \max\left( 0, \min \left(1, h^{\ell+\halb}_{\pm} \right) \right),
\label{eqn.fluxlimiter} 
\end{equation}
where
\begin{equation}
h^{\ell+\halb}_{-} = \frac{\mathcal{E}^{\ell} - \mathcal{E}^{\ell-1}}{ \mathcal{E}^{\ell+1} - \mathcal{E}^{\ell} }, \qquad \textnormal{and} \qquad 
h^{\ell+\halb}_{+} = \frac{\mathcal{E}^{\ell+2} - \mathcal{E}^{\ell+1}}{ \mathcal{E}^{\ell+1} - \mathcal{E}^{\ell} }
\end{equation}
are the ratios of the total energy potential slopes, see the SLIC scheme presented in \cite{toro-book} for further details on flux limiting strategies.
Note that an alternative approach to the use of flux limiters is the definition of a fixed numerical dissipation.

The dot product of $\p^{\ell}$ by \eqref{eqn.goddis_sd1d} yields
\begin{equation}
\label{eqn.pgoddis_sd1d}
\frac{d\mathcal{E}^{\ell}}{dt} + \frac{1}{\Delta x} \left(F^{\ell+\halb}-F^{\ell-\halb}\right)
= \frac{1}{\Delta x} \p^{\ell}\cdot \left(\g^{\ell+\halb} - \g^{\ell-\halb}\right) 
+\p^{\ell}\cdot \mathbf{P}^{\ell},
\end{equation}
where the left hand side has already been studied in the Godunov formalism. We therefore focus on the right hand side of the former equation obtaining
\begin{eqnarray} 
\label{eqn.E.diss.2} 
\p^{\ell} \cdot \mathbf{P}^{\ell} + \p^{\ell} \cdot \frac{\g^{\ell+\halb} - \g^{\ell-\halb}}{\Delta x} && \nonumber \\  
=\p^{\ell} \cdot \mathbf{P}^{\ell} + \frac{1}{\Delta x} \left( \halb \p^{\ell} \cdot \g^{\ell+\halb} + \halb \p^{\ell+1} \cdot \g^{\ell+\halb} + 
\halb \p^{\ell} \cdot \g^{\ell+\halb} - \halb \p^{\ell+1} \cdot \g^{\ell+\halb} \right)  && \nonumber \\ 
-\frac{1}{\Delta x} \left( \halb \p^{\ell} \cdot \g^{\ell-\halb} + \halb \p^{\ell-1} \cdot \g^{\ell-\halb} + \halb \p^{\ell} \cdot \g^{\ell-\halb} - \halb \p^{\ell-1} \cdot \g^{\ell-\halb} \right)     && \nonumber \\ 
=\p^{\ell} \cdot \mathbf{P}^{\ell}  +  \halb \frac{\p^{\ell+1} + \p^{\ell}}{\Delta x}  \cdot \epsilon^{\ell+\halb} \frac{\Delta \q^{\ell+\halb}}{\Delta x} - \halb \frac{ \p^{\ell} + \p^{\ell-1}}{\Delta x} \cdot \epsilon^{\ell-\halb} \frac{\Delta \q^{\ell-\halb}}{\Delta x}   && \nonumber \\   
- \halb \frac{\p^{\ell+1} - \p^{\ell}}{\Delta x} \cdot \epsilon^{\ell+\halb} \frac{\Delta \q^{\ell+\halb}}{\Delta x}  
- \halb \frac{\p^{\ell} - \p^{\ell-1}}{\Delta x} \cdot \epsilon^{\ell-\halb} \frac{\Delta \q^{\ell-\halb}}{\Delta x}. 
\end{eqnarray}
Besides, applying path integration yields
\begin{equation}
\label{eqn.pathintegraldis_sd1d}
\int \limits_{\q^{\ell}}^{\q^{\ell+1}} \p \, \cdot \, d \q  =  \int 
\limits_{\q^{\ell}}^{\q^{\ell+1}} \partial_\q \mathcal{E} \cdot d \q = \mathcal{E}^{\ell+1} - 
\mathcal{E}^{\ell} = \Delta \mathcal{E}^{\ell+\halb}.
\end{equation}
So 
$\halb ( \p^{\ell+1} + \p^{\ell} ) \cdot \Delta \q^{\ell+\halb}$
can be seen as an approximation of $\Delta \mathcal{E}^{\ell+\halb}$.
As a consequence of \eqref{eqn.E.diss.2} and \eqref{eqn.pathintegraldis_sd1d}, the energy flux including convective and diffusive terms is 
\begin{equation}
F^{\ell+\halb}_d = F^{\ell+\halb} - \halb ( \p^{\ell+1} + \p^{\ell} ) \cdot \epsilon^{\ell+\halb} \frac{\Delta \q^{\ell+\halb}}{\Delta x} 
\approx   F^{\ell+\halb} -  \epsilon^{\ell+\halb} \frac{\Delta \mathcal{E}^{\ell+\halb}}{\Delta x}.
\end{equation} 
\noindent To transform the jumps in $\p$ variables into jumps in $\q$ variables, we need to introduce a Roe-type matrix $\partial^2_{\q\q} \tilde{\mathcal{E}}^{\ell+\halb}$ 
verifying the Roe property
\begin{equation}
	\label{eqn.roeprop2}
	\partial^2_{\q\q} \tilde{\mathcal{E}}^{\ell+\halb} \cdot (\q^{\ell+1}-\q^{\ell} ) = \p^{\ell+1} - \p^{\ell}.
\end{equation}
\noindent For its calculation, we introduce another segment path $\tilde{\boldsymbol{\psi}}$, written in terms of $\q$  
\begin{equation} 
	\label{eqn.path2_sd1d} 
	\tilde{\boldsymbol{\psi}}(s) =  \q^{\ell} + s \left( \q^{\ell+1} - \q^{\ell} \right), \qquad 0 \leq s \leq 1,  
\end{equation}  
allowing to compute the sought Roe matrix as 
\begin{equation}
	\tilde{\boldsymbol{\mathcal{H}}}^{\ell+\halb} = \partial^2_{\q\q}\tilde{\mathcal{E}}^{\ell+\halb} = \int \limits_0^1 \partial^2_{\q\q} \mathcal{E}\left(\tilde{\boldsymbol{\psi}}(s)\right) ds =: 
	\left( \partial^2_{\p\p} \tilde{L}^{\ell+\halb} \right)^{-1},
	\label{eqn.EqqLqqRoe}
\end{equation} 
which satisfies \eqref{eqn.roeprop2} by construction. 
Substituting the obtained flux in \eqref{eqn.pgoddis_sd1d} 
and taking into account \eqref{eqn.roeprop2} 
gives
\begin{eqnarray} 
\label{eqn.E.diss.3} 
\frac{d}{dt} \mathcal{E}^{\ell} + \frac{F_d^{\ell+\halb} - F_d^{\ell-\halb}}{\Delta x} = \p^{\ell} \cdot \mathbf{P}^{\ell} &&  \\ 
- \halb \epsilon^{\ell+\halb} \frac{\q^{\ell+1} - \q^{\ell}}{\Delta x} \cdot \tilde{\boldsymbol{\mathcal{H}}}^{\ell+\halb} \frac{\q^{\ell+1}-\q^{\ell}}{\Delta x}  
- \halb \epsilon^{\ell-\halb} \frac{\q^{\ell} - \q^{\ell-1}}{\Delta x} \cdot \tilde{\boldsymbol{\mathcal{H}}}^{\ell-\halb} \frac{\q^{\ell}-\q^{\ell-1}}{\Delta x}. \nonumber
\end{eqnarray}
Thus, defining the production term $\mathbf{P}^{\ell} = (0,\mathbf{0},\Pi^{\ell},\mathbf{0},\mathbf{0})^T$ 
\begin{equation} 
\p^{\ell} \cdot \mathbf{P}^{\ell} = T^{\ell} \Pi^{\ell} = 
\halb \epsilon^{\ell+\halb} \frac{\Delta \q^{\ell+\halb}}{\Delta x} \cdot \tilde{\boldsymbol{\mathcal{H}}}^{\ell+\halb} \frac{\Delta \q^{\ell+\halb}}{\Delta x} +  
\halb \epsilon^{\ell-\halb} \frac{\Delta \q^{\ell-\halb}}{\Delta x} \cdot \tilde{\boldsymbol{\mathcal{H}}}^{\ell-\halb} \frac{\Delta \q^{\ell-\halb}}{\Delta x},
\label{eqn.production.condition} 
\end{equation}
we obtain the sought semi-discrete total energy conservation law 
\begin{equation} 
\label{eqn.E.diss} 
\frac{d}{dt} \mathcal{E}^{\ell} + \frac{F_d^{\ell+\halb} - F_d^{\ell-\halb}}{\Delta x} = 0.    
\end{equation}
Note that, as expected, the above definition provides a zero production term for all equations but for \eqref{eqn.entropy}.
The final compatible flux including convective and diffusive terms reads
\begin{eqnarray}
\f^{\ell+\halb}_{\p,d} &=& 
 \int \limits_{0}^{1} \f(\boldsymbol{\psi}(s)) ds  
- \frac{\epsilon^{\ell+\halb}}{\Delta x} \left( \q^{\ell+1} - \q^{\ell} \right).  
\label{eqn.p.scheme.diss} 
\end{eqnarray}

\subsection{Compatible discretization of the terms related to the distortion field}
%Computation of $\tilde{u}_{A}$ for the distortion equations to get a compatible advection for $A$
The momentum flux in \eqref{eqn.momentum} gathers four terms. The first two, in black, belong to 
the Euler subsystem and have already been studied in the previous sections. The third term, 
$\textcolor{red}{\sigma_{ik}}$, is related to the distortion field and thus its compatibility must 
be analysed together with the distortion transport equations, \eqref{eqn.deformation}, and the 
terms $\textcolor{red}{\mathcal{E}_3}$ and $\textcolor{red}{\sigma_{ik}}$ in the energy equation 
\eqref{eqn.energy}.
Let us consider the red terms in \eqref{eqn.deformation} (except for the convective term $ 
v_1\partial_{x} 
A_{ik} $), 
\begin{equation}
	\frac{\partial (A_{im} v_m)}{\partial x} - 
	v_m \frac{\partial A_{im}}{\partial x}  = A_{im}\frac{\partial v_m }{\partial x},
\end{equation}
and the following chosen discretization
\begin{equation}
	\Delta x A_{im} \partial_x \partial v_m \approx A_{im}^{\ell+\halb} \left(v_{m}^{\ell+1}-v_{m}^{\ell}\right) \quad \mathrm{with} \qquad  A_{im}^{\ell+\halb} = \halb\left( A_{im}^{\ell+1} + A_{im}^{\ell}\right).
\end{equation}
Multiplication of equations \eqref{eqn.momentum}, \eqref{eqn.deformation} by $\partial_{\rho v_{i }} \mathcal{E}=v_{i}$, $\partial_{A_{ik}} \mathcal{E} =\alpha_{ik}$  and assuming a compatible discretization with the term $\partial_x \left( v_{i}\sigma_{ik}\right)$ in \eqref{eqn.energy} leads to
\begin{eqnarray}
	v_{i}^{\ell}\left(\sigma_{ik}^{\ell+\halb} - \sigma_{ik}^{\ell}\right) 
	+ v_{i}^{\ell+1}\left(\sigma_{ik}^{\ell+1} - \sigma_{ik}^{\ell+\halb}\right)
	+ \alpha_{ik}^{\ell} \halb A_{im}^{\ell+\halb} \left( v_{m}^{\ell+1} - v_{m}^{\ell} \right) \notag \\
	+ \alpha_{ik}^{\ell+1} \halb A_{im}^{\ell+\halb} \left( v_{m}^{\ell+1} - v_{m}^{\ell} \right)
	= v_{i}^{\ell+1} \sigma_{ik}^{\ell+1} - v_{i}^{\ell} \sigma_{ik}^{\ell}.
\end{eqnarray}
We therefore obtain the following discretization for $\sigma_{ik}^{\ell+\halb} $:
\begin{equation}
	\sigma_{ik}^{\ell+\halb} = \halb \left(\alpha_{mk}^{\ell+1}+ \alpha_{mk}^{\ell}\right)A_{mi}^{\ell+\halb}.
\end{equation}

We now focus on the remaining flux term $v_{1}\partial_{x} A_{ik}$. We multiply the continuity equation \eqref{eqn.conti} by the dual variable $\partial_{\rho} \mathcal{E}_{3}=E_{3}$ and \eqref{eqn.deformation} by $\partial_{A_{ik}} \mathcal{E}=\alpha_{ik}$ and impose the compatibility condition with the energy conservation equation yielding  
\begin{eqnarray}
	E_{3}^{\ell}  \left(\rho v_1^{\ell+\halb}-\rho v_1^{\ell}\right)
	+E_{3}^{\ell+1}  \left(\rho v_1^{\ell+1}-\rho v_1^{\ell+\halb}\right)
	+ \alpha^{\ell}_{ik}\halb \tilde{v}^{\ell+\halb}_{A \,1} 
	\left(A^{\ell+1}_{ik}-A^{\ell}_{ik}\right) \notag\\
	+ \alpha^{\ell+1}_{ik}\halb \tilde{v}^{\ell+\halb}_{A \,1} 
	\left(A^{\ell+1}_{ik}-A^{\ell}_{ik}\right)
	= \rho v^{\ell+1}_{1}E_{3}^{\ell+1} - \rho v^{\ell}_{1}E_{3}^{\ell},
\end{eqnarray}
where the approximation of the averaged velocity $\tilde{v}^{\ell+\halb}_{A \,1}$ still needs to be defined.
Collecting terms, we get
\begin{equation}
	\label{eqn.derivutildeA_sd1D}
	- \rho v^{\ell+\halb}_{1}\left( E_{3}^{\ell+1} - E_{3}^{\ell} \right)
	+\halb \tilde{v}^{\ell+\halb}_{A \,1} \left(\alpha^{\ell+1}_{ik} +\alpha^{\ell}_{ik} \right) \left(A^{\ell+1}_{ik}-A^{\ell}_{ik}\right) =0.
\end{equation}
Hence, the average velocity must be discretised as
\begin{equation}
	\label{eqn.averagevelocityA_sd1D}
	\tilde{v}^{\ell+\halb}_{A \,1} = \frac{\rho v^{\ell+\halb}_{1}\left( E_{3}^{\ell+1} - E_{3}^{\ell} \right)}{\halb  \left(\alpha^{\ell+1}_{ik} +\alpha^{\ell}_{ik} \right) \left(A^{\ell+1}_{ik}-A^{\ell}_{ik}\right)} 
\end{equation}
if the denominator in \eqref{eqn.averagevelocityA_sd1D} is non-zero, otherwise we set 
$
	\tilde{v}^{\ell+\halb}_{A \,1} = \halb \left(v^{\ell+1}_{1}+v^{\ell}_{1} \right).%\mathbf{n} .
$ 
Finally, if $E_{3}^{\ell+1} - E_{3}^{\ell} = 0$, from \eqref{eqn.derivutildeA_sd1D}, we get $\tilde{v}^{\ell+\halb}_{A \,1} = 0$.

\subsection{Compatible discretization of the terms related to the thermal impulse}
Similarly to what has been done for the distortion field, in this section we derive the discretization of the red terms in \eqref{eqn.momentum}, \eqref{eqn.heatflux}, \eqref{eqn.energy} related to the heat flux.
First, we focus on terms $\partial_x\omega_{i1}$ in 
\eqref{eqn.momentum}, 
$J_{m} \partial_x v_{m}= 
\partial_x (J_{m}v_{m}) - v_{m} \partial_x J_{m}$ in 
\eqref{eqn.heatflux} and 
$\partial_x \left( v_{i}\omega_{i1}\right)$ in \eqref{eqn.energy}. Multiplying the momentum equation by
$\partial_{\rho v_{i}} \mathcal{E} = v_{i}$, the thermal impulse equation by $\partial_{J_{1}} \mathcal{E}=\beta_{1}$ and requiring compatibility with the energy equation we get
\begin{eqnarray}
v_{i}^{\ell}\left(\omega_{i1}^{\ell+\halb}-\omega_{i1}^{\ell}\right)
+ v_{i}^{\ell+1}\left(\omega_{i1}^{\ell+1}-\omega_{i1}^{\ell+\halb}\right)
+\beta_{1}^{\ell}\halb J_{m}^{\ell+\halb} \left( v_{m}^{\ell+1} -  v_{m}^{\ell} \right) \notag \\
+\beta_{1}^{\ell+1}\halb J_{m}^{\ell+\halb} \left( v_{m}^{\ell+1} -  v_{m}^{\ell} \right)
= v_{i}^{\ell+1} \omega_{i1}^{\ell+1} -v_{i}^{\ell} \omega_{i1}^{\ell}.
\label{eqn.fluxcomp_sd1D}
\end{eqnarray}
Relabeling the repeated index $m$ 
and defining 
$
{J}^{\ell+\halb}_{i} = \halb\left({J}^{\ell+1}_{i}+{J}^{\ell}_{i}\right),
$
yields 
\begin{equation}
- \omega_{i1}^{\ell+\halb}\left( v_{i}^{\ell+1} - v_{i}^{\ell} \right)
+ \halb \left(\beta_{1}^{\ell+1}+\beta_{1}^{\ell}\right){J}^{\ell+\halb}_{i} \left( v_{i}^{\ell+1} 
- v_{i}^{\ell} \right) = 0.
\end{equation}
Thus choosing
$
	\omega_{i1}^{\ell+\halb} = \halb \left(\beta_{1}^{\ell+1}+\beta_{1}^{\ell}\right) {J}^{\ell+\halb}_{i}
$
gives the sought compatibility.

Next, we need to compute the discretization related to the term $v_{1} \partial_x J_{k}$ in \eqref{eqn.heatflux}.
Multiplication of \eqref{eqn.conti} by $\partial_{\rho} \mathcal{E}_{4}=E_{4}$ and addition of \eqref{eqn.heatflux} multiplied by $\partial_{J_{k}} \mathcal{E}=\beta_{k}$ 
yields
\begin{eqnarray}
	E_{4}^{\ell} \left(\rho v_1^{\ell+\halb}-\rho v_1^{\ell}\right)
	+E_{4}^{\ell+1} \left(\rho v_1^{\ell+1}-\rho v_1^{\ell+\halb}\right)
	+\halb \tilde{v}^{\ell+\halb}_{J \, 1} \beta^{\ell}_{k} \left(J^{\ell+1}_{k}-J^{\ell}_{k}\right) \notag\\
	+\halb \tilde{v}^{\ell+\halb}_{J \, 1} \beta^{\ell+1}_{k} \left(J^{\ell+1}_{k}-J^{\ell}_{k}\right)
	= \rho v^{\ell+1}_{1}E_{4}^{\ell+1} - \rho v^{\ell}_{1}E_{4}^{\ell}.
\end{eqnarray}
Hence,
\begin{equation}
	\tilde{v}^{\ell+\halb}_{J \, 1} = \frac{\rho v^{\ell+\halb}_{1}\left( E_{4}^{\ell+1} - E_{4}^{\ell} \right)}{\halb  \left(\beta^{\ell+1}_{k} +\beta^{\ell}_{k} \right) \left(J^{\ell+1}_{k}-J^{\ell}_{k}\right)}
\end{equation}
is the compatible discretization for the advection speed related to the thermal impulse.
Analogous to the previous section, for null denominator and $E_{4}^{\ell+1} - E_{4}^{\ell}\neq 0$ we define $\tilde{v}^{\ell+\halb}_{J \, 1}$ as the arithmetic average of the velocity in the two related cells.

It now just remains to establish the discrete compatibility between the term ${\beta_{k}}$ in equation 
\eqref{eqn.entropy}, ${T}$ in equation \eqref{eqn.heatflux} and 
${h_{k}}$ in equation \eqref{eqn.energy}.
Let us assume we have the following given discretization for the gradient of $T$:
\begin{equation}
	\Delta x \partial_x T \approx \halb\left(T^{\ell+1}-T^{\ell} \right).
\end{equation}
Then, multiplication of the thermal impulse equation \eqref{eqn.heatflux} by $\partial_{J_{k}} \mathcal{E} = 
\beta_{k}$ and the entropy relation by $\partial_{\rho S} \mathcal{E}=T$ gives
\begin{eqnarray}
	T^{\ell}\left(\beta_{k}^{\ell+\halb}-\beta_{k}^{\ell}\right) 
	+ T^{\ell+1}\left(\beta_{k}^{\ell+1}-\beta_{k}^{\ell+\halb}\right)
	+ \beta_{k}^{\ell} \halb \left(T^{\ell+1}-T^{\ell}\right) \notag\\
	+ \beta_{k}^{\ell+1} \halb \left(T^{\ell+1}-T^{\ell}\right) 
	= \beta_{k}^{\ell+1}T^{\ell+1}-\beta_{k}^{\ell}T^{\ell}.
\end{eqnarray}
Hence, by simply defining 
$
	\beta_k^{\ell+\halb} = \halb\left(\beta_{k}^{\ell+1}+\beta_{k}^{\ell}\right)
$
we get a compatible discretization of the equations.

\subsection{Compatible discretization of relaxation terms}
Finally, it is easy to see that the relaxation terms, in green in \eqref{eqn.entropy}-\eqref{eqn.heatflux}, cancel. Multiplication of $\partial_{\rho S} \mathcal{E}=T$, $\partial_{A_{ik}} \mathcal{E}=\alpha_{ik}$ and $\partial_{J_{k}} \mathcal{E}=\beta_{k}$ by the green terms in \eqref{eqn.entropy}-\eqref{eqn.heatflux}, respectively, and adding the result gives
\begin{equation}
	T  \frac{\alpha_{ik} \alpha_{ik} }{\theta_1(\tau_1) T}  + T
	\frac{\beta_i \beta_i}{\theta_2(\tau_2) T} - \alpha_{ik} \frac{ \alpha_{ik} }{\theta_1(\tau_1)} -\beta_k \frac{\beta_k}{\theta_2(\tau_2)} = 0.
\end{equation}
Thus, the compatibility is proven by construction.

%%%%%%%%%%%%%%%%%%%%%%%%%%%%%%%%%%%%%%%%%%%%%%%%%%%%%%%%%%%%%%%%%%
\section{Thermodynamically compatible semi-discrete finite volume scheme for the complete model in two space dimensions}
\label{sec.scheme.2d}
The derivation of the thermodynamically compatible semi-discrete finite volume scheme for the 
complete model in two space dimensions can be done following the steps described in the previous 
section. Here we summarize the final scheme and provide the mathematical proofs of the marginal nonlinear stability in the energy norm and of the semi-discrete cell entropy inequality. Let us 
consider the spatial control volume $\Omega^{\ell} $ with circumcenter $\mathbf{x}^{\ell}$, 
one of its neighbors  $ \Omega^{\err} $ and the common edge $ \partial \Omega^{\ell\err} $, 
$\normal = (n_1,n_2)^T$ being the outward unit normal vector to the face 
$\partial\Omega^{\ell\err}$ \textcolor{black}{and $N_{\ell}$ being the set of neighbors of cell $\Omega^\ell$.}
%where $s$ corresponds to the superscript for spatial discretization in $y$-direction 
%and $\ell$ is the cell index. 
The final 
semi-discrete finite volume scheme reads
\begin{subequations}\label{eqn.schemeGPR2D}
	\begin{align}
		 \frac{\partial \rho^{\ell}}{\partial t} = & -\frac{1}{\left|\Omega^{\ell}\right|} \sum_{\err\in N_{\ell}} \left|\partial\Omega^{\ell\err}\right| D_{\rho}^{\ell \err,-}
		  \textcolor{blue}{+ \frac{1}{\left|\Omega^{\ell}\right|} \sum_{\err\in N_{\ell}} \left|\partial\Omega^{\ell\err}\right| g_{\rho,\,\normal}^{\ell\err}}, \label{eqn.schemeGPR2D_mass}\\
		  \frac{\partial (\rho v_{i}^{\ell})}{\partial t} = & -\frac{1}{\left|\Omega^{\ell}\right|} 
		  \sum_{\err\in N_{\ell}} 
		  \left|\partial\Omega^{\ell\err}\right| D_{\rho v_{i}}^{\ell \err,-}
		  \textcolor{red}{- \frac{1}{\left|\Omega^{\ell}\right|} \sum_{\err\in N_{\ell}} 
		  \left|\partial\Omega^{\ell\err}\right|\sigma_{ik}^{\ell\err,\,-}n_k} 
		   \nonumber\\
		 & \textcolor{red}{- \frac{1}{\left|\Omega^{\ell}\right|} \sum_{\err\in N_{\ell}} 
		 \left|\partial\Omega^{\ell\err}\right|\omega_{ik}^{\ell\err,\,-}n_k}
		  \textcolor{blue}{+ \frac{1}{\left|\Omega^{\ell}\right|} \sum_{\err\in N_{\ell}}     \left|\partial\Omega^{\ell\err}\right| g_{\rho v_{i},\,\normal}^{\ell\err}},\label{eqn.schemeGPR2D_momentum} \\
		  \frac{\partial (\rho S^{\ell})}{\partial t} = & -\frac{1}{\left|\Omega^{\ell}\right|} 
		  \sum_{\err\in N_{\ell}}
		  \left|\partial\Omega^{\ell\err}\right| D_{\rho S}^{\ell \err,-}
		  \textcolor{red}{-\frac{1}{\left|\Omega^{\ell}\right|} \sum_{\err\in N_{\ell}}  
		  \left|\partial\Omega^{\ell\err}\right| \left( \beta_{k}^{\ell\err}-\beta_{k}^{\ell} 
		  \right)n_k } \nonumber\\
		 & \textcolor{blue}{+ \frac{1}{\left|\Omega^{\ell}\right|} \sum_{\err\in N_{\ell}}     \left|\partial\Omega^{\ell\err}\right| g_{\rho S,\,\normal}^{\ell\err}}
		 \textcolor{blue}{+\frac{1}{\left|\Omega^{\ell}\right|} \sum_{\err\in N_{\ell}} 
		 \left|\partial\Omega^{\ell\err}\right| \Pi^{\ell\err,-}_{\normal}}
		 \textcolor{darkgreen}{+\frac{\alpha_{ik}^{\ell}\alpha_{ik}^{\ell}}{\theta_1^{\ell}\left(\tau_{1}\right) T^{\ell} } }  
		 \textcolor{darkgreen}{+\frac{\beta_{i}^{\ell}\beta_{i}^{\ell}}{\theta_2^{\ell}\left(\tau_{2}\right) T^{\ell} } }, \label{eqn.schemeGPR2D_entropy}\\
		 \textcolor{red}{\frac{\partial A_{ik}^{\ell}}{\partial t}} = &
		 \textcolor{red}{- \frac{1}{\left|\Omega^{\ell}\right|} \sum_{\err\in N_{\ell}} 
		 \left|\partial\Omega^{\ell\err}\right| 
		 \frac{1}{2}A^{\ell\err}_{im}\left(v_{m}^{\err}-v_{m}^{\ell} \right) 
		 n_k}\nonumber\\
		 &
		 \textcolor{red}{- \frac{1}{\left|\Omega^{\ell}\right|} \sum_{\err\in N_{\ell}} \left|\partial\Omega^{\ell\err}\right| \frac{1}{2} \tilde{u}_{A,\,\normal}^{\ell\err}   \left(A_{ik}^{\err}-A_{ik}^{\ell}\right)}
		  \textcolor{blue}{+ \frac{1}{\left|\Omega^{\ell}\right|} \sum_{\err\in N_{\ell}} \left|\partial\Omega^{\ell\err}\right| g_{A_{ik},\,\normal}^{\ell\err}}
		  \textcolor{darkgreen}{-\frac{\alpha_{ik}^{\ell}}{\theta_1^{\ell}\left(\tau_{1}\right) } },  
		 \label{eqn.schemeGPR2D_dist} \\
		 \textcolor{red}{\frac{\partial J_{k}^{\ell}}{\partial t}} = &
		 \textcolor{red}{- \frac{1}{\left|\Omega^{\ell}\right|} \sum_{\err\in N_{\ell}} 
		 \left|\partial\Omega^{\ell\err}\right| 
		 \frac{1}{2}J^{\ell\err}_{i}\left(v_{m}^{\err}-v_{m}^{\ell} \right) n_k}
		 \textcolor{red}{- \frac{1}{\left|\Omega^{\ell}\right|} \sum_{\err\in N_{\ell}} 
		 \left|\partial\Omega^{\ell\err}\right| \frac{1}{2} T^{\ell\err,-} n_k} \nonumber\\
		 &
		 \textcolor{red}{- \frac{1}{\left|\Omega^{\ell}\right|} \sum_{\err\in N_{\ell}} \left|\partial\Omega^{\ell\err}\right| \frac{1}{2} \tilde{u}_{J,\,\normal}^{\ell\err}   \left(J_{k}^{\err}-J_{k}^{\ell}\right)}
		  \textcolor{blue}{+ \frac{1}{\left|\Omega^{\ell}\right|} \sum_{\err\in N_{\ell}} \left|\partial\Omega^{\ell\err}\right| g_{J_{k},\,\normal}^{\ell\err}}
		 \textcolor{darkgreen}{-\frac{\beta_{i}^{\ell}}{\theta_2^{\ell}\left(\tau_{2}\right)}}
		 \label{eqn.schemeGPR2D_heatf}
	\end{align}
\end{subequations}
with
\begin{gather}
	D_{\q}^{\ell \err,-} = \left(f_{\q,\,k}^{\ell\err}-f_{\q,\,k}^{\ell}\right)n_k,  
	\label{eqn.fluctuations_2D}\\
	 g_{\q,\,\normal}^{\ell\err} = 
	 \epsilon^{\ell\err}\frac{\q^{\err}-\q^{\ell}}{\delta^{\ell\err}}= 
	 \epsilon^{\ell\err}\frac{\Delta
	  \q^{\ell\err}}{\delta^{\ell\err}}, \quad \delta^{\ell\err} = \left\| \mathbf{x}^\err - \mathbf{x}^\ell \right\| = \Delta x n_{1}+\Delta y n_{2}, 
	 \label{eqn.diffusion_2D}\\
	\sigma_{jk}^{\ell\err,-} = \sigma_{jk}^{\ell\err} -\sigma_{jk}^{\ell} ,
	\quad \sigma_{jk}^{\ell\err} = \halb A_{ij}^{\ell\err}\left(\alpha_{ik}^{\ell}+\alpha_{ik}^{\err}\right), 
	\label{eqn.sigma_2D} \\
	\omega_{jk}^{\ell\err,-} = \omega_{jk}^{\ell\err}-\omega_{jk}^{\ell},\quad \omega_{jk}^{\ell\err} = \halb	 \left(\beta_{k}^{\ell}+\beta_{k}^{\err}\right)J_{i}^{\ell\err},
	\label{eqn.omegabeta_2D} \\
	\Pi^{\ell\err,-}_{\normal} = \frac{1}{2}\epsilon^{\ell\err} \frac{\Delta \q^{\ell\err}}{T^{\ell}} \cdot  \partial^2_{\q \q} \mathcal{E}^{\ell\err} \frac{\Delta \q^{\ell\err}}{\delta^{\ell\err}}, \qquad 
	T^{\ell} = \frac{\left( \rho^{\ell}\right)^{\gamma-1}}{\left( \gamma-1\right) c_{v}} e^{\frac{ S^{\ell}}{c_{v}}} ,\label{eqn.production_2D}\\
	A^{\ell\err}_{im} = \frac{1}{2}\left( A^{\ell}_{im}+A^{\err}_{im}\right), \qquad 
	\tilde{u}_{A,\,\normal}^{\ell\err}  = 	\tilde{v}_{A,\,j}^{\ell\err} n_{j} =
	\frac{f_{\rho,\,j}^{\ell\err}n_{j}\left(E_{3}^{\err}-E_{3}^{\ell}\right) }{\frac{1}{2} \left( 
	\alpha_{ik}^{\ell}+\alpha_{ik}^{\err}\right) \left(A_{ik}^{\err}-A_{ik}^{\ell} \right) }, 
	\label{eqn.utildeA_2D}\\
	J^{\ell\err}_{i}\! =\! \frac{1}{2}\left( J^{\ell}_{i}+J^{\err}_{i}\right), \,\,\, 
	\tilde{u}_{J,\,\normal}^{\ell\err} \! = \!\tilde{v}_{J,\,j}^{\ell\err} n_{j} = 
	\frac{f_{\rho,\,j}^{\ell\err}n_{j}\left(E_{4}^{\err}-E_{4}^{\ell}\right) }{\frac{1}{2} \left( 
	\beta_{k}^{\ell}+\beta_{k}^{\err}\right) \left(J_{k}^{\err}-J_{k}^{\ell} \right) },\quad 
	T^{\ell\err,-}\! =\! T^{\err}-T^{\ell}. \label{eqn.utildeJ_2D}
\end{gather}
\begin{theorem}\label{theorem.energy.semidiscrete}
	The thermodynamically compatible semi-discrete finite volume \\
scheme \eqref{eqn.schemeGPR2D}
	admits the semi-discrete energy conservation law
	\begin{equation}\label{eqn.schemeGPR2D_energy}
		\frac{\partial \mathcal{E}^{\ell}}{\partial t} = 
		-\frac{1}{\left|\Omega^{\ell}\right|} \sum_{\err\in N_{\ell}} \left|\partial\Omega^{\ell\err}\right| D_{\mathcal{E}}^{\ell \err,-}
		 \textcolor{blue}{+ \frac{1}{\left|\Omega^{\ell}\right|} \sum_{\err\in N_{\ell}} \left|\partial\Omega^{\ell\err}\right| g_{\mathcal{E},\,\normal}^{\ell\err}}
	\end{equation}	
	with
	\begin{equation}
		D_{\mathcal{E}}^{\ell \err,-} +  D_{\mathcal{E}}^{ \ell \err,+} = D_{\mathcal{E}}^{\ell \err,-} +  D_{\mathcal{E}}^{ \err \ell,-} =  F^{\err} -F^{\ell}.
	\end{equation}
	Assuming that the jumps on the boundary vanish, the scheme is nonlinearly marginally stable in the energy norm, i.e. the scheme satisfies the identity 
	\begin{equation}
		\int_{\Omega} \frac{\partial \mathcal{E}^{\ell}}{\partial t} dV
		= \sum_{\ell} | \Omega^{\ell} | \frac{\partial \mathcal{E}^{\ell}}{\partial t} = 0.
	\end{equation}
\end{theorem}
\begin{proof}
	We start considering the contributions of the dot product of vector
	$$ \p^{\ell}=\partial_{\q} \mathcal{E}^{\ell}= \left( \partial_{\rho} \mathcal{E}^{\ell}, \partial_{\rho v_{i}} \mathcal{E}^{\ell}, \partial_{\rho S} \mathcal{E}^{\ell}, \partial_{A_{ik}} \mathcal{E}^{\ell}, \partial_{J_{k}} \mathcal{E}^{\ell} \right)^{T}$$
	with the time derivative terms in \eqref{eqn.schemeGPR2D}:
	\begin{gather}\label{eqn.Etime_2d}
		  \partial_{\rho} \mathcal{E}^{\ell} \frac{\partial \rho^{\ell}}{\partial t} 
		+ \partial_{\rho v_{i}} \mathcal{E}^{\ell} \frac{\partial \rho v_{i}^{\ell}}{\partial t}
		+ \partial_{\rho S} \mathcal{E}^{\ell} \frac{\partial \rho S^{\ell}}{\partial t} 
		+ \partial_{A_{ik}} \mathcal{E}^{\ell} \frac{\partial A_{ik}^{\ell}}{\partial t}
		+ \partial_{J_{k}} \mathcal{E}^{\ell} \frac{\partial J_{k}^{\ell}}{\partial t} 
		= \partial_{\q} \mathcal{E}^{\ell} \frac{\partial \q^{\ell}}{\partial t}
		= \frac{\partial \mathcal{E}^{\ell}}{\partial t}.
	\end{gather}

	We now define the fluctuations associated to the total energy equation as
	\begin{eqnarray}\label{eqn.Efluc_2d}
		D_{\mathcal{E}}^{\ell \err,-} & =  &\partial_{\rho}\mathcal{E}^{\ell}_{1} D_{\rho}^{\ell 
		\err,-}
		+ \partial_{\rho}\mathcal{E}^{\ell}_{2} D_{\rho}^{\ell \err,-}
		\textcolor{red}{+ \partial_{\rho}\mathcal{E}^{\ell}_{3} D_{\rho}^{\ell \err,-}}
		\textcolor{red}{+ \partial_{\rho}\mathcal{E}^{\ell}_{4} D_{\rho}^{\ell \err,-}}
		+ \partial_{\rho v_{i}}\mathcal{E}^{\ell} D_{\rho v_{i}}^{\ell \err,-} 
		\nonumber\\
		&&\textcolor{red}{+\partial_{\rho v_{i}}\mathcal{E}^{\ell}\left(\sigma_{ik}^{\ell\err,-}n_k 
		+\omega_{ik}^{\ell\err,-}n_k \right) } 
		+ \partial_{\rho S}\mathcal{E}^{\ell} D_{\rho S}^{\ell \err,-}
		\textcolor{red}{+ \partial_{\rho S}\mathcal{E}^{\ell}  \left(\beta_{k}^{\ell\err} - 
		\beta_{k}^{\ell}\right) n_k  }
		\nonumber\\
		&&\textcolor{red}{+ 
		\partial_{A_{ik}}\mathcal{E}^{\ell}\frac{1}{2\delta^{\ell\err}}A^{\ell\err}_{im}\left(v_{m}^{\err}-v_{m}^{\ell}
		 \right) 
		n_k }
		\textcolor{red}{+ \partial_{A_{ik}}\mathcal{E}^{\ell}  \frac{1}{2} 
		\tilde{u}_{A,\,\normal}^{\ell\err}   \left(A_{ik}^{\err}-A_{ik}^{\ell}\right)  }
		\nonumber\\
		&&\textcolor{red}{+ \partial_{J_{k}}\mathcal{E}^{\ell} 
		\frac{1}{2\delta^{\ell\err}}J^{\ell\err}_{i}\left(v_{m}^{\err}-v_{m}^{\ell} \right) n_k} 
		\textcolor{red}{+ \partial_{J_{k}}\mathcal{E}^{\ell} \frac{1}{2} 
		\tilde{u}_{J,\,\normal}^{\ell\err}   \left(J_{k}^{\err}-J_{k}^{\ell}\right) 
		+ \partial_{J_{k}}\mathcal{E}^{\ell}  \frac{1}{2} T^{\ell\err,-} n_k  }.
	\end{eqnarray}

	On the other hand, from \eqref{eqn.diffusion_2D} and applying relations analogous to the ones introduced in \eqref{eqn.E.diss.2} and \eqref{eqn.roeprop2}, we have
	\begin{eqnarray} 
		\frac{1}{\left|\Omega^{\ell}\right|}\! \sum_{\err\in N_{\ell}} \! \left|\partial\Omega^{\ell\err}\right|\left( \p^{\ell} \cdot \mathbf{P}^{\ell\err,-}_{\normal} + \p^{\ell} \cdot \g^{\ell\err}_{\normal}\right)  =
		\sum_{\err\in N_{\ell}} \!\frac{\left|\partial\Omega^{\ell\err}\right|}{\left|\Omega^{\ell}\right|}\left( \p^{\ell} \cdot \mathbf{P}^{\ell\err,-}_{\normal}+ \p^{\ell} \cdot
		\epsilon^{\ell\err}\frac{\Delta \q^{\ell\err}}{\delta^{\ell\err}}\right)  
		\nonumber\\
		\!=\!\sum_{\err\in N_{\ell}} \! \frac{\left|\partial\Omega^{\ell\err}\right|}{\left|\Omega^{\ell}\right|} \! \left( \p^{\ell}\! \cdot \mathbf{P}^{\ell\err,-}_{\normal}
		\!+\! \halb \p^{\ell}\! \cdot \epsilon^{\ell\err}\frac{\Delta \q^{\ell\err}}{\delta^{\ell\err}} 
		\!+\! \halb \p^{\err}\! \cdot \epsilon^{\ell\err}\frac{\Delta \q^{\ell\err}}{\delta^{\ell\err}} 
		\!+\! \halb\p^{\ell}\! \cdot \epsilon^{\ell\err}\frac{\Delta \q^{\ell\err}}{\delta^{\ell\err}} 
		\!-\! \halb \p^{\err}\! \cdot \epsilon^{\ell\err}\frac{\Delta \q^{\ell\err}}{\delta^{\ell\err}}\right)
		\nonumber\\
		=\!\sum_{\err\in N_{\ell}} \!\frac{\left|\partial\Omega^{\ell\err}\right|}{\left|\Omega^{\ell}\right|}\left( \p^{\ell} \cdot \mathbf{P}^{\ell\err,-}_{\normal}
		+ \halb \left( \p^{\ell}+\p^{\err}\right)  \cdot \epsilon^{\ell\err}\frac{\Delta \q^{\ell\err}}{\delta^{\ell\err}} 
		- \halb\left(\p^{\err} -\p^{\ell}\right)  \cdot \epsilon^{\ell\err}\frac{\Delta \q^{\ell\err}}{\delta^{\ell\err}} \right) \nonumber \\
		=\!\sum_{\err\in N_{\ell}} \!\frac{\left|\partial\Omega^{\ell\err}\right|}{\left|\Omega^{\ell}\right|}\left( \p^{\ell} \cdot \mathbf{P}^{\ell\err,-}_{\normal}
		+ \epsilon^{\ell\err}\frac{\Delta \mathcal{E}^{\ell\err}}{\delta^{\ell\err}} 
		- \halb\epsilon^{\ell\err}\frac{\Delta \q^{\ell\err}}{\delta^{\ell\err}} 
		\partial^2_{\q\q}\mathcal{E}^{\ell\err}\Delta \q^{\ell\err} \right)\!. 
		\label{eqn.diss.2d} 		 
	\end{eqnarray}
	Substitution of $\mathbf{P}^{\ell\err,-}_{\normal} = \left(0,\mathbf{0},\Pi^{\ell\err,-}_{\normal},\mathbf{0},\mathbf{0}\right)$  combined with \eqref{eqn.production_2D} yields
	\begin{equation}\label{eqn.diff_2D}
		\frac{1}{\left|\Omega^{\ell}\right|}\! \sum_{\err\in N_{\ell}}   \!  \left|\partial\Omega^{\ell\err}\right| \left( \p^{\ell} \cdot  \g_{\normal}^{\ell\err} +  \p^{\ell} \cdot \mathbf{P}^{\ell\err,-}_{\normal} \right) 		 
		= \frac{1}{\left|\Omega^{\ell}\right|} \!\sum_{\err\in N_{\ell}}\! \left|\partial\Omega^{\ell\err}\right|  \epsilon^{\ell\err}\frac{\Delta \mathcal{E}^{\ell\err}}{\delta^{\ell\err}} 
		= \frac{1}{\left|\Omega^{\ell}\right|}\! \sum_{\err\in N_{\ell}}\! \left|\partial\Omega^{\ell\err}\right| g_{\mathcal{E},\,\normal}^{\ell\err}.
	\end{equation}
	Finally, taking into account the dot product of $\p^{\ell}$ by the diffusion terms in \eqref{eqn.schemeGPR2D} and applying \eqref{eqn.diff_2D},
	we get
	\begin{eqnarray} \label{eqn.Ediff_2D}
		\p^{\ell} \cdot \frac{1}{\left|\Omega^{\ell}\right|}\! \sum_{\err\in N_{\ell}} \!    \left|\partial\Omega^{\ell\err}\right| \g_{\normal}^{\ell\err} 
		+\partial_{\rho S} \mathcal{E}^{\ell} \frac{1}{\left|\Omega^{\ell}\right|}\! \sum_{\err\in N_{\ell}}\! 
			\left|\partial\Omega^{\ell\err}\right| \Pi^{\ell\err}_{\normal}
		\\ %
		\qquad\qquad= \!\sum_{\err\in N_{\ell}} \! \frac{\left|\partial\Omega^{\ell\err}\right|}{\left|\Omega^{\ell}\right|}  \left( 
		\p^{\ell} \cdot \epsilon^{\ell\err}\frac{\Delta \q^{\ell\err}}{\delta^{\ell\err}}
		 + \partial_{\rho S} \mathcal{E}^{\ell}  \frac{1}{4}\epsilon^{\ell\err} \frac{\Delta \q^{\ell\err}}{T^{\ell}} \partial^2_{\q\q} \mathcal{E}^{\ell\err} \frac{\Delta \q^{\ell\err}}{\delta^{\ell\err}} \right)\!= \!\frac{1}{\left|\Omega^{\ell}\right|}\! \sum_{\err\in N_{\ell}}\! \left|\partial\Omega^{\ell\err}\right| g_{\mathcal{E},\,\normal}^{\ell\err}.
		 \nonumber 		 
	\end{eqnarray}

	From \eqref{eqn.Etime_2d}, \eqref{eqn.Efluc_2d}, \eqref{eqn.Ediff_2D} and noting that the dot product of $ \partial_{\q} \mathcal{E}^{\ell}$ by the green terms in \eqref{eqn.schemeGPR2D_mass}-\eqref{eqn.schemeGPR2D_heatf} is zero, we conclude
	\begin{equation*}
		\frac{\partial \mathcal{E}^{\ell}}{\partial t} = 
		-\frac{1}{\left|\Omega^{\ell}\right|} \sum_{\err\in N_{\ell}} \left|\partial\Omega^{\ell\err}\right| D_{\mathcal{E}}^{\ell \err,-}
		\textcolor{blue}{+ \frac{1}{\left|\Omega^{\ell}\right|} \sum_{\err\in N_{\ell}} \left|\partial\Omega^{\ell\err}\right| g_{\mathcal{E},\,\normal}^{\ell\err}}.
	\end{equation*}

	The second part of the proof concerns marginal stability. Integration of equation \eqref{eqn.schemeGPR2D_energy} over the computational domain $\Omega$ gives
	\begin{equation*}
		\int_{\Omega} \frac{\partial \mathcal{E}^{\ell}}{\partial t} dV
		= \sum_{\ell} \left| \Omega^{\ell}\right| \frac{\partial \mathcal{E}^{\ell}}{\partial t}
		= -\sum_{\ell} \sum_{\err\in N_{\ell}} \left|\partial\Omega^{\ell\err}\right| D_{\mathcal{E}}^{\ell \err,-}
		\textcolor{blue}{+ \sum_{\ell} \sum_{\err\in N_{\ell}} \left|\partial\Omega^{\ell\err}\right| g_{\mathcal{E},\,\normal}^{\ell\err}}.
	\end{equation*}
	Assuming that the solution on the boundaries of the domain tends to a constant value, the jumps on $\q$ become zero at $\partial \Omega$ and fluctuations and dissipative terms vanish. Besides, the remaining dissipative terms can be seen as a telescopic sum which cancels.
	Reordering of the first summation in the right hand side of the former equation to cluster the contributions at each face we obtain
	\begin{equation*}
		\int_{\Omega} \frac{\partial \mathcal{E}^{\ell}}{\partial t} dV
		= \sum_{\ell} \left| \Omega^{\ell}\right| \frac{\partial \mathcal{E}^{\ell}}{\partial t} = -\sum_{\ell\err} \left|\partial\Omega^{\ell\err}\right|  \left(D_{\mathcal{E}}^{\ell \err,-}
		+  D_{\mathcal{E}}^{ \err \ell,-}\right) .
	\end{equation*}
	Consequently, marginal stability is proven given that the contributions of fluctuations in the interior cell boundaries cancel.
	Let us focus on a face $\partial \Omega^{\ell\err}$. We start analysing the terms corresponding 
	with the Godunov formalism (black terms) in \eqref{eqn.Efluc_2d}:
	\begin{eqnarray}
		  \partial_{\rho}\mathcal{E}^{\ell}_{1} D_{\rho}^{\ell \err,-}
		+ \partial_{\rho}\mathcal{E}^{\ell}_{2} D_{\rho}^{\ell \err,-}		
		+ \partial_{\rho v_{i}}\mathcal{E}^{\ell} D_{\rho v_{i}}^{\ell \err,-}
		+ \partial_{\rho S}\mathcal{E}^{\ell} D_{\rho S}^{\ell \err,-} \nonumber \\
		+ \partial_{\rho}\mathcal{E}^{\err}_{1} D_{\rho}^{ \err\ell,-}
		+ \partial_{\rho}\mathcal{E}^{\err}_{2} D_{\rho}^{ \err\ell,-}		
		+ \partial_{\rho v_{i}}\mathcal{E}^{\err} D_{\rho v_{i}}^{ \err\ell,-} 
		+ \partial_{\rho S}\mathcal{E}^{\err} D_{\rho S}^{ \err\ell,-}
		\nonumber \\
		= -\left( \partial_{\rho}\mathcal{E}^{\err}_{1} - \partial_{\rho}\mathcal{E}^{\ell}_{1} 
		\right) 
		f_{\rho,\,k}^{\ell\err} n_k
		-\left( \partial_{\rho}\mathcal{E}^{\err}_{2} - \partial_{\rho}\mathcal{E}^{\ell}_{2} 
		\right) 
		f_{\rho,\,k}^{\ell\err} n_k
		-\left( \partial_{\rho v_{i}}\mathcal{E}^{\err} - \partial_{\rho v_{i}}\mathcal{E}^{\ell} 
		\right) f_{\rho 
		v_{i},\,k}^{\ell\err} n_k
		\nonumber \\
		-\left( \partial_{\rho S}\mathcal{E}^{\err} - \partial_{\rho S}\mathcal{E}^{\ell} \right) 
		f_{\rho 
		S,\,k}^{\ell\err} n_k
		-\partial_{\rho}\mathcal{E}^{\ell}_{1} f_{\rho,\,k}^{\ell}n_k 
		-\partial_{\rho}\mathcal{E}^{\ell}_{2} f_{\rho,\,k}^{\ell}n_k 
		-\partial_{\rho v_{i}}\mathcal{E}^{\ell} f_{\rho v_{i},\,k}^{\ell}n_k
		\nonumber \\
		-\partial_{\rho S}\mathcal{E}^{\ell} f_{\rho S,\,k}^{\ell}n_k
		+\partial_{\rho}\mathcal{E}^{\err}_{1} f_{\rho,\,k}^{\err}n_k
		+\partial_{\rho}\mathcal{E}^{\err}_{2} f_{\rho,\,k}^{\err}n_k
		+\partial_{\rho v_{i}}\mathcal{E}^{\err} f_{\rho v_{i},\,k}^{\err}n_k
		+\partial_{\rho S}\mathcal{E}^{\err} f_{\rho S,\,k}^{\err}n_k
		\nonumber \\
		= -\left(\p^{\err} -\p^{\ell}\right) \cdot f_{\q,\,k}^{\ell\err}n_k + \p^{\err} \cdot 
		f_{\q,\,k}^{\err}n_k - \p^{\ell} \cdot f_{\q,\,k}^{\ell}n_k
		\nonumber \\
		= \left( \p^{\err} \cdot f_{\q,\,k}^{\err} - (v_{k}L)^{\err} \right) n_k 
		 -\left( \p^{\ell} \cdot f_{\q,\,k}^{\ell} - (v_{k}L)^{\ell} \right) n_k
		= F^{\err}_{G} - F^{\ell}_{G}, \label{eqn.mstabilityblack_2D}
	\end{eqnarray}
with $ F_{G} $ standing for the black terms in the energy flux in \eqref{eqn.energy}.
	Regarding the red terms in \eqref{eqn.Efluc_2d}, we have
	\begin{eqnarray}
		\textcolor{black}{  \partial_{\rho}\mathcal{E}^{\ell}_{3} D_{\rho}^{\ell \err,-}}
		\textcolor{black}{+ \partial_{\rho}\mathcal{E}^{\ell}_{4} D_{\rho}^{\ell \err,-}}
		\textcolor{black}{+ \partial_{\rho 
		v_{i}}\mathcal{E}^{\ell}\left(\sigma_{ik}^{\ell\err,-}n_k 
		+\omega_{ik}^{\ell\err,-}n_k \right) } 
		\textcolor{black}{+ \partial_{\rho S}\mathcal{E}^{\ell}  \left(\beta_{k}^{\ell\err} - 
		\beta_{k}^{\ell}\right) n_k  }
		\nonumber\\
		\textcolor{black}{+ 
		\partial_{A_{ik}}\mathcal{E}^{\ell}\frac{1}{2\delta^{\ell\err}}A^{\ell\err}_{im}\left(v_{m}^{\err}-v_{m}^{\ell}
		 \right) 
		n_k }
		\textcolor{black}{+ \partial_{A_{ik}}\mathcal{E}^{\ell}  \frac{1}{2} 
		\tilde{u}_{A,\,\normal}^{\ell\err}   \left(A_{ik}^{\err}-A_{ik}^{\ell}\right)  }
		\nonumber\\
		\textcolor{black}{+ \partial_{J_{k}}\mathcal{E}^{\ell} 
		\frac{1}{2\delta^{\ell\err}}J^{\ell\err}_{i}\left(v_{m}^{\err}-v_{m}^{\ell} \right) n_k} 
		\textcolor{black}{+ \partial_{J_{k}}\mathcal{E}^{\ell} \frac{1}{2} 
		\tilde{u}_{J,\,\normal}^{\ell\err}   \left(J_{k}^{\err}-J_{k}^{\ell}\right)   }
		\textcolor{black}{+ \partial_{J_{k}}\mathcal{E}^{\ell}  \frac{1}{2} T^{\ell\err,-} n_k 
		\nonumber} \\
		\textcolor{black}{+ \partial_{\rho}\mathcal{E}^{\err}_{3} D_{\rho}^{\err\ell,-}}
		\textcolor{black}{+ \partial_{\rho}\mathcal{E}^{\err}_{4} D_{\rho}^{\err\ell,-}}
		\textcolor{black}{-\partial_{\rho v_{i}}\mathcal{E}^{\err}\left(\sigma_{ik}^{\err\ell,-}n_k 
		+\omega_{ik}^{\err\ell,-}n_k \right) } 
		\textcolor{black}{- \partial_{\rho S}\mathcal{E}^{\err}  \left(\beta_{k}^{\err\ell} - 
		\beta_{k}^{\err}\right) n_k  }
		\nonumber\\
		\textcolor{black}{- 
		\partial_{A_{ik}}\mathcal{E}^{\err}\frac{1}{2}A^{\err\ell}_{im}\left(v_{m}^{\ell}-
		 v_{m}^{\err}\right) 
		n_k }
		\textcolor{black}{+ \partial_{A_{ik}}\mathcal{E}^{\err}  \frac{1}{2} 
		\tilde{u}_{A,\,\normal}^{\err\ell}   \left(A_{ik}^{\ell}-A_{ik}^{\err}\right)  }
		\nonumber\\
		\textcolor{black}{- \partial_{J_{k}}\mathcal{E}^{\err} 
		\frac{1}{2}J^{\err\ell}_{i}\left(v_{m}^{\ell}- v_{m}^{\err}\right) n_k} 
		\textcolor{black}{+ \partial_{J_{k}}\mathcal{E}^{\err} \frac{1}{2} 
		\tilde{u}_{J,\,\normal}^{\err\ell}   \left(J_{k}^{\ell}-J_{k}^{\err}\right)   }
		\textcolor{black}{- \partial_{J_{k}}\mathcal{E}^{\err}  \frac{1}{2} T^{\err\ell,-}n_k }. 
		\nonumber
	\end{eqnarray}
	Substitution of $\partial_{\q} \mathcal{E}$ by its expression in state variables, $\q$, together with 
	\eqref{eqn.fluctuations_2D} yields
	\begin{eqnarray}
		\textcolor{black}{E^{\ell}_{3} \left(  f^{\ell\err}_{\rho,\,k}-\! f^{\ell}_{\rho,\,k} 
		\right)\! n_k }
		\textcolor{black}{- E^{\err}_{3}  \left(  f^{\ell\err}_{\rho,\,k} -\! 
		f^{\err}_{\rho,\,k}\right)\! n_k}
		\textcolor{black}{+ \alpha_{ik}^{\ell}  \frac{1}{2} \tilde{u}_{A,\,\normal}^{\ell\err}   \left(A_{ik}^{\err}- \! A_{ik}^{\ell}\right)  }
		\textcolor{black}{+ \alpha_{ik}^{\err}  \frac{1}{2} \tilde{u}_{A,\,\normal}^{\err\ell}   \left(A_{ik}^{\ell}- \! A_{ik}^{\err}\right)  }
		\nonumber\\
		\textcolor{black}{+  E^{\ell}_{4} \left(  f^{\ell\err}_{\rho,\,k}- f^{\ell}_{\rho,\,k} 
		\right)n_k }
		\textcolor{black}{- E^{\err}_{4}  \left(  f^{\ell\err}_{\rho,\,k} - 
		f^{\err}_{\rho,\,k}\right) n_k}
		\textcolor{black}{+ \beta_{k}^{\ell} \frac{1}{2} \tilde{u}_{J,\,\normal}^{\ell\err}   \left(J_{k}^{\err}-J_{k}^{\ell}\right)   }
		\textcolor{black}{+ \beta_{k}^{\err} \frac{1}{2} \tilde{u}_{J,\,\normal}^{\err\ell}   \left(J_{k}^{\ell}-J_{k}^{\err}\right)   }
		\nonumber\\
		\textcolor{black}{+\left[ v_{i}^{\ell}\! \left(\sigma_{ik}^{\ell\err}-\!\sigma_{ik}^{\ell} \right) 
			-v_{i}^{\err}\! \left(\sigma_{ik}^{\ell\err}-\!\sigma_{ik}^{\err} \right)
			+ \alpha_{ik}^{\ell}\frac{1}{2}A^{\ell\err}_{im}\! \left(v_{m}^{\err}-\! v_{m}^{\ell} \right) 
			- \alpha_{ik}^{\err}\frac{1}{2}A^{\err\ell}_{im}\! 
			\left(v_{m}^{\ell}-\! v_{m}^{\err}\right)\right]  n_k }
		\nonumber\\
		\textcolor{black}{+\left[ v_{i}^{\ell} \left( \omega_{ik}^{\ell\err}-\omega_{ik}^{\ell} \right) 
			-v_{i}^{\err}\left(\omega_{ik}^{\ell\err}-\omega_{ik}^{\err} \right)
			+ \beta_{k}^{\ell} \frac{1}{2}J^{\ell\err}_{i}\left(v_{m}^{\err}-v_{m}^{\ell} \right)
			- \beta_{k}^{\err} \frac{1}{2}J^{\err\ell}_{i}\left(v_{m}^{\ell}- 
			v_{m}^{\err}\right)\right]   n_k}
		\nonumber\\
		\textcolor{black}{+ \left[ T^{\ell} \left( \beta_{k}^{\ell\err} - \beta_{k}^{\ell} \right)   
			-T^{\err}  \left( \beta_{k}^{\ell\err}  - \beta_{k}^{\err} \right)  
			+ \beta_{k}^{\ell}  \frac{1}{2} T^{\ell\err,-} 
			- \beta_{k}^{\err}  \frac{1}{2} T^{\err\ell,-}\right] n_k }.	\nonumber	
	\end{eqnarray}
	Taking into account \eqref{eqn.sigma_2D}-\eqref{eqn.utildeJ_2D} and collecting terms gives
	\begin{eqnarray}
		\textcolor{black}{\left( \rho v_{k}^{\err}E^{\err}_{3} 		
		+ \rho v_{k}^{\err}E^{\err}_{4} 
		+ v_{i}^{\err}\sigma_{ik}^{\err}
		+ v_{i}^{\err}\omega_{ik}^{\err}
		+  \beta_{k}^{\err}T^{\err}\right)  n_k} 
		\nonumber\\ 
		-\textcolor{black}{\left( \rho v_{k}^{\ell}E^{\ell}_{3}	 
		+ \rho v_{k}^{\ell}E^{\ell}_{4} 
		+ v_{i}^{\err}\sigma_{ik}^{\err}
		+ v_{i}^{\err}\omega_{ik}^{\err}
		+ \beta_{k}^{\ell}T^{\ell}\right)n_k  }.	\label{eqn.mstabilityred_2D}
	\end{eqnarray}
	Gathering \eqref{eqn.mstabilityblack_2D} and \eqref{eqn.mstabilityred_2D}, we obtain
	\begin{eqnarray}
		D_{\mathcal{E}}^{\ell \err,-} +  D_{\mathcal{E}}^{ \err \ell,-} &=& F^{\err}_{G}  + 
		\textcolor{red}{\left( \rho v_{k}^{\err}E^{\err}_{3} 		
			+ \rho v_{k}^{\err}E^{\err}_{4} 
			+ v_{i}^{\err}\sigma_{ik}^{\err}
			+ v_{i}^{\err}\omega_{ik}^{\err}
			+  \beta_{k}^{\err}T^{\err}\right)  n_k} -
		\nonumber\\ 
		&&  F^{\ell}_{G} -\textcolor{red}{\left( \rho v_{k}^{\ell}E^{\ell}_{3}	 
			+ \rho v_{k}^{\ell}E^{\ell}_{4} 
			+ v_{i}^{\ell}\sigma_{ik}^{\ell}
			+ v_{i}^{\ell}\omega_{ik}^{\ell}
			+ \beta_{k}^{\ell}T^{\ell}\right)n_k  }
		%\nonumber\\
		=  F^{\err}-F^{\ell}.
	\end{eqnarray}
	Thus the fluctuations can be seen as the difference between fluxes which will cancel out
	when adding the contributions of all cells, and hence the scheme is marginally stable in the 
	energy 
	norm, as claimed:
	\begin{equation}
		\int_{\Omega} \frac{\partial \mathcal{E}^{\ell}}{\partial t} dV
		= \sum_{\ell} | \Omega^{\ell} | \frac{\partial \mathcal{E}^{\ell}}{\partial t} = -\sum_{\ell\err} \left( D_{\mathcal{E}}^{\ell \err,-} +  D_{\mathcal{E}}^{ \ell\err,+}\right)  
		= -\sum_{\ell\err}\left(F^{\err}-F^{\ell} \right) = 0.
	\end{equation}
\end{proof}	

%\clearpage 

\begin{theorem}
Assuming $T^{\ell} >0$ and $\boldsymbol{\mathcal{H}}^{\ell\err}=\partial^2_{\q \q} \mathcal{E}^{\ell\err} \geq 0$ the semi-discrete finite volume scheme \eqref{eqn.schemeGPR2D} with production term \eqref{eqn.production_2D} satisfies the semi-discrete cell entropy inequality 
\begin{equation}
	\frac{\partial \rho S^{\ell}}{\partial t}  + \sum_{\err\in N_{\ell}}
	\frac{\left|\partial\Omega^{\ell\err}\right|}{\left|\Omega^{\ell}\right|}  D_{\rho S}^{\ell \err,-}
	+ \sum_{\err\in N_{\ell}}  
	\frac{\left|\partial\Omega^{\ell\err}\right|}{\left|\Omega^{\ell}\right|} \left( 
	\beta_{k}^{\ell\err}-\beta_{k}^{\ell} \right)n_k   
	 {-  \sum_{\err\in N_{\ell}}  \frac{\left|\partial\Omega^{\ell\err}\right|}{\left|\Omega^{\ell}\right|} g_{\rho S,\,\normal}^{\ell\err}}
	 \geq 0.   \label{eqn.schemeGPR2D_entropy.ineq}
\end{equation} 
\end{theorem}
\begin{proof}
The proof is an immediate consequence of the discretization \eqref{eqn.schemeGPR2D_entropy} with \eqref{eqn.production_2D}:  
\begin{eqnarray}
\frac{\partial \rho S^{\ell}}{\partial t}  + \sum_{\err\in N_{\ell}}
\frac{\left|\partial\Omega^{\ell\err}\right|}{\left|\Omega^{\ell}\right|}  D_{\rho S}^{\ell \err,-}
+ \sum_{\err\in N_{\ell}}  
\frac{\left|\partial\Omega^{\ell\err}\right|}{\left|\Omega^{\ell}\right|} \left( 
\beta_{k}^{\ell\err}-\beta_{k}^{\ell} \right)n_k   
{-  \sum_{\err\in N_{\ell}}  \frac{\left|\partial\Omega^{\ell\err}\right|}{\left|\Omega^{\ell}\right|} g_{\rho S,\,\normal}^{\ell\err}} && \nonumber \\ 
= \textcolor{blue}{\frac{1}{\left|\Omega^{\ell}\right|} \sum_{\err\in N_{\ell}} 
	\left|\partial\Omega^{\ell\err}\right| \Pi^{\ell\err,-}_{\normal}}
\textcolor{darkgreen}{+\frac{\alpha_{ik}^{\ell}\alpha_{ik}^{\ell}}{\theta_1^{\ell}\left(\tau_{1}\right) T^{\ell} } }  
\textcolor{darkgreen}{+\frac{\beta_{i}^{\ell}\beta_{i}^{\ell}}{\theta_2^{\ell}\left(\tau_{2}\right) T^{\ell} } } \geq 0, && \label{eqn.schemeGPR2D_entropy2}
\end{eqnarray} 	
since $\Pi^{\ell\err,-}_{\normal} = \frac{1}{2}\epsilon^{\ell\err} \frac{\Delta \q^{\ell\err}}{T^{\ell}} \partial^2_{\q \q} \mathcal{E}^{\ell\err} \frac{\Delta \q^{\ell\err}}{\delta^{\ell\err}} \geq 0$ due to $\partial^2_{\q \q} \mathcal{E} \geq 0$ and $\theta_{1,2}^{\ell} > 0$ as well as $T^{\ell} > 0$.  		
\end{proof}

%%%%%%%%%%%%%%%%%%%%%%%%%%%%%%%%%%%%%%%%%%%%%%%%%%%%%%%%%%%%%%%%%%
\section{Thermodynamically compatible fully-discrete finite volume scheme for the Euler subsystem}
\label{sec.scheme.fd}

In this section we present a \textit{fully-discrete} finite volume scheme for the Euler subsystem 
of \eqref{eqn.GPR}, i.e. for the black and blue terms. For simplicity, we restrict the 
considerations to one space dimension. As before, the spatial control volumes are denoted by 
$\Omega^{\ell} = [x^{\ell-\halb}, x^{\ell+\halb}]$. The scheme reads  
\begin{equation} 
\label{eqn.fluct.fd} 
 \frac{\q^{n+1,\ell}-\q^{n,\ell}}{\Delta t}   = 
 - \frac{ (\f^{\ell+\halb}_{\tilde{\p}} - \f^{\ell}) - 
 	( \f^{\ell-\halb}_{\tilde{\p}} - \f^{\ell}) }{\Delta x} 
 +
 \textcolor{blue}{
 \frac{\g^{\ell+\halb}_{\tilde{\p}} - \g^{\ell+\halb}_{\tilde{\p}}}{\Delta x}
 + \tilde{\mathbf{P}}^{\ell}},
\end{equation}
Again, the subscript $ \tilde{\p} $ refers to the fact that the flux is evaluated using the segment path in $ 
\p $ variables 
defined below, similar to \eqref{eqn.path1_sd1d}-\eqref{eqn.p.scheme_sd1d}.
In order to construct a thermodynamically compatible \textit{fully-discrete} scheme, where the total energy conservation law \eqref{eqn.energy} is a \textit{consequence} of the discrete equations \eqref{eqn.fluct.fd}, we introduce a new average quantity $\tilde{\p}^{\ell}$. Since by construction one has 
\begin{equation}
\int \limits_{\q^{n,\ell}}^{\q^{n+1,\ell}} \partial_{\q} \mathcal{E} \cdot d\q = \mathcal{E}^{n+1,\ell}  - \mathcal{E}^{n,\ell},   
\end{equation}
for any path connecting $\q^{n,\ell}$ with $\q^{n+1,\ell}$, we define the quantity $\tilde{\p}^{\ell}$ as 
\begin{equation}
\label{eqn.ptilde.def} 
\tilde{\p}^{\ell} = \int \limits_0^1 \partial_{\q} \mathcal{E}(\boldsymbol{\tau}(s)) ds,   
\end{equation}
with the straight-line segment path 
\begin{equation}
\label{eqn.timepath} 
\boldsymbol{\tau} = \boldsymbol{\tau}(s) = \q^{n,\ell} + s \left( \q^{n+1,\ell}  - \q^{n,\ell} \right). 
\end{equation}
Therefore, $\tilde{\p}^{\ell}$ satisfies the Roe-type property  
\begin{equation}
\tilde{\p}^{\ell} \cdot \left( \q^{n+1,\ell}  - \q^{n,\ell} \right) = \mathcal{E}^{n+1,\ell}  - \mathcal{E}^{n,\ell},  
\label{eqn.roeprop.fd}  
\end{equation}
which is fundamental for the construction of our thermodynamically compatible fully-discrete scheme. 
We now multiply \eqref{eqn.fluct.fd} with $\tilde{\p}^{\ell}$ from the left and neglecting the viscous fluxes $\g^{\ell \pm \halb}$ leads to 
\begin{equation} 
\label{eqn.fluct.pt} 
\tilde{\p}^{\ell} \cdot \frac{\q^{n+1,\ell}  - \q^{n,\ell}}{\Delta t} = \frac{\mathcal{E}^{n+1,\ell}  - \mathcal{E}^{n,\ell}}{\Delta t}   = - \tilde{\p}^{\ell} \cdot \frac{ (\f^{\ell+\halb}_{\tilde{\p}} - \f^{\ell}) + ( \f^{\ell} - \f^{\ell-\halb}_{\tilde{\p}} ) }{\Delta x}.   
\end{equation}
To obtain a \textit{conservative} form of the fully discrete energy conservation law, we  require  
\begin{equation}
	\label{eqn.aux1.fd}
\tilde{\p}^{\ell} \cdot ( \f^{\ell+\halb}_{\tilde{\p}}  - \f^{\ell} ) + \tilde{\p}^{\ell+1}  \cdot (\f^{\ell+1} - \f^{\ell+\halb}_{\tilde{\p}} ) = \tilde{F}^{\ell+1} - \tilde{F}^{\ell}. 
\end{equation} 
Using the parametrization \eqref{eqn.par0} and the associated relations \eqref{eqn.par02} we get
\begin{equation} 
- \partial_{\p} (\widetilde{v_1 L})^{\ell+\halb} \cdot \left( \tilde{\p}^{\ell+1} - \tilde{\p}^{\ell} 
\right)  + \tilde{\p}^{\ell+1} \cdot \f^{\ell+1} - \tilde{\p}^{\ell} \cdot \f^{\ell} = 
\tilde{\p}^{\ell+1} \cdot \f^{\ell+1} - \widetilde{v_1 L}^{\ell+1} - \tilde{\p}^{\ell} \cdot \f^{\ell} + \widetilde{v_1 L}^{\ell}.  
\end{equation} 
Hence, the numerical flux $\f^{\ell+\halb}_{\tilde{\p}} = \partial_{\p} (\widetilde{v_1 L})^{\ell+\halb} $ 
must satisfy 
the following jump condition: 
%in terms of the $\tilde{\p}$ variables:  
\begin{equation}
\label{eqn.fluxcondition}
\f^{\ell+\halb}_{\tilde{\p}} \cdot \left( \tilde{\p}^{\ell+1} - \tilde{\p}^{\ell} \right)  =  
\partial_{\p} (\widetilde{v_1 L})^{\ell+\halb} \cdot \left( \tilde{\p}^{\ell+1} - \tilde{\p}^{\ell} \right) 
= \widetilde{v_1 L}^{\ell+1} - \widetilde{v_1 L}^{\ell}. 
\end{equation} 
We choose again a simple straight line segment path, this time in the $\tilde{\p}$ 
variables: 
	\begin{equation} 
	\label{eqn.path1.tilde} 
	\boldsymbol{\psi}(s) = 
	\boldsymbol{\psi}(\tilde{\p}^{\ell},\tilde{\p}^{\ell+1},s) = \tilde{\p}^{\ell} + s \left( \tilde{\p}^{\ell+1} - \tilde{\p}^{\ell} \right), \qquad 0 \leq s \leq 1,
	\end{equation}  
Using the same reasoning as for the semi-discrete scheme \eqref{eqn.pathint_sd1d}-\eqref{eqn.p.scheme.diss}, we find the thermodynamically compatible numerical flux of the \textit{fully discrete} $\p$-scheme as 
	\begin{equation}
	\label{eqn.p.scheme}
	\f^{\ell+\halb}_{\tilde{\p}} = 
	\left( f^{\ell+\halb}_{\rho}, f^{\ell+\halb}_{\rho v_i}, f^{\ell+\halb}_{\rho S} \right)^T 
	 = \int \limits_{0}^{1} \f(\boldsymbol{\psi}(\tilde{\p}^{\ell},\tilde{\p}^{\ell+1},s)) ds, 
	\end{equation} 
with the jump $\Delta \tilde{\p}^{\ell+\halb} = \tilde{\p}^{\ell+1} - \tilde{\p}^{\ell}$ and the numerical viscosity flux defined as 
\begin{equation}
	\g^{\ell+\halb}_{\tilde{\p}} 
	= 
	\left( g^{\ell+\halb}_{\rho}, g^{\ell+\halb}_{\rho v_i}, g^{\ell+\halb}_{\rho S} \right)^T
	= \epsilon^{\ell+\halb} \,  \partial^2_{\p\p} \tilde{L}^{\ell+\halb} \, \frac{\tilde{\p}^{\ell+1} - \tilde{\p}^{\ell} }{\Delta x}. 
	\label{eqn.numvisc.fd}
\end{equation}
The corresponding production term reads $\tilde{\mathbf{P}}^{\ell} = 
(0,\mathbf{0},\tilde{\Pi}^{\ell})^T$ with 
\begin{equation} 
	\tilde{\p}^{\ell} \cdot \tilde{\mathbf{P}}^{\ell} = \tilde{T}^{\ell} \tilde{\Pi}^{\ell} = 
	\halb \epsilon^{\ell+\halb} \frac{\Delta \tilde{\p}^{\ell+\halb}}{\Delta x} \cdot \partial^2_{\p\p} \tilde{L}^{\ell+\halb} \frac{\Delta \tilde{\p}^{\ell+\halb}}{\Delta x} +  
	\halb \epsilon^{\ell-\halb} \frac{\Delta \tilde{\p}^{\ell-\halb}}{\Delta x} \cdot \partial^2_{\p\p} \tilde{L}^{\ell-\halb} \frac{\Delta \tilde{\p}^{\ell-\halb}}{\Delta x}.
	\label{eqn.production.fd} 
\end{equation}

	The disadvantage of the $\p$-scheme is that it requires the expression of the physical flux 
	$\f$ in terms of the $\p$ variables, or, equivalently, it requires the variable transformation 
	$\q=\q(\p)$, which in general may be quite cumbersome. However, for the Euler subsystem at least 
	this conversion is simple and analytic. \textcolor{black}{Recall that 
	$\partial^2_{\p\p} \tilde{L}^{\ell-\halb} = (\partial^2_{\q\q} \tilde{\mathcal{E}}^{\ell-\halb})^{-1} $, see \eqref{eqn.EqqLqqRoe}. }
	Note that the proposed fully-discrete scheme is \textit{implicit}, since $\tilde{\p}^{\ell}$ is a function of $\q^{n,\ell}$ and $\q^{n+1,\ell}$, see \eqref{eqn.ptilde.def} and \eqref{eqn.timepath}. In order to obtain a simple and straightforward implementation of the fully-discrete scheme, we propose the following predictor-corrector approach, based on a Picard-type iteration, \textcolor{black}{similar to the iterative procedure employed in the fully-discrete kinetic energy-preserving 
	scheme proposed for the Euler equations in \cite{Compatiblekin}}: 
\begin{equation} 
\label{eqn.fv.picard} 
\q^{n+1,\ell}_{k+1}  = \q^{n,\ell} - \frac{\Delta t}{\Delta x} \left( \, \f^{\ell+\halb}_{\tilde{\p}_k} - \f^{\ell-\halb}_{\tilde{\p}_k} \,\right)
+ \frac{\Delta t}{\Delta x} \left( \, \g^{\ell+\halb}_{\tilde{\p}_k} - \g^{\ell-\halb}_{\tilde{\p}_k} \,\right) 
+ \tilde{\mathbf{P}}^{\ell}_k, 
\end{equation}	
with the quantity $\tilde{\p}^{\ell}_k$ defined as 
\begin{equation}
\tilde{\p}^{\ell}_k = \int \limits_0^1 \partial_{\q} \mathcal{E}(\boldsymbol{\tau}) ds, 
\qquad \textnormal{with} \qquad 
\boldsymbol{\tau} = \q^{n,\ell} + s \left( \q^{n+1,\ell}_k  - \q^{n,\ell} \right).       
\end{equation}
As initial guess for the iterative scheme we set $\q_0^{n+1,\ell} = \q^{n,\ell}$ and the iterations are stopped when the following condition is satisfied: 
\begin{equation}
    \sum \limits_\ell \left( \mathcal{E}^{n+1,\ell}_{k} - \mathcal{E}\left(\q^{n+1,\ell}_{k+1} 
    \right) \right)^2 < \delta^2. 
\end{equation}
\textcolor{black}{with $\delta>0$ an arbitrarily small tolerance, typically of the order of the machine precision. } 
Recall that 
$
    \mathcal{E}^{n+1,\ell}_{k} = \mathcal{E}^{n,\ell}_{k} + \tilde{\p}^{\ell}_{k} \cdot \left( 
    \q^{n+1,\ell}_{k} - \q^{n,\ell} \right)
$, see \eqref{eqn.roeprop.fd}.
This completes the description of the fully-discrete Godunov formalism for the inviscid Euler subsystem. 

\begin{theorem}
	\label{theorem.energy.fullydiscrete}
	The thermodynamically compatible fully-discrete finite volume \\ scheme 
	%\eqref{eqn.fluct.fd} 
\begin{equation} 
	\label{eqn.fluct.fd2} 
	\frac{\q^{n+1,\ell}-\q^{n,\ell}}{\Delta t}  = - \frac{ (\f^{\ell+\halb}_{\tilde{\p}} - \f^{\ell}) - ( \f^{\ell-\halb}_{\tilde{\p}} - \f^{\ell}) }{\Delta x}  
	+ \frac{\g^{\ell+\halb}_{\tilde{\p}} - \g^{\ell-\halb}_{\tilde{\p}} }{\Delta x} 
	+
	\tilde{\mathbf{P}}^{\ell}. 
\end{equation}	
	with production term $\tilde{\mathbf{P}}^{\ell}$ according to  \eqref{eqn.production.fd} and fluxes \eqref{eqn.p.scheme} and \eqref{eqn.numvisc.fd} verifies the fully discrete energy conservation law
	\begin{equation}
		\frac{\mathcal{E}^{n+1,\ell} - \mathcal{E}^{n,\ell}}{\Delta t} =  - \frac{1}{\Delta x} \left( 
		\tilde{D}_{\mathcal{E}}^{\ell+\halb,-} +  \tilde{D}_{\mathcal{E}}^{\ell-\halb,+}  \right) 
		+		
		\frac{g_{\mathcal{E}}^{\ell+\halb} - g_{\mathcal{E}}^{\ell-\halb}  }{\Delta x}.
		\label{eqn.energy.fd}		
	\end{equation} 
	The fluctuations above are defined as 
	\begin{equation}
		\label{eqn.Efluct.fd}
		\tilde{D}_{\mathcal{E}}^{\ell+\halb,-} = \tilde{\p}^{\ell} \cdot
		(\f^{\ell+\halb}_{\tilde{\p}} - \f^{\ell}), \qquad 
		\tilde{D}_{\mathcal{E}}^{\ell-\halb,+} = \tilde{\p}^{\ell} \cdot
		(\f^{\ell} - \f^{\ell-\halb}_{\tilde{\p}})  
	\end{equation}
	and satisfy
	\begin{equation}
		\label{eqn.fluct.cond.fd}
		\tilde{D}_{\mathcal{E}}^{\ell+\halb,-} + \tilde{D}_{\mathcal{E}}^{\ell+\halb,+} = 
		\tilde{F}^{\ell+1} - \tilde{F}^{\ell} = 
		F(\tilde{\p}^{\ell+1}) - F(\tilde{\p}^{\ell}).  
	\end{equation} 
	The numerical viscosity flux  in \eqref{eqn.energy.fd}  reads 
	\begin{equation}
		g_{\mathcal{E}}^{\ell+\halb} =  
		\halb \frac{\epsilon^{\ell+\halb}}{\Delta x} \left( 
		\tilde{\p}^{\ell} + \tilde{\p}^{\ell+1} \right) \cdot 
		\partial_{\p\p} \tilde{L}^{\ell+\halb} \left( \tilde{\p}^{\ell+1} - \tilde{\p}^{\ell} 
		\right). 
		\label{eqn.viscflux.fd}
	\end{equation}
	As a consequence, for vanishing jumps on the boundary, the scheme is nonlinearly marginally stable in the energy norm. 
\end{theorem}	

\begin{proof}
	Multiplying \eqref{eqn.fluct.fd2} with $\tilde{\p}$ defined according to \eqref{eqn.ptilde.def} and using the Roe property \eqref{eqn.roeprop.fd} yields 
	\begin{equation}
		\tilde{\p}^{\ell} \cdot \frac{\q^{n+1,\ell}-\q^{n,\ell}}{\Delta t} = 
 		\frac{\mathcal{E}^{n+1,\ell} - \mathcal{E}^{n,\ell}}{\Delta t}. 			
	\end{equation}
	Furthermore, using \eqref{eqn.Efluct.fd} one has immediately 
	\begin{equation}
		\tilde{\p}^{\ell} \cdot \left( \f^{\ell+\halb}_{\tilde{\p},d} - \f^{\ell} \right)  +
		\tilde{\p}^{\ell} \cdot \left( \f^{\ell} - \f^{\ell+\halb}_{\tilde{\p},d} \right) = \tilde{D}_{\mathcal{E}}^{\ell+\halb,-} + \tilde{D}_{\mathcal{E}}^{\ell-\halb,+}. 
	\end{equation}
By construction, \eqref{eqn.aux1.fd}-\eqref{eqn.p.scheme}, the fluctuations satisfy \eqref{eqn.fluct.cond.fd}. For the numerical viscosity we get 
after some calculations, see \eqref{eqn.E.diss.2} and \eqref{eqn.diss.2d} for the semi-discrete case,  
\begin{eqnarray}
	\tilde{\p}^{\ell} \cdot \frac{ \g^{\ell+\halb}_{\tilde{\p}} - \g^{\ell-\halb}_{\tilde{\p}} }{\Delta x} 
	+ \tilde{\p}^{\ell} \cdot \tilde{\mathbf{P}} && \nonumber \\   
	= \frac{\epsilon^{\ell+\halb}}{\Delta x} \tilde{\p}^{\ell} \cdot  \partial^2_{\p\p} \tilde{L}^{\ell+\halb} \left( \tilde{\p}^{\ell+1} - \tilde{\p}^{\ell} \right) -
	\frac{\epsilon^{\ell-\halb}}{\Delta x} \tilde{\p}^{\ell} \cdot  \partial^2_{\p\p} \tilde{L}^{\ell-\halb} \left( \tilde{\p}^{\ell} - \tilde{\p}^{\ell-1} \right) 
	+ \tilde{\p}^{\ell} \cdot \tilde{\mathbf{P}} && \nonumber \\
=  \halb \frac{\tilde{\p}^{\ell+1} + \tilde{\p}^{\ell}}{\Delta x}  \cdot \epsilon^{\ell+\halb} \partial^2_{\p\p} \tilde{L}^{\ell+\halb} \frac{\Delta \tilde{\p}^{\ell+\halb}}{\Delta x} - \halb \frac{ \tilde{\p}^{\ell} + \tilde{\p}^{\ell-1}}{\Delta x} \cdot \epsilon^{\ell-\halb} \partial^2_{\p\p} \tilde{L}^{\ell-\halb} \frac{\Delta \tilde{\p}^{\ell-\halb}}{\Delta x}  &&  \nonumber \\    
-\halb \frac{\Delta \tilde{\p}^{\ell+\halb}}{\Delta x}  \cdot \epsilon^{\ell+\halb} \partial^2_{\p\p} \tilde{L}^{\ell+\halb} \frac{\Delta \tilde{\p}^{\ell+\halb}}{\Delta x} - \halb \frac{ \Delta \tilde{\p}^{\ell-\halb}}{\Delta x} \cdot \epsilon^{\ell-\halb} \partial^2_{\p\p} \tilde{L}^{\ell-\halb} \frac{\Delta \tilde{\p}^{\ell-\halb}}{\Delta x} + \tilde{\p}^{\ell} \cdot \tilde{\mathbf{P}} &&  \nonumber \\   
= \frac{g_{\mathcal{E}}^{\ell+\halb} - g_{\mathcal{E}}^{\ell-\halb}}{\Delta x}
\end{eqnarray}
due to the definition \eqref{eqn.viscflux.fd} and the production term that satisfies 
\eqref{eqn.production.fd}. Multiplication of  \eqref{eqn.energy.fd} by $\Delta t \Delta x$ and 
summation over $\Omega^\ell$ yields 
\begin{equation} 
		\sum \limits_{\ell} \Delta x \left( \mathcal{E}^{n+1,\ell} - \mathcal{E}^{n,\ell} \right) =  - \sum \limits_{\ell}  \Delta t \left( 
\tilde{D}_{\mathcal{E}}^{\ell+\halb,-} +  \tilde{D}_{\mathcal{E}}^{\ell-\halb,+}  
+		
g_{\mathcal{E}}^{\ell+\halb} - g_{\mathcal{E}}^{\ell-\halb} 
\right) = 0. 
\label{eqn.energysum.fd}
\end{equation} 
The terms on the right hand side of \eqref{eqn.energysum.fd} are a telescopic sum that vanishes because the fluctuations satisfy \eqref{eqn.fluct.cond.fd} and since the jumps vanish at the boundary.
\end{proof}

\begin{theorem}
	The fully-discrete finite volume scheme \eqref{eqn.fluct.fd2}  
	with production term $\tilde{\mathbf{P}}^{\ell}$ according to  \eqref{eqn.production.fd} and fluxes \eqref{eqn.p.scheme} and \eqref{eqn.numvisc.fd} satisfies the fully discrete cell entropy inequality 
	\begin{equation}
		(\rho S)^{n+1,\ell} \geq  
		(\rho S)^{n,\ell} - \frac{\Delta t}{\Delta x} \left(
		f_{\rho S}^{\ell+\halb}-f_{\rho S}^{\ell-\halb} \right) 
		+ 
		\frac{\Delta t}{\Delta x} \left(
		g_{\rho S}^{\ell+\halb}-g_{\rho S}^{\ell-\halb} \right), 
		\label{eqn.entropyinequality.fd} 
	\end{equation}
	assuming that $\tilde{T}^{\ell} > 0$ and 
	$\partial^2_{\p\p} \tilde{L}^{\ell \pm \halb} > 0$. 
\end{theorem}
\begin{proof}
	The fully-discrete form of the entropy equation \eqref{eqn.entropy} according to the scheme \eqref{eqn.fluct.fd2} reads 
	\begin{equation}
	(\rho S)^{n+1,\ell} =  
	(\rho S)^{n,\ell} - \frac{\Delta t}{\Delta x} \left(
	f_{\rho S}^{\ell+\halb}-f_{\rho S}^{\ell-\halb} \right) 
	+ 
	\frac{\Delta t}{\Delta x} \left(
	g_{\rho S}^{\ell+\halb}-g_{\rho S}^{\ell-\halb} \right) + 
	\Delta t \, \tilde{\Pi}^{\ell}, 
\end{equation}	
with 
\begin{equation}
	 \tilde{\Pi}^{\ell} = \frac{1}{\tilde{T}^{\ell}} \left(  
	\halb \epsilon^{\ell+\halb} \frac{\Delta \tilde{\p}^{\ell+\halb}}{\Delta x} \cdot \partial^2_{\p\p} \tilde{L}^{\ell+\halb} \frac{\Delta \tilde{\p}^{\ell+\halb}}{\Delta x} +  
	\halb \epsilon^{\ell-\halb} \frac{\Delta \tilde{\p}^{\ell-\halb}}{\Delta x} \cdot \partial^2_{\p\p} \tilde{L}^{\ell-\halb} \frac{\Delta \tilde{\p}^{\ell-\halb}}{\Delta x} \right) \geq 0,
\end{equation}
since we assume $\tilde{T}^{\ell} > 0$ and $\partial^2_{\p\p} \tilde{L}^{\ell \pm \halb} > 0$, hence one directly obtains the inequality \eqref{eqn.entropyinequality.fd}. 
\end{proof} 

\section{Numerical results}
\label{sec.results}

The new schemes for hyperbolic and thermodynamically compatible PDE systems (HTC schemes) proposed 
in this paper do \textit{not} discretize the energy conservation law \eqref{eqn.energy} explicitly, 
but consider the \textit{entropy inequality} \eqref{eqn.entropy} instead. Semi-discrete / 
fully-discrete energy conservation is obtained as a mere consequence of the thermodynamically 
compatible discretization of the PDEs \eqref{eqn.conti}-\eqref{eqn.heatflux}. As such, the proposed 
approach is \textit{different} from most existing finite volume discretizations. The main aim 
of the following numerical test problems is  therefore to show that the scheme is able to compute 
correct solutions to problems with shock waves, as predicted by Theorems 
\ref{theorem.energy.semidiscrete} and \ref{theorem.energy.fullydiscrete} on the semi-discrete and 
fully-discrete energy conservation, respectively. Furthermore, we check numerically whether the  
relaxation limit of the model (Navier-Stokes limit) is properly captured, i.e. when for 
sufficiently small relaxation times $\tau_1$ and $\tau_2$ the behaviour of the medium becomes the 
one of a viscous heat-conducting Newtonian fluid. For more advanced applications of the GPR model, 
the reader is referred to 
\cite{GPRmodel,GPRmodelMHD,FrontierADERGPR,crack1,nonNewtonian2021,GRGPR}. 
In the following tests, when a viscosity coefficient $\mu$ is specified together with a shear  sound speed $c_s$, the corresponding relaxation time $\tau_1$ is calculated as $\tau_1 = 6 \mu / (\rho_0 c_s^2)$, according to \eqref{eqn.asymptoticlimit} and the results of the asymptotic analysis carried out in  \cite{GPRmodel}. In all tests of this section, the semi-discrete HTC schemes are integrated in time using an explicit third order TVD Runge-Kutta scheme, see \cite{shuosher1,shu2}. For more efficient IMEX Runge-Kutta schemes in the case of stiff relaxation source terms, see the work of Pareschi \& Russo \cite{PareschiRusso2000,PareschiRusso2005} as well as \cite{Naldi_Pareschi,Caflish_Jin_Russo,Buet_Despres,JinPareschiToscani}.    
\textcolor{black}{In all numerical tests carried out with the semi-discrete scheme, we assume the time step $\Delta t$ to be small enough so that time discretization errors can be neglected concening the conservation of total energy.
	In the following, if not stated otherwise, the numerical viscosity $\epsilon^{\ell+\halb}$ is chosen according to \eqref{eqn.viscosity}. When explicit values of $\epsilon$ are provided, 
the numerical dissipation is chosen as a constant, $\epsilon^{\ell+\halb}=\epsilon$.  }

\subsection{Numerical convergence study} 
\label{sec.numconv} 

In this section, we solve the isentropic vortex problem forwarded in \cite{HuShuTri} in order to 
verify the accuracy of the proposed HTC schemes. We apply the schemes to the pure inviscid Euler 
equations, i.e. to the black terms in \eqref{eqn.conti}-\eqref{eqn.entropy}, setting $\gamma=1.4$, 
$c_s=0$, $c_h=0$ and $\epsilon=0$. The computational domain is $\Omega = [0,10]^2$ with periodic 
boundaries everywhere. The initial conditions for the perturbations are given in 
\cite{HuShuTri,GPRmodel} and are not repeated here to save space. 
The background velocity is chosen as $\mathbf{v}_0=\mathbf{0}$ so that a stationary vortex is 
obtained. In this situation, the exact solution is given by the initial condition for all times. 
Simulations are run with the semi-discrete HTC scheme until a final time of $t=0.25$ using an 
equidistant Cartesian grid composed of $N_x \times N_y$ control volumes. The $L^2$ errors obtained 
with the semi-discrete HTC schemes at the final time for the density $\rho$, the momentum density 
$\rho v_1$ and the entropy density $\rho S$ are shown in Table \ref{tab.conv.semi} together with 
the corresponding convergence rates. The results for the fully discrete HTC scheme are reported in 
Table \ref{tab.conv.full}. 
One can observe that all proposed HTC schemes are of second order of accuracy.  

\begin{table}
	\renewcommand{\arraystretch}{1.1}
	\caption{Numerical convergence results for the isentropic vortex problem, obtained with the semi-discrete HTC scheme proposed in this paper. The reported $L^2$ error norms refer to a final time of $t=0.25$.  } 
	\label{tab.conv.semi} 
	\begin{center} 
	\begin{tabular}{ccccccc}
	\hline 
	$N_x=N_y$  &  $\|\rho\|^2 $	& $\| \rho v_1 \|^2$ & $ \| \rho S \|^2 $   & $\mathcal{O}(\rho)$ & $\mathcal{O}(\rho v_1)$ & $\mathcal{O}(\rho S)$ \\ 
	\hline
	32	       & 6.1094E-03 & 9.1324E-03 & 4.7896E-04 &     &     &     \\ 
	64	       & 1.5602E-03 & 2.3633E-03 & 1.3256E-04 & 2.0 & 2.0 & 1.9 \\
	128		   & 3.9230E-04 & 5.9585E-04 & 3.3972E-05 & 2.0 & 2.0 & 2.0 \\
	256        & 9.8232E-05 & 1.4928E-04 & 8.5455E-06 & 2.0 & 2.0 & 2.0 \\
	512        & 2.4626E-05 & 3.7369E-05 & 2.1397E-06 & 2.0 & 2.0 & 2.0 \\
	\hline
	\end{tabular} 	
	\end{center}
\end{table}

\begin{table}
	\renewcommand{\arraystretch}{1.1}
	\caption{Numerical convergence results for the isentropic vortex problem, obtained with the fully-discrete HTC scheme proposed in this paper. The reported $L^2$ error norms refer to a final time of $t=0.25$.  } 
	\label{tab.conv.full} 
	\begin{center} 
		\begin{tabular}{ccccccc}
			\hline 
			$N_x=N_y$  &  $\|\rho\|^2 $	& $\| \rho v_1 \|^2$ & $ \| \rho S \|^2 $   & $\mathcal{O}(\rho)$ & $\mathcal{O}(\rho v_1)$ & $\mathcal{O}(\rho S)$ \\ 
			\hline
		32	        & 6.2046E-03 & 9.1891E-03 & 4.8312E-04 &     &     &     \\  
		64	        & 1.5749E-03 & 2.3742E-03 & 1.3384E-04 & 2.0 & 2.0 & 1.9 \\
		128	        & 3.9424E-04 & 5.9737E-04 & 3.4170E-05 & 2.0 & 2.0 & 2.0 \\
		256         & 9.8481E-05 & 1.4948E-04 & 8.5718E-06 & 2.0 & 2.0 & 2.0 \\
		512         & 2.4658E-05 & 3.7394E-05 & 2.1430E-06 & 2.0 & 2.0 & 2.0 \\
			\hline
		\end{tabular} 	
	\end{center}
\end{table}

\subsection{Simple shear motion in solids and fluids} 
\label{sec.shearlayer} 

We first apply the new HTC schemes to simple shear motion in solids and fluids. The one-dimensional computational domain is $\Omega = [-0.5,+0.5]$ and the initial condition of the problem, which is also prescribed at the boundaries of $\Omega$, is given by   
$\rho=1$, $v_1=v_3=0$, $\textcolor{red}{p=1}$, $\AAA=\mathbf{I}$, $\mathbf{J}=\mathbf{0}$, while the velocity 
component $v_2$ is $v_2 = -v_0$ for $x<0$
and $v_2=+v_0$ for $x\geq 0$, with $v_0=0.1$. The remaining parameters of this test are 
$\gamma=1.4$, $c_v=1$, $\rho_0=1$, $c_s=1$ and $c_h=0$. The calculations are carried out with the 
new HTC schemes on a grid composed of $1024$ control volumes up to a final time of $t=0.4$. 
In the Navier-Stokes limit of the GPR model, a reference solution can be obtained by the exact solution of the incompressible Navier-Stokes equations for the first problem of Stokes, see e.g. \cite{GPRmodel,SIGPR,Hybrid1,Hybrid2}. 
%, which for velocity component $v_2$ reads 
%\begin{equation}
%\label{eqn.fpstokes} 
%v_2(x,t) = v_0 \, \textnormal{erf}\left( \halb \frac{x}{\sqrt{\nu t}} %\right),  
%\end{equation} 
%with $\nu = \mu / \rho_0$. 
For the solid limit of the GPR model ($\tau_1 \to \infty$), this initial condition leads to two shear waves traveling to the left and right, respectively, with speed $c_s$. A reference solution can be obtained using a classical second order MUSCL-Hancock scheme \cite{toro-book} on a fine mesh of 32000 cells. We stress that for all cases with $\mu > 0$ the HTC scheme has been run without any numerical viscosity, i.e. setting $\epsilon=0$.   
The comparison between the numerical solutions obtained with the new HTC schemes and the aforementioned reference solutions is presented in Fig. \ref{fig.shear}, where one can observe an excellent agreement for all cases. 

\begin{figure}[!htbp]
	\begin{center}
		\begin{tabular}{cc} 
			\includegraphics[width=0.38\textwidth]{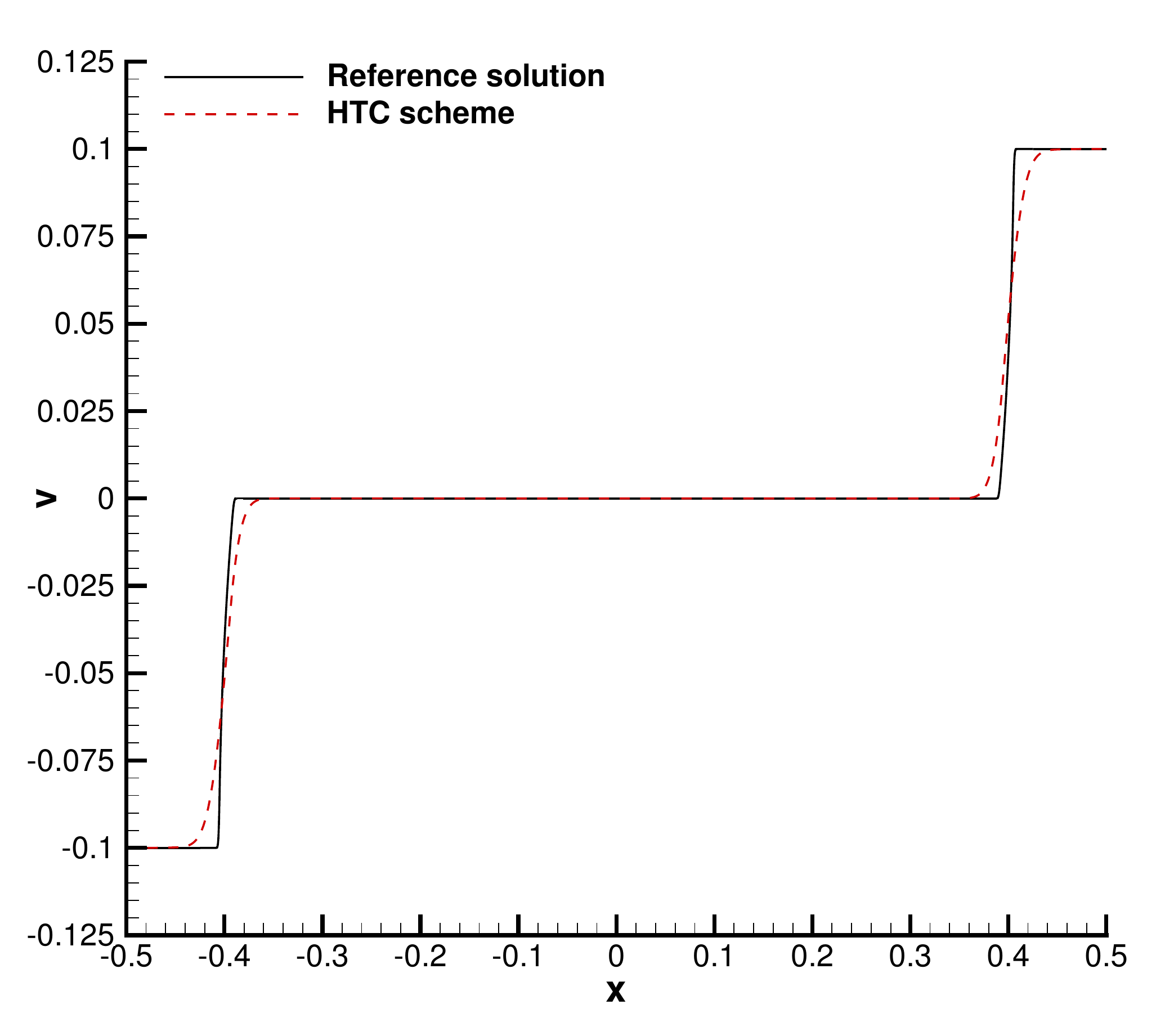}   & 
			\includegraphics[width=0.38\textwidth]{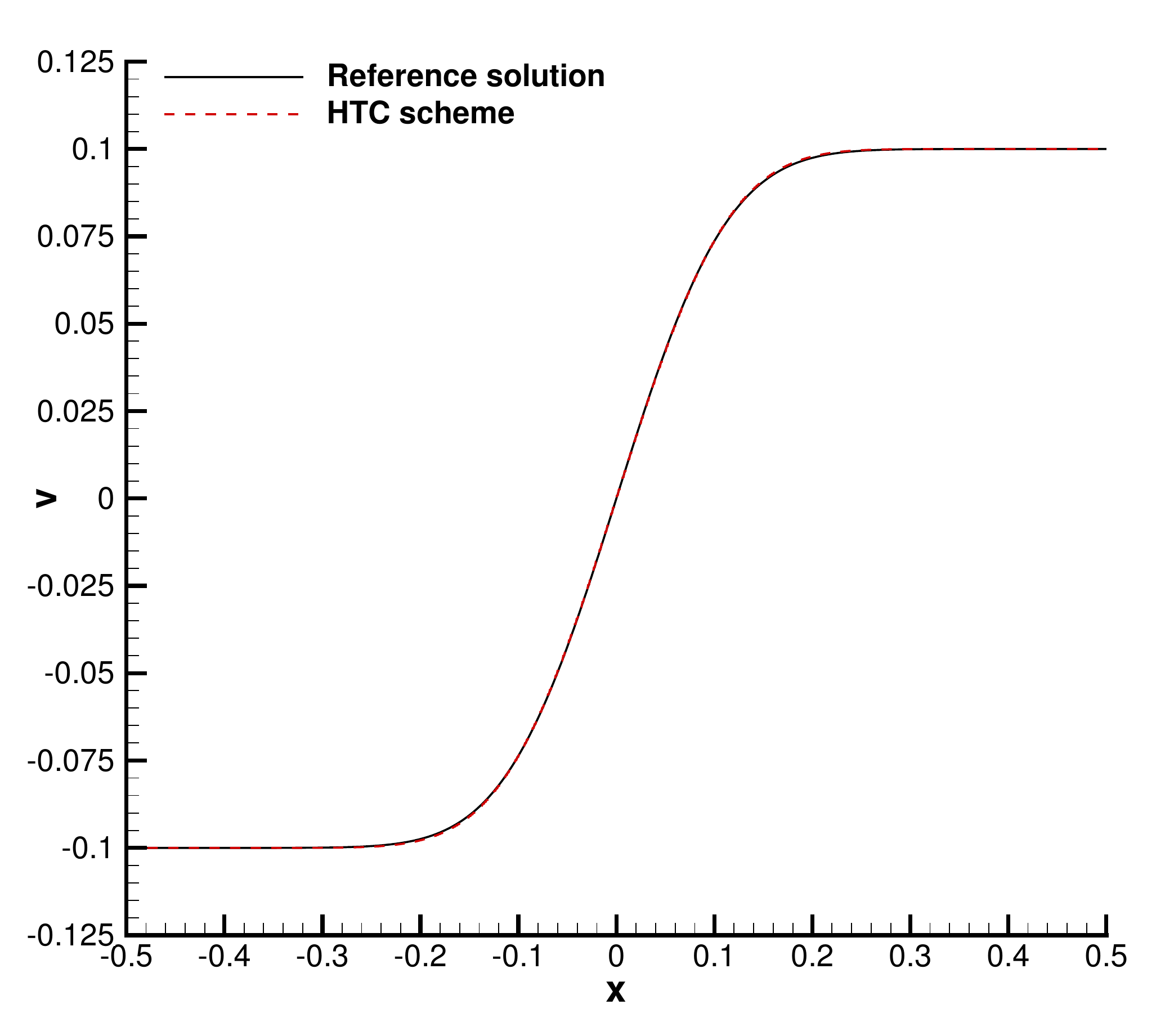}  \\  
			\includegraphics[width=0.38\textwidth]{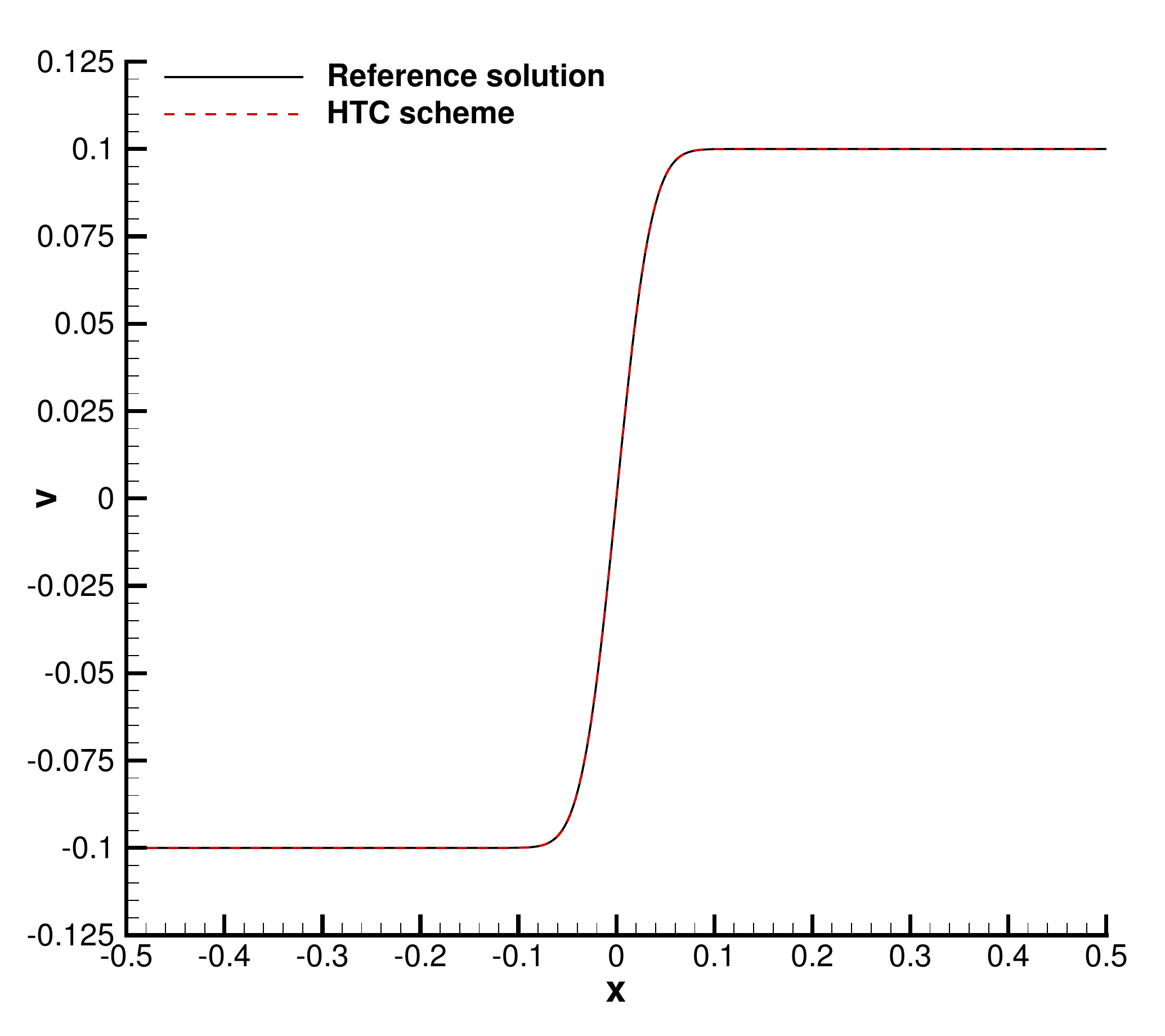}   & 
			\includegraphics[width=0.38\textwidth]{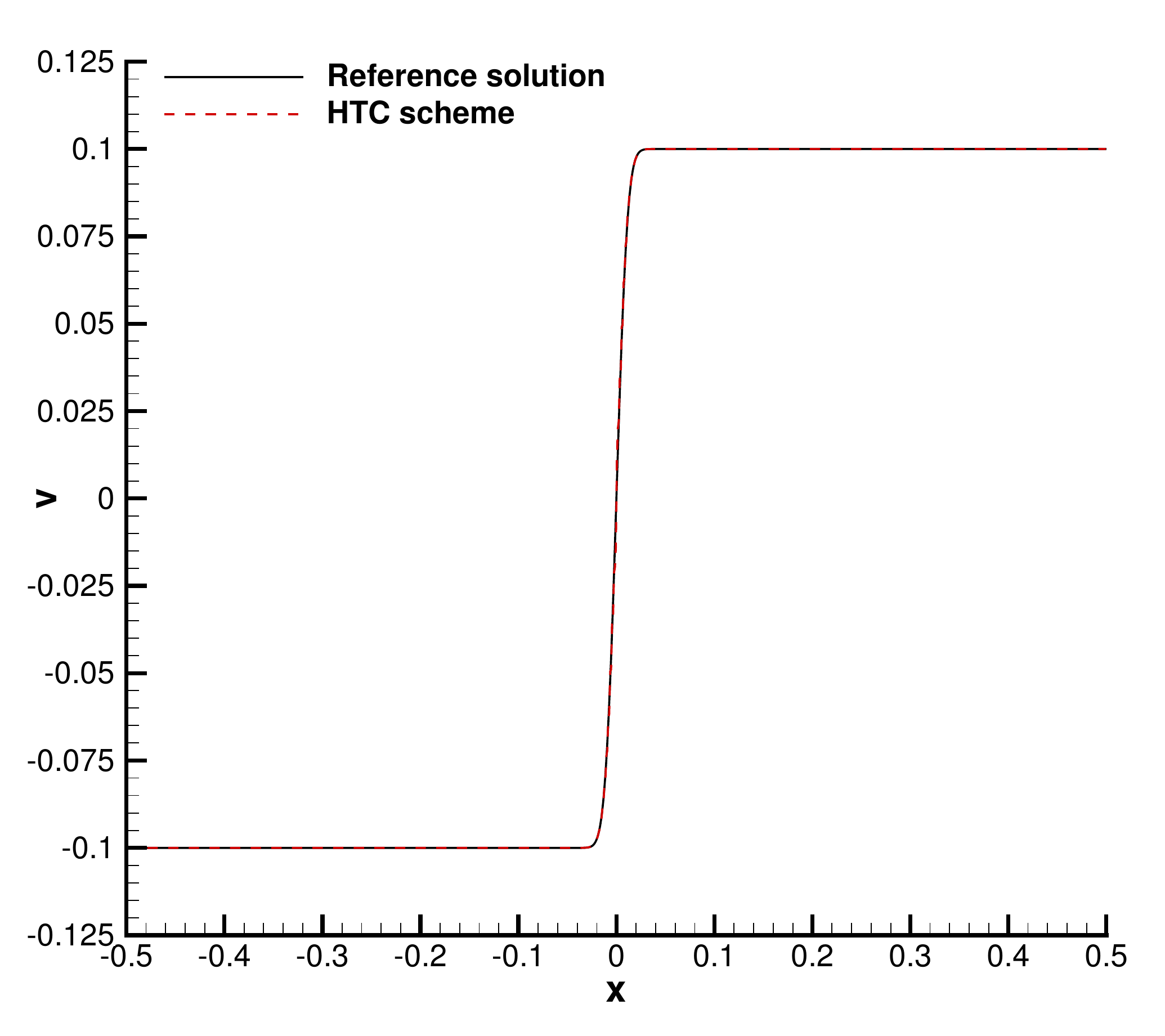}  
		\end{tabular} 
	    \vspace*{-4mm}
		\caption{Numerical solution at time $t=0.4$ obtained with the new thermodynamically compatible HTC schemes for the GPR model applied to a simple shear flow in fluids and in an elastic solid. Results for the solid (top left) and for fluids with different viscosities: $\mu=10^{-2}$ (top right), $\mu=10^{-3}$ (bottom left) and $\mu=10^{-4}$ (bottom right). 
		For fluids, this test corresponds to the first problem of Stokes, which has an exact analytical solution.  } 
		\label{fig.shear}
	\end{center}
\end{figure}

\subsection{Riemann problems} 
\label{sec.rp} 

In this section, we solve a set of Riemann problems with initial data according to Table \ref{tab.ic.rp}, for both the Euler equations of compressible gasdynamics, which are a subset of the GPR model (black terms in \eqref{eqn.GPR}), and for the full GPR model in both its fluid and solid limit. The initial discontinuity is located in $x_c$. For the Euler equations, we consider semi-discrete as well as fully-discrete schemes and the exact solution of the Riemann problem has been provided in \cite{toro-book}, while for the GPR model we consider two types of completely independent numerical reference solutions. The first reference solution is obtained by using a classical MUSCL-Hancock finite volume scheme on a fine mesh of 128000 elements, discretizing the total energy conservation law \eqref{eqn.energy} instead of the entropy inequality \eqref{eqn.entropy}. An alternative reference solution is obtained by solving the GPR model \eqref{eqn.conti}-\eqref{eqn.entropy} with the entropy inequality in its vanishing viscosity limit, using a  fourth order ADER-DG scheme on a fine mesh composed of 14400 order elements, including also the quadratic entropy production term in \eqref{eqn.entropy}. In this case, thermodynamic compatibility is achieved simply at the aid of a fully resolved simulation employing sufficiently fine meshes in combination with high order of accuracy in space and time, see \cite{SWETurbulence}. The numerical results obtained with the semi-discrete and fully-discrete HTC schemes for the compressible Euler equations are shown in Figs. \ref{fig.RP12} and \ref{fig.RP1s}, while the numerical results obtained with the semi-discrete HTC scheme applied to the fluid and solid limits of the GPR model are presented in Fig. \ref{fig.RP3} and \ref{fig.RP4}, respectively, together with the reference solution obtained with the MUSCL-Hancock scheme solving the energy conservation law \eqref{eqn.energy}, as well as the reference solution obtained with the high order ADER-DG scheme applied to the viscous system \eqref{eqn.conti}-\eqref{eqn.entropy}. \textcolor{black}{The effective mesh resolution is provided for each test case in the corresponding figure caption.} In all cases we can note an excellent agreement between the numerical solution obtained with the new HTC schemes forwarded in this paper and the available exact or numerical reference solutions. 

\textcolor{black}{Test problem RP1s was proposed by Toro in \cite{toro-book} and includes a sonic rarefaction. Simulations are carried out on several meshes and the obtained quantities $\rho$, $p$, $u=v_1$ and $S$ are shown in Fig. \ref{fig.RP1s}. We observe that the thermodynamically compatible schemes proposed in this paper do \textit{not} exhibit any sonic glitch, compared to other Godunov-type finite volume schemes, see \cite{toro-book}.}

\textcolor{black}{A quantitative study concerning the influence of the number of Gauss-Legendre quadrature nodes $n_{GP}$ and the chosen time discretization on the total energy conservation error can be found for a smoothed version of RP1 with initial data $\q(x,0)=\halb (\q_L + \q_R) + \halb (\q_R - \q_L) \, \textnormal{erf}(x/\chi)$ with $\chi=0.01$ in Table \ref{tab.energy.error.analysis}. As expected, the conservation error of the semi-discrete schemes is dominated by the time discretization and the chosen time step size (CFL number), while the energy conservation error of the fully discrete scheme is independent of the time step size and is dominated only by the numerical quadrature rule used in \eqref{eqn.p.scheme.diss}.} 

\begin{table}[!htbp]
	\renewcommand{\arraystretch}{1.05}
	\caption{Initial states left (L) and right (R) for density $\rho$, velocity $\mathbf{v}=(u,v,0)$ and pressure $p$  
		for a set of Riemann problems solved on the domain $\Omega=[-\frac{1}{2},+\frac{1}{2}]$ using the new HTC schemes. 		
		The Riemann problems include the pure Euler equations (RP1, RP2 \textcolor{black}{and RP1s}), as well as the fluid and solid limit of the GPR model (RP3 and RP4). 
		For the GPR model (RP3 and RP4) we initialize $\AAA$ and $\mathbf{J}$ as 
		$\mathbf{A}=\sqrt[3]{\rho} \, \mathbf{I}$ and  $\mathbf{J}=\mathbf{0}$ and set $c_s=c_h=1$. 
		The relaxation times have been chosen as $\tau_1  = \tau_2 = 2 \cdot 10^{-5}$ for RP3 and $\tau_1 = \tau_2 = 10^{20}$ for RP4. In all cases we set $\gamma=1.4$. 
	\vspace{-1mm}} 
	%\ip{Couldn't find $ c_s $ for RP3 and RP4. Is it $ c_s=1 $?}} 
	\begin{center} 
		\begin{tabular}{ccccccccc} 
			\hline
			RP & $\rho_L$ & $u_L$ & $v_L$ & $p_L$ & $\rho_R$ & $u_R$ & $v_R$ & $p_R$  \\ 
			\hline
			RP1 &  1.0      &  0.0       &  0.0 & 1.0     & 0.125      &  0.0        &  0.0 & 0.1      \\
			\textcolor{black}{RP1s} &  1.0      &  \textcolor{black}{0.75}       &  0.0 & 1.0     & 0.125      &  0.0        &  0.0 & 0.1       \\
			RP2 &  5.99924  & 19.5975    & 0.0  & 460.894 & 5.99242    & -6.19633    &  0.0 & 46.095   \\ 
			RP3 &  1.0      &  0.0       & -0.2 & 1.0     & 0.5        &  0.0        & +0.2 & 0.5      \\ 
			RP4 &  1.0      &  0.0       & -0.2 & 1.0     & 0.5        &  0.0        & +0.2 & 0.5      \\ 
			\hline
		\end{tabular}
	\end{center} 
	\label{tab.ic.rp}
\end{table}
 
\begin{table}[!htbp]
	\renewcommand{\arraystretch}{1.05}
	\textcolor{black}{
	\caption{Total energy conservation error depending on the time discretization and the number of Gauss-Legendre quadrature points $n_{GP}$ for the calculation of the thermodynamically compatible flux \eqref{eqn.p.scheme.diss}. \vspace{-3mm} } 
	\begin{center} 
		\begin{tabular}{llllll} 
			\hline
			CFL & 0.5 & 0.4 & 0.3 & 0.2 & 0.1   \\ 
			\hline
			\multicolumn{6}{c}{semi-discrete HTC scheme + TVD Runge-Kutta $\mathcal{O}3$ } \\ 
			\hline			
			$n_{GP} = 3$ & $2.90 \cdot 10^{-5}$ & $1.55 \cdot 10^{-5}$ &$6.30 \cdot 10^{-6}$ &$2.00 \cdot 10^{-6}$ &$3.00 \cdot 10^{-7}$    \\ 
			$n_{GP} = 5$ & $2.90 \cdot 10^{-5}$ & $1.54 \cdot 10^{-5}$ &$6.31 \cdot 10^{-5}$ &$1.99 \cdot 10^{-5}$ &$2.45 \cdot 10^{-7}$    \\ 
			\hline
			\multicolumn{6}{c}{semi-discrete HTC scheme  + classical Runge-Kutta $\mathcal{O}4$ } \\ 
			\hline			
			$n_{GP} = 3$ & $2.23 \cdot 10^{-6}$ & $9.51 \cdot 10^{-7}$ &$3.07 \cdot 10^{-7}$ &$5.25 \cdot 10^{-8}$ &$8.33 \cdot 10^{-9}$    \\ 
			$n_{GP} = 5$ & $2.24 \cdot 10^{-6}$ & $9.64 \cdot 10^{-7}$ &$3.19 \cdot 10^{-7}$ &$6.53 \cdot 10^{-8}$ &$4.43 \cdot 10^{-9}$    \\ 
			\hline
			\multicolumn{6}{c}{Fully-discrete HTC scheme } \\ 			
			\hline
			$n_{GP} = 3$ & $1.80 \cdot 10^{-9}$ & $1.81 \cdot 10^{-9}$ &$1.82 \cdot 10^{-9}$ &$1.83 \cdot 10^{-9}$ &$1.84 \cdot 10^{-9}$    \\ 
			$n_{GP} = 5$ & $2.70 \cdot 10^{-13}$ & $2.70 \cdot 10^{-13}$ & $2.70 \cdot 10^{-13}$ & $2.70 \cdot 10^{-13}$ & $2.70 \cdot 10^{-13}$   \\ 
			\hline			
		\end{tabular}
	\end{center} 
}
	\label{tab.energy.error.analysis}
\end{table}

\begin{figure}[!htbp]
	\begin{center}
			\includegraphics[trim=10 10 10 10,clip,width=0.45\textwidth]{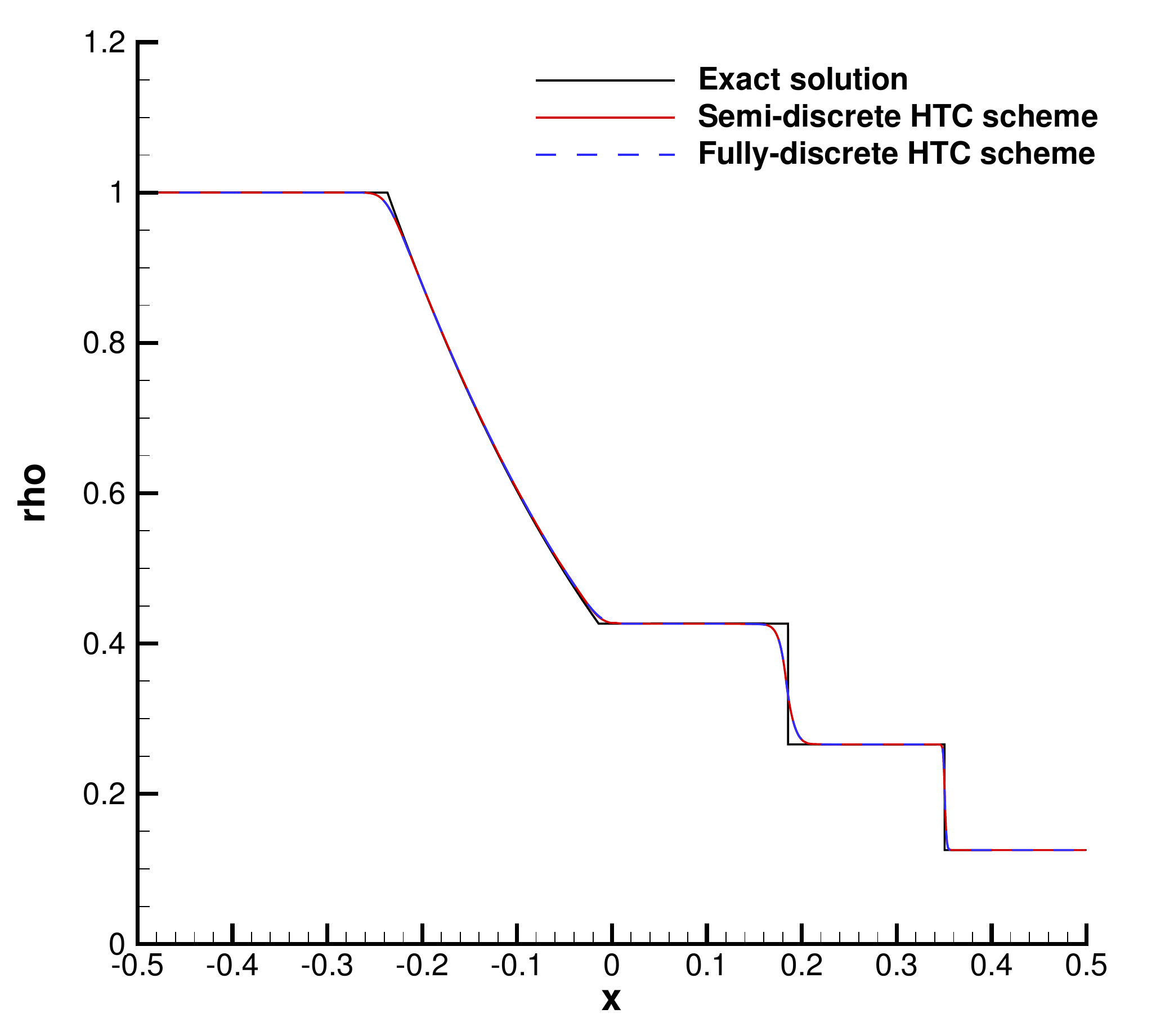}   
			\includegraphics[trim=10 10 10 10,clip,width=0.45\textwidth]{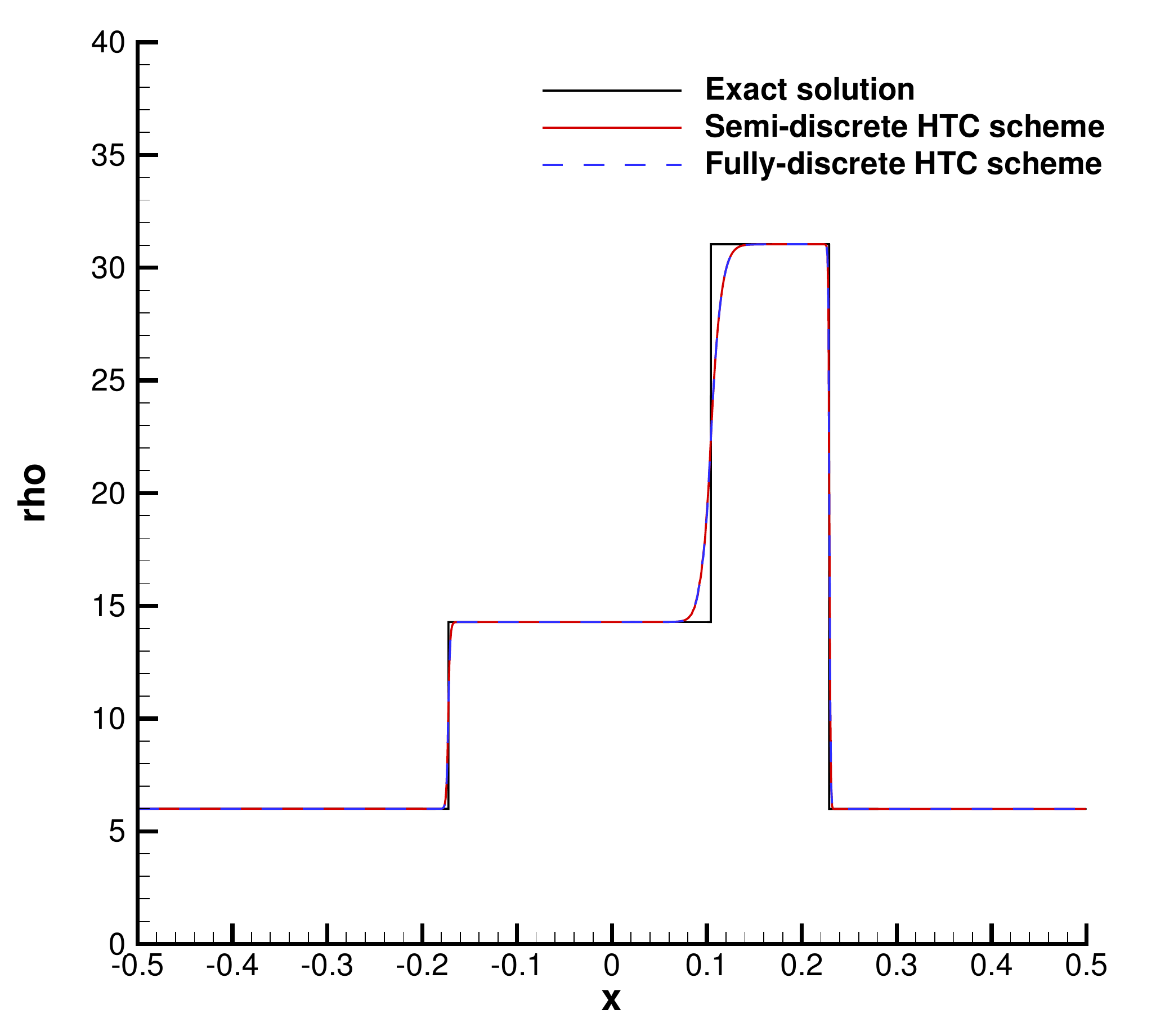}    
		\caption{Numerical results for Riemann problems RP1 ($x_c=0$) and RP2 ($x_c=-0.2$) at times $t=0.2$ and $t=0.035$, respectively, obtained with the semi-discrete (red solid line) and the fully-discrete (dashed blue line) HTC schemes on 1024 elements applied to the compressible Euler equations. The exact solution of the compressible Euler equations is represented by the black solid line. }  
		\label{fig.RP12}
	\end{center}
\end{figure}

\begin{figure}[!htbp]
	\begin{center}
		\includegraphics[trim=10 10 10 10,clip,width=0.45\textwidth]{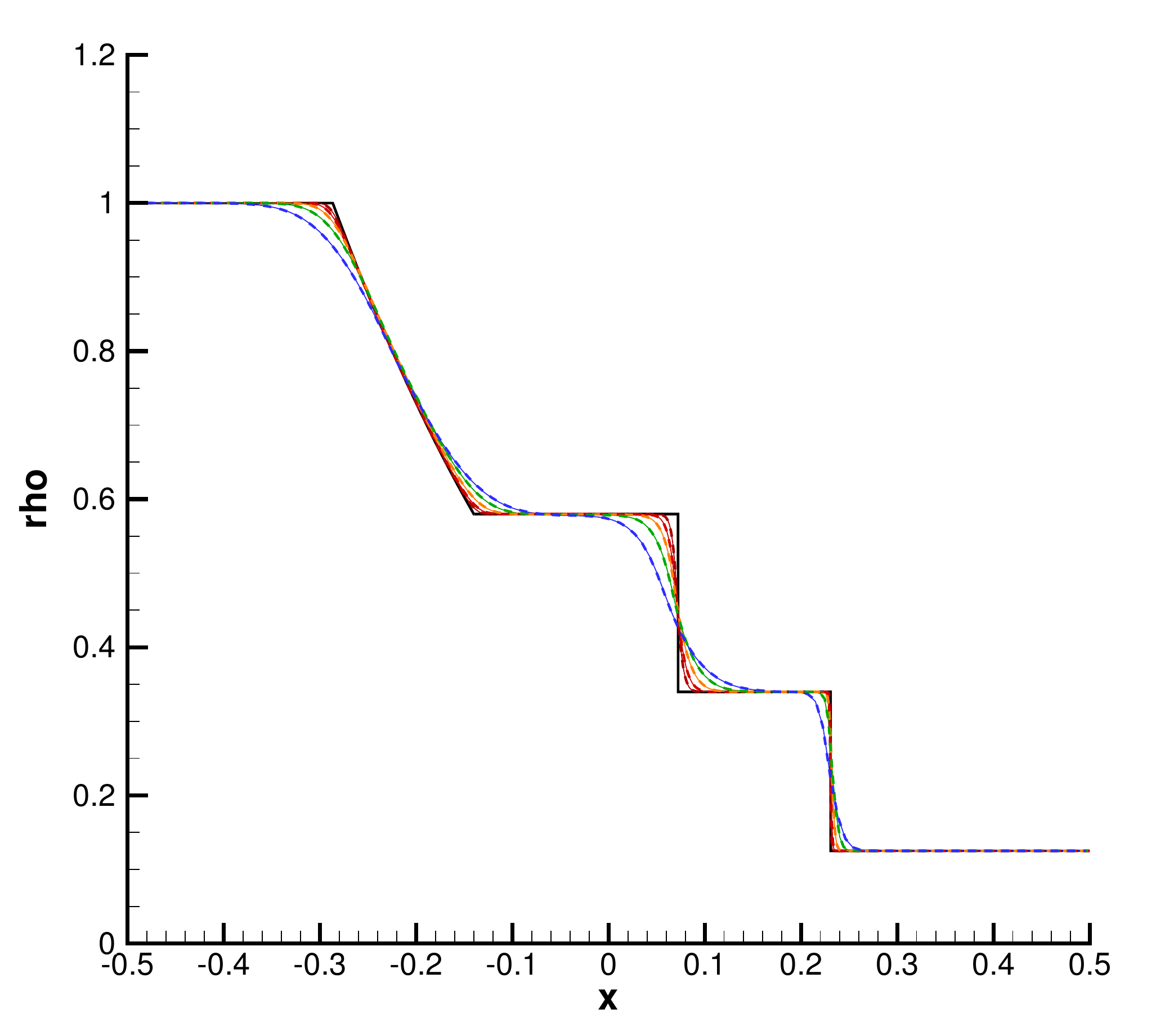}    
		\includegraphics[trim=10 10 10 10,clip,width=0.45\textwidth]{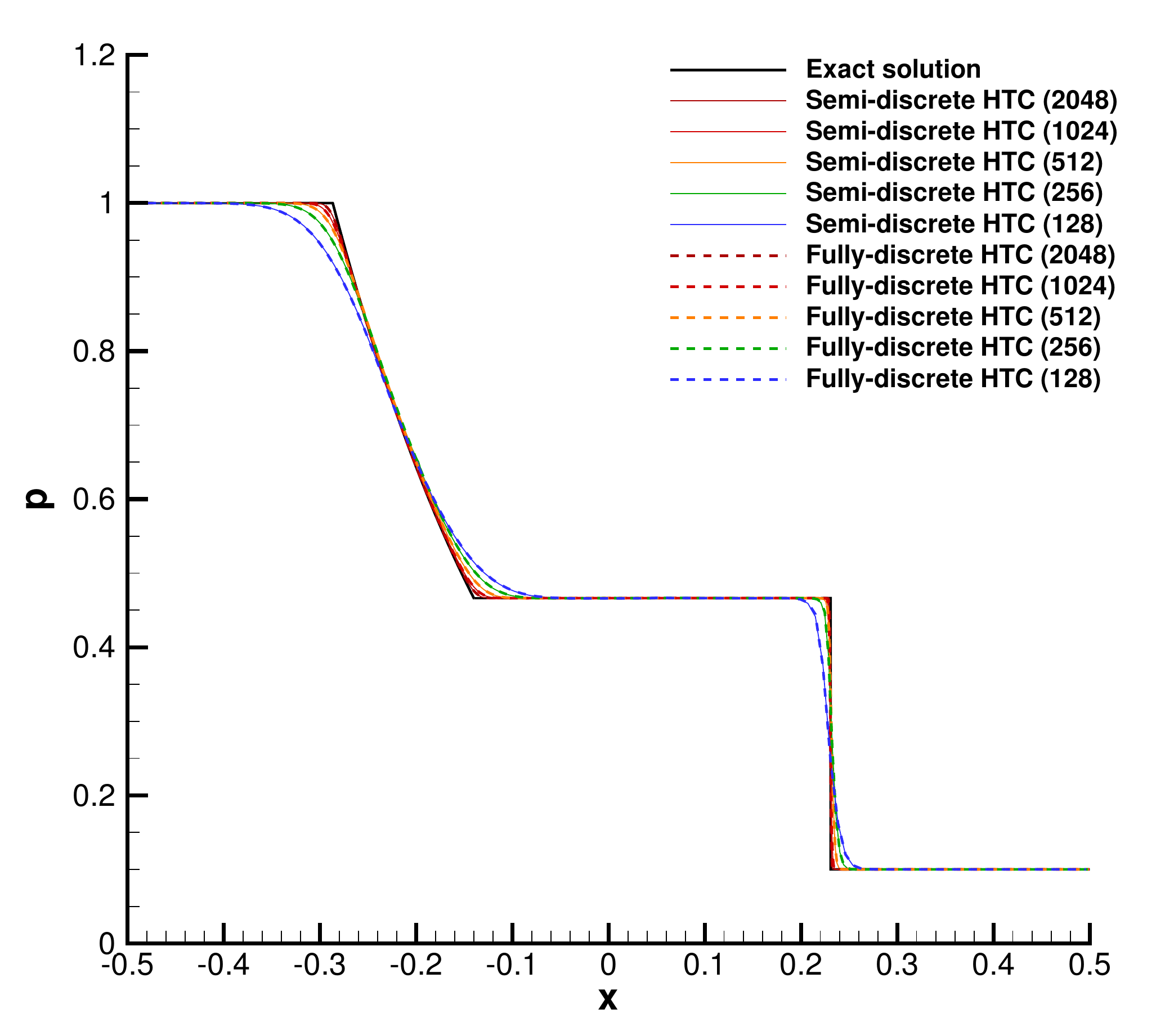}     \\ 
		\includegraphics[trim=10 10 10 10,clip,width=0.45\textwidth]{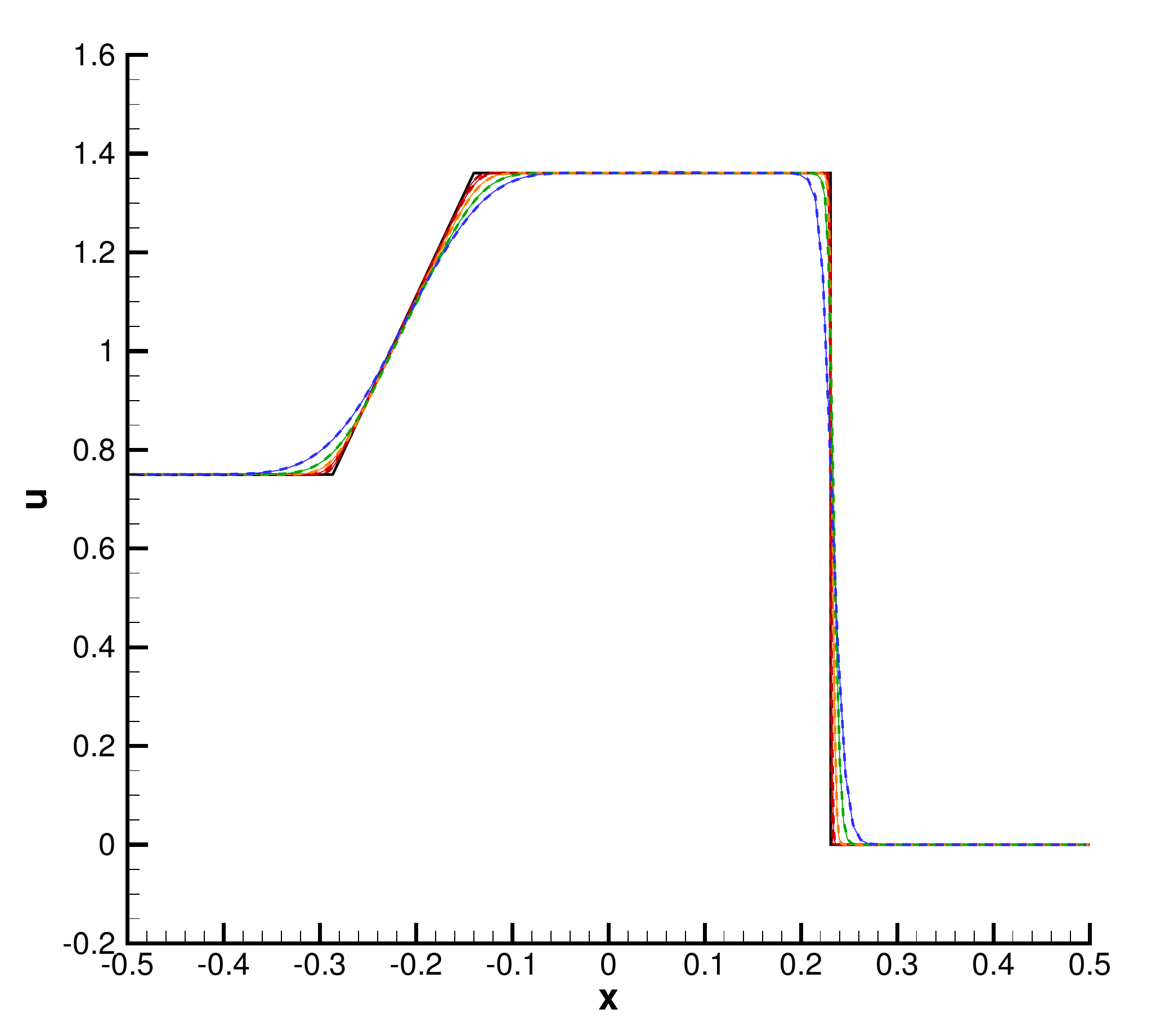}    
		\includegraphics[trim=10 10 10 10,clip,width=0.45\textwidth]{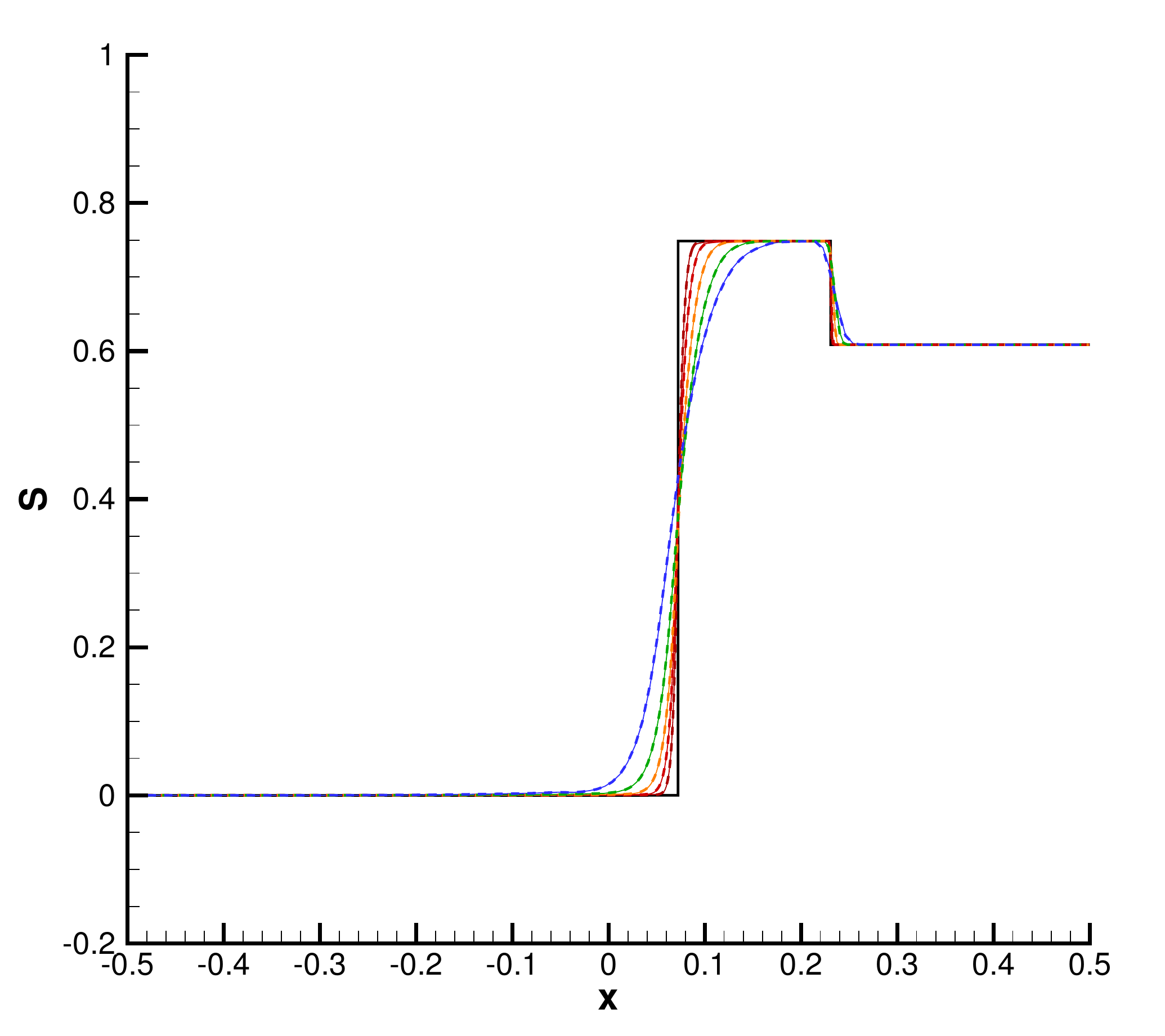}      
		\caption{\textcolor{black}{Numerical results for Riemann problem RP1s ($x_c=-0.2$) at time $t=0.2$ obtained with the semi-discrete (solid lines) and the fully-discrete (dashed lines) HTC schemes on 2048, 1024, 512, 256 and 128 elements applied to the compressible Euler equations. The exact solution of the compressible Euler equations is represented by the black solid line.} }  
		\label{fig.RP1s}
	\end{center}
\end{figure}

\begin{figure}[!htbp]
	\begin{center}
			\includegraphics[trim=10 10 10 10,clip,width=0.45\textwidth]{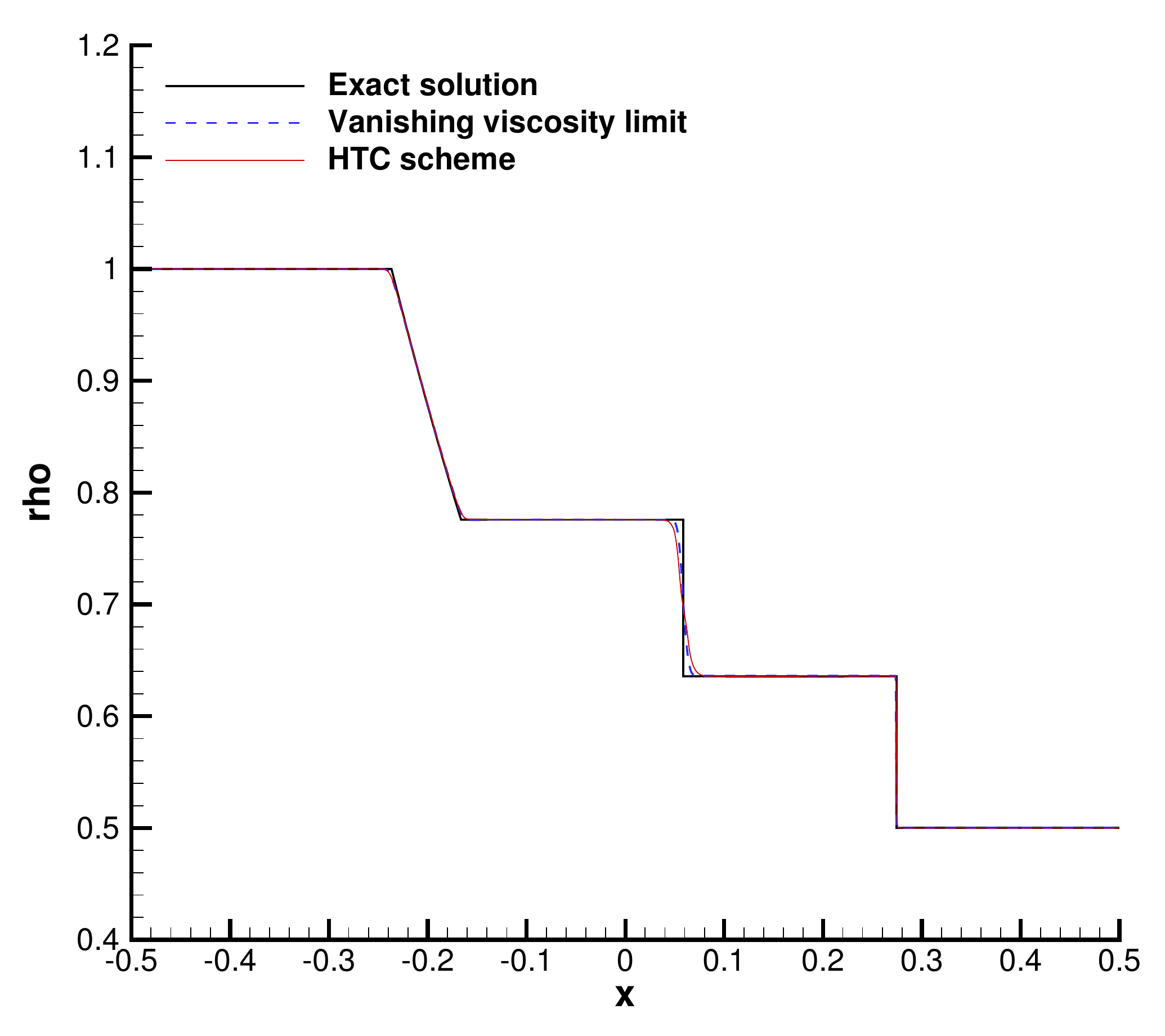}    
			\includegraphics[trim=10 10 10 10,clip,width=0.45\textwidth]{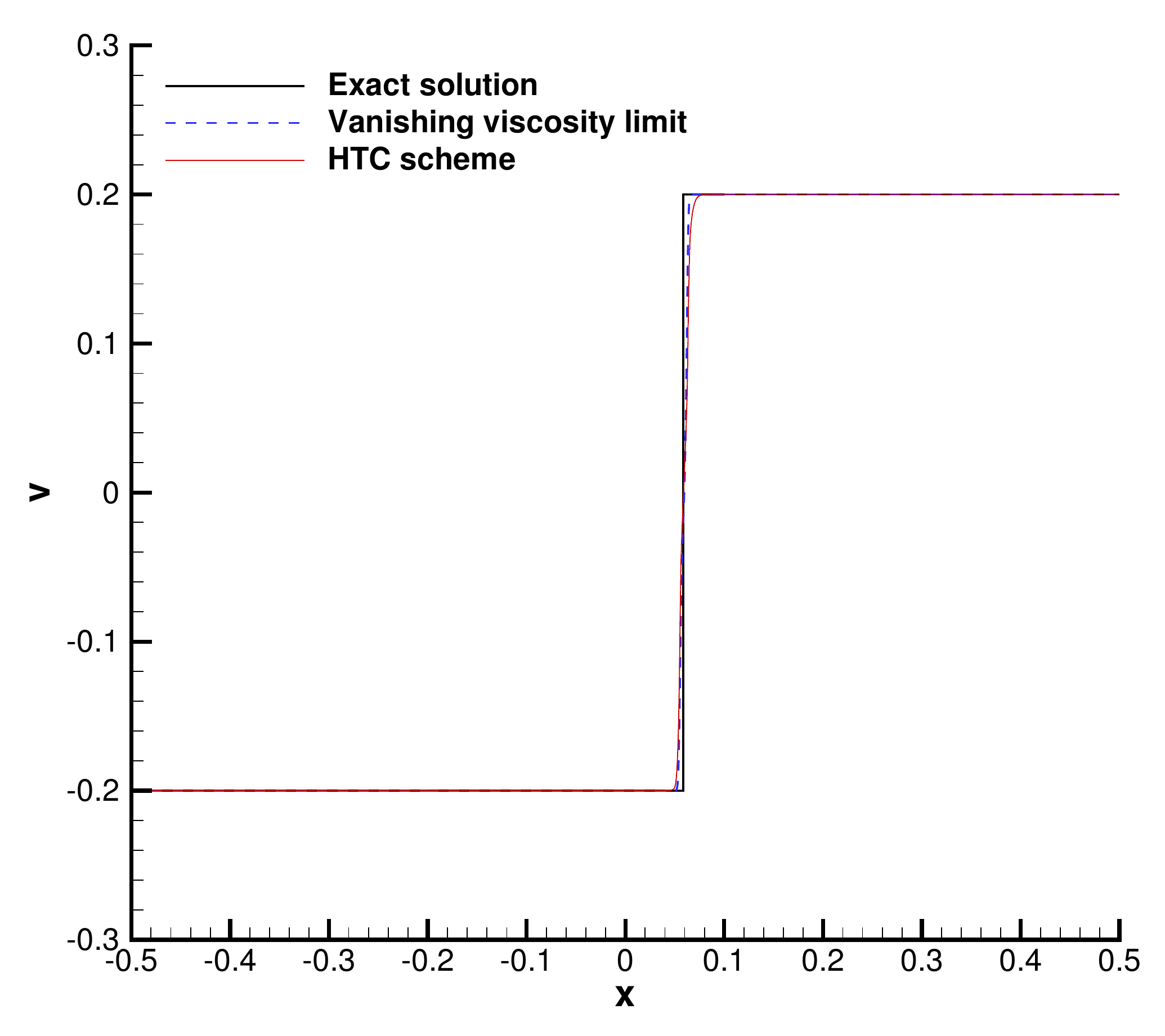}    
		\caption{Numerical results at time $t=0.2$ for Riemann problem RP3 ($x_c=0$) obtained with the HTC scheme (red solid line) on 1024 elements, the fourth order ADER-DG scheme applied to the vanishing viscosity limit of the viscous equations \eqref{eqn.conti}-\eqref{eqn.entropy} using $\epsilon = 2 \cdot 10^{-5}$ on 14400 elements (dashed blue line) and the exact solution of the compressible Euler equations (black solid line).   } 
		\label{fig.RP3}
	\end{center}
\end{figure}

\begin{figure}[!htbp]
	\begin{center}
			\includegraphics[trim=10 10 10 10,clip,width=0.45\textwidth]{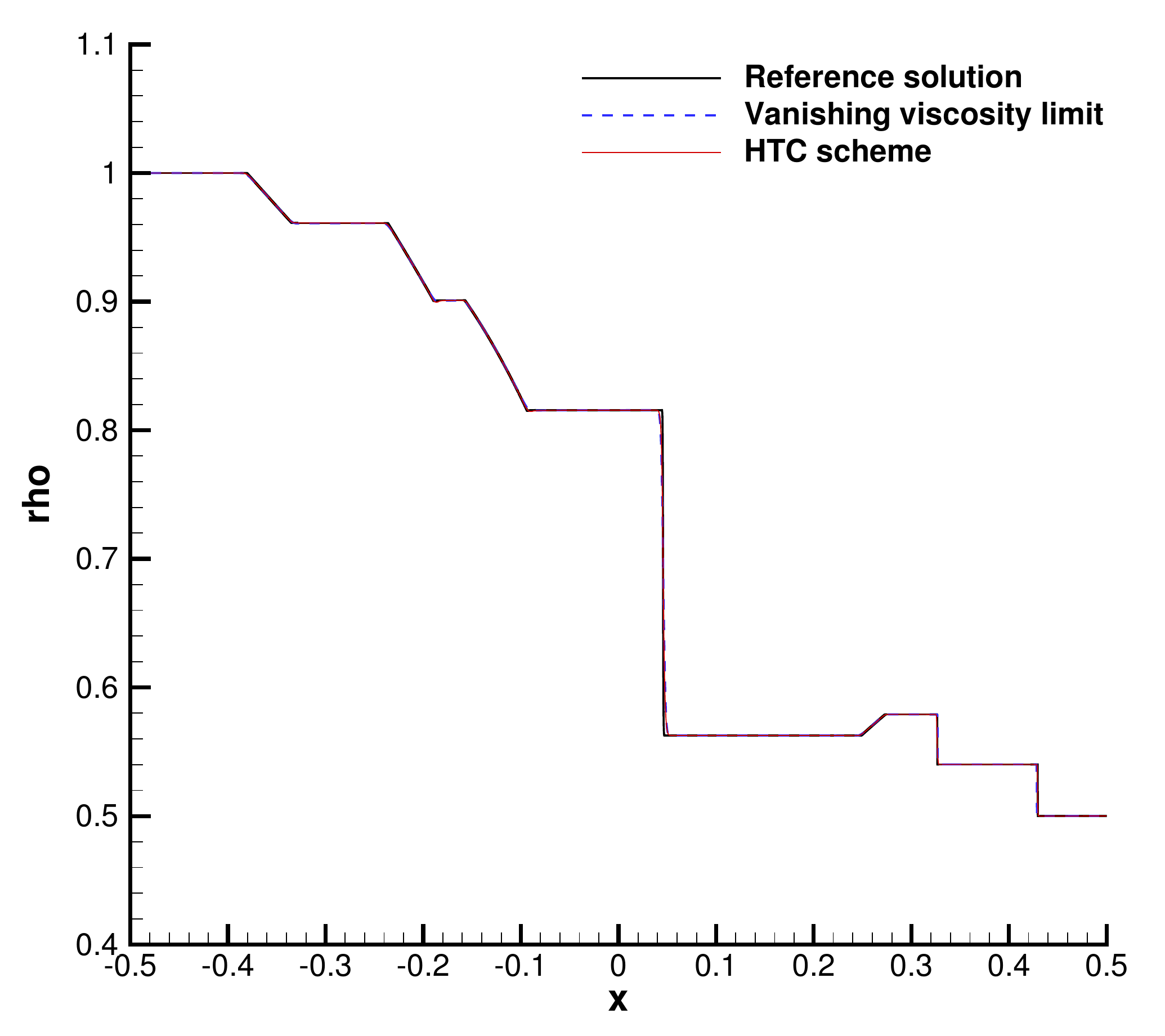}    
			\includegraphics[trim=10 10 10 10,clip,width=0.45\textwidth]{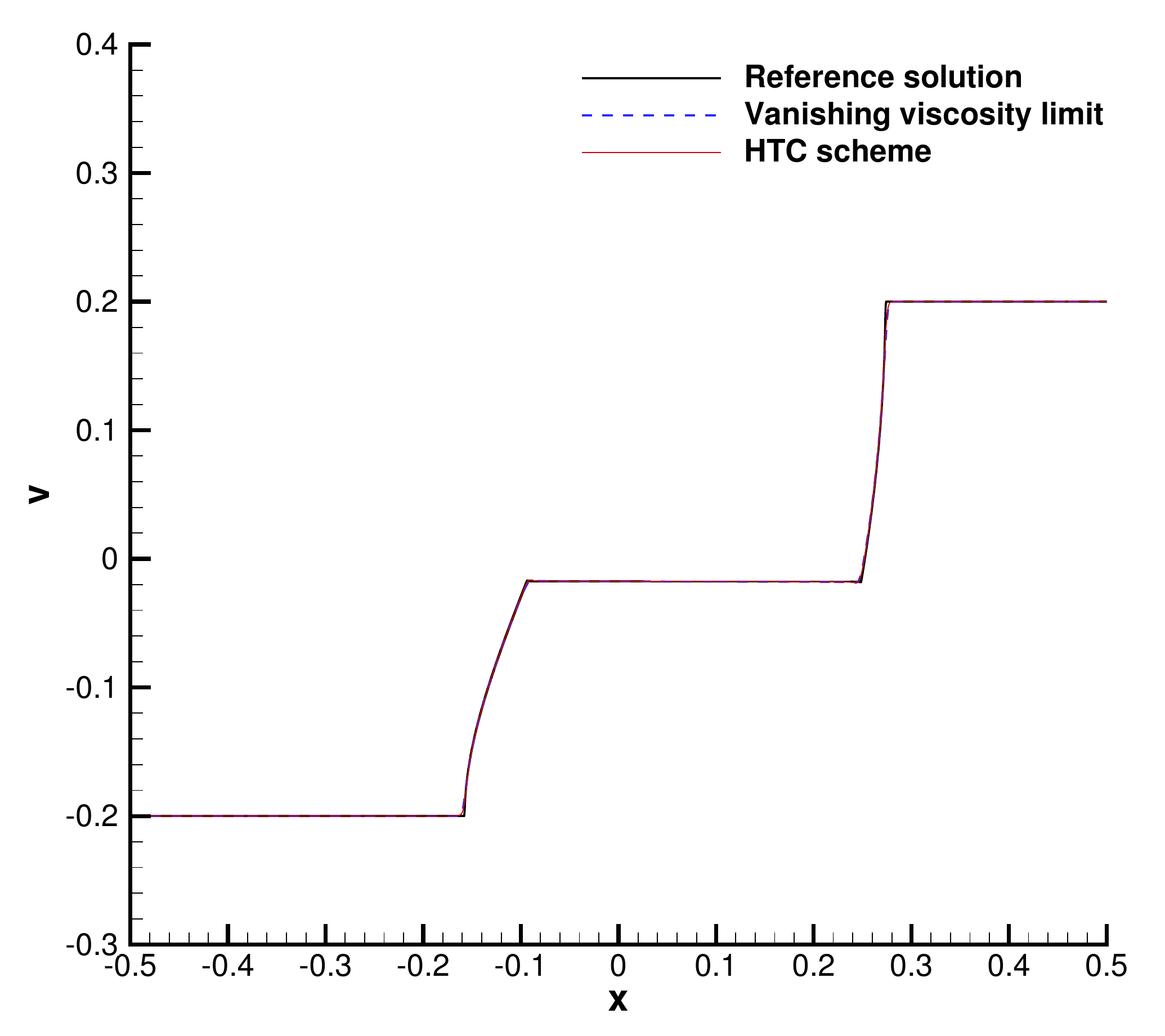}  \\  
			\includegraphics[trim=10 10 10 10,clip,width=0.45\textwidth]{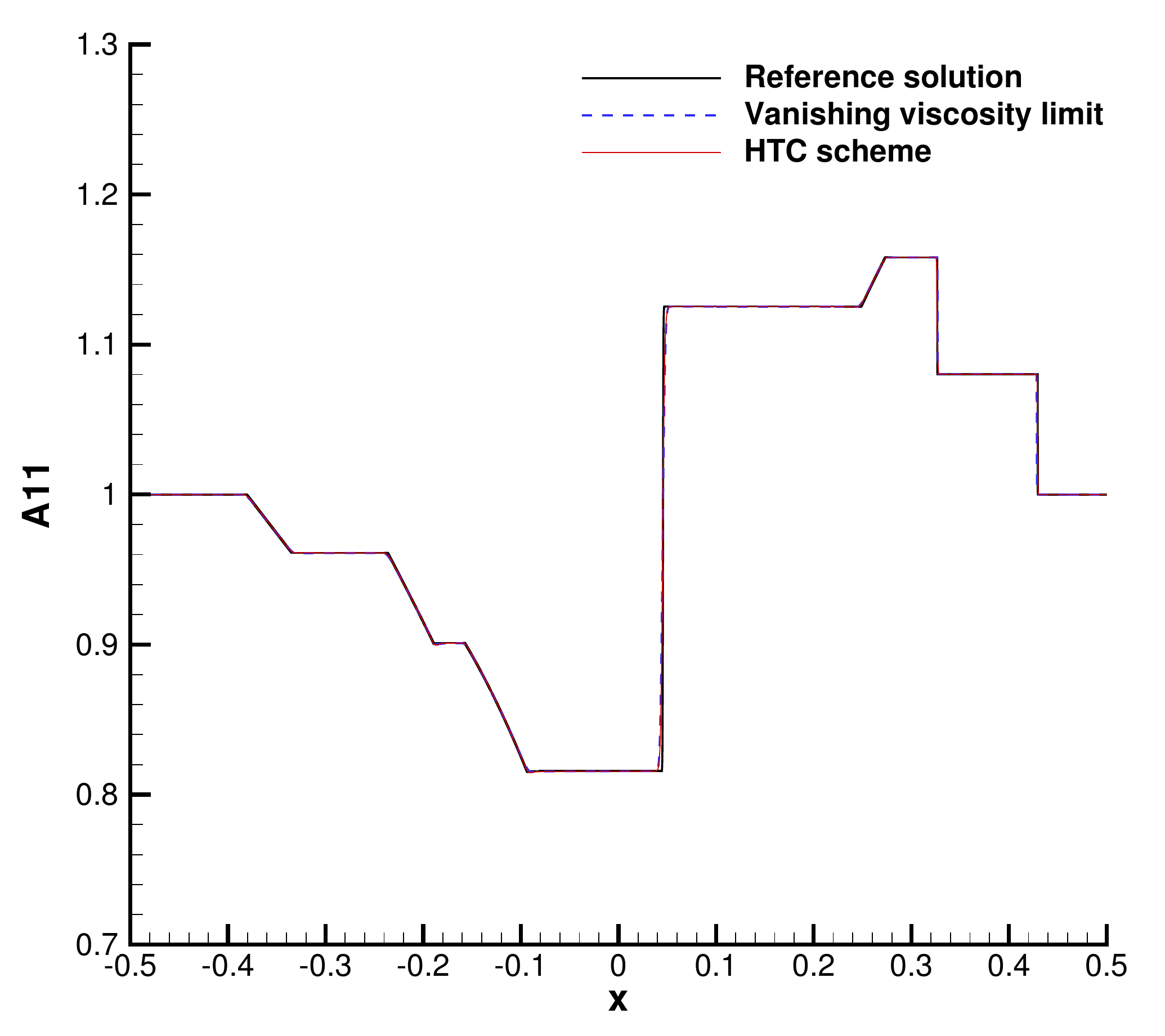}    
			\includegraphics[trim=10 10 10 10,clip,width=0.45\textwidth]{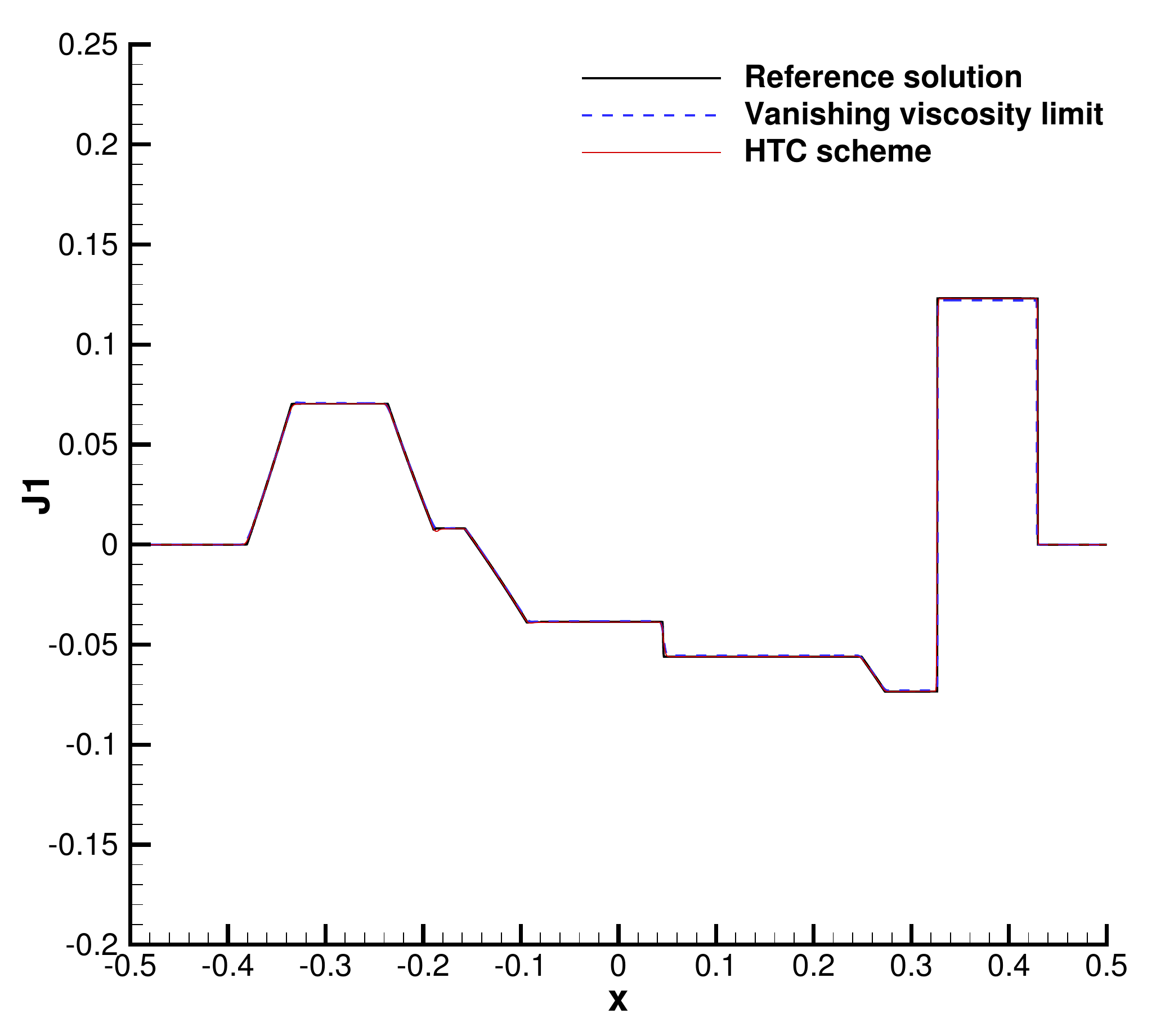}  
		\caption{Numerical results at time $t=0.2$ for Riemann problem RP4 ($x_c=0$) obtained with the HTC 
		scheme (red solid line) on 10000 control volumes, a fourth order ADER-DG scheme applied to 
		the vanishing viscosity limit of the viscous equations 
		\eqref{eqn.conti}-\eqref{eqn.entropy} using $\epsilon = 2 \cdot 10^{-5}$ (dashed blue line) 
		on 144000 elements and the reference solution obtained with a MUSCL-Hancock scheme applied 
		to the model with the energy conservation law \eqref{eqn.energy} instead of the entropy 
		inequality \eqref{eqn.entropy} (black solid line) using 128000 elements.   } 
		\vspace*{-4mm}
		\label{fig.RP4}
	\end{center}
\end{figure}

\subsection{Viscous shock wave} 
\label{sec.viscshock} 

Consider a stationary viscous shock wave at a shock Mach number of $M_s=2$. 
For Prandtl number Pr$=0.75$ there exists an exact solution of the 
compressible Navier-Stokes equations, see \cite{Becker1923,GPRmodel}. 
The computational domain $\Omega=[-0.5,+0.5]$ is covered with $1024$ control volumes and the shock 
wave is centered at $x=0$. We assume that the fluid is moving into the shock wave from right to left. 
The data in front of the shock are $\rho_0 =1$, $v^0_1=-2$, $v^0_2=v_3=0$ and $p^0=1/\gamma$
so that the associated sound speed is $c^0 = 1$ and the corresponding Reynolds number based on a 
reference length $L=1$ is given by $Re_s= \rho^0 \, c^0 \, M_s \, L \, \mu^{-1}$. 
The parameters are set as $\gamma = 1.4$, $c_v = 2.5$, $c_h=c_s=50$, $\mu=2 
\cdot 10^{-2}$ and $\lambda = 9 \frac{1}{3} \cdot 10^{-2}$, hence the shock Reynolds number is 
$Re_s=100$. 
At $t=0$ we set $\AAA=\sqrt[3]{\rho} \, \mathbf{I}$ and $\mathbf{J}=\mathbf{0}$. 
The comparison between the numerical solution obtained with the semi-discrete HTC scheme 
applied to \eqref{eqn.GPR} and the exact solution of the compressible Navier-Stokes equations 
is shown in Fig. \ref{fig.vshock}. For all quantities an excellent agreement is achieved. 

\begin{figure}[!htbp]
	\begin{center}
		\includegraphics[trim=5 10 10 10,clip,width=0.32\textwidth]{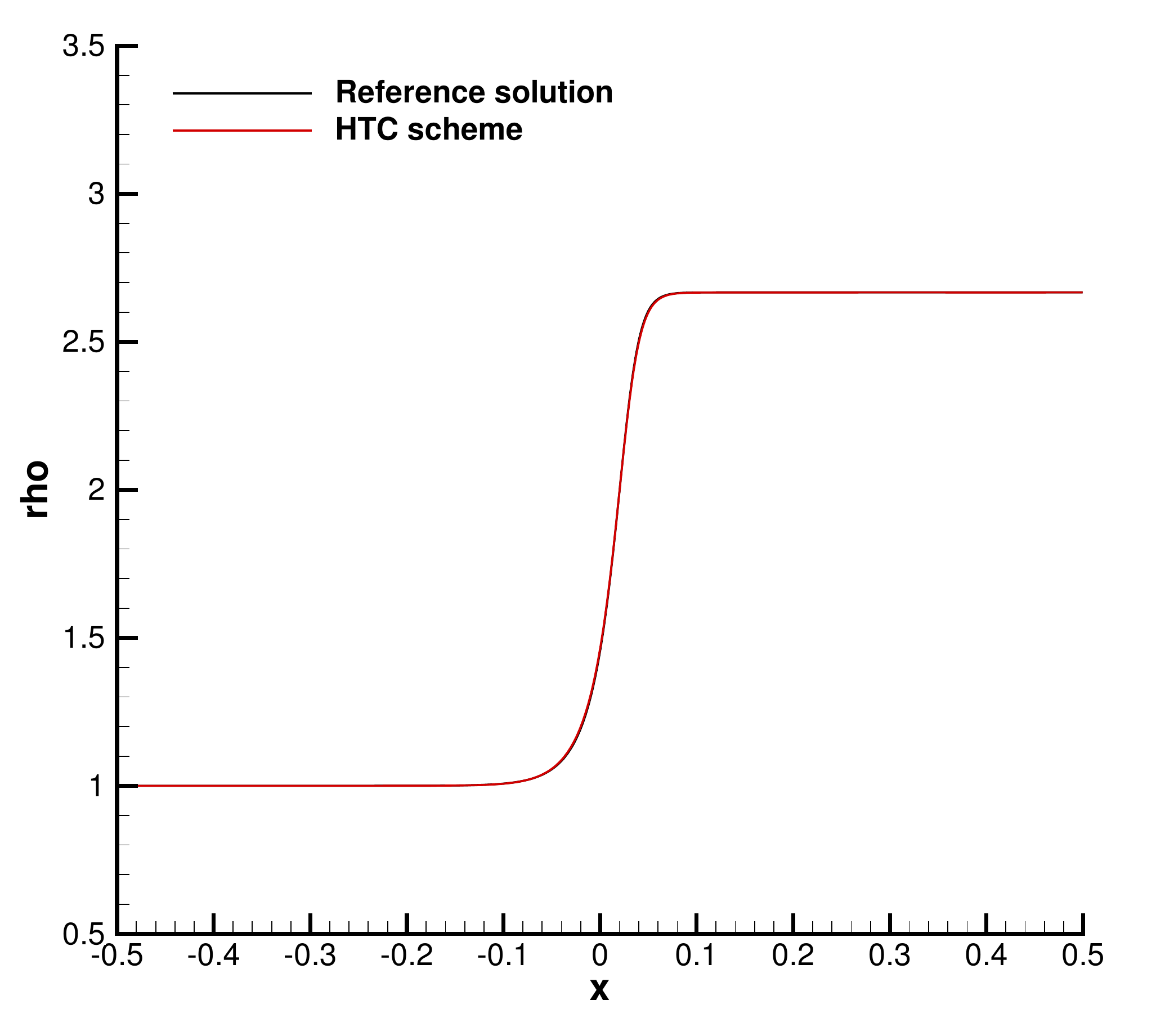}   
		\includegraphics[trim=5 10 10 10,clip,width=0.32\textwidth]{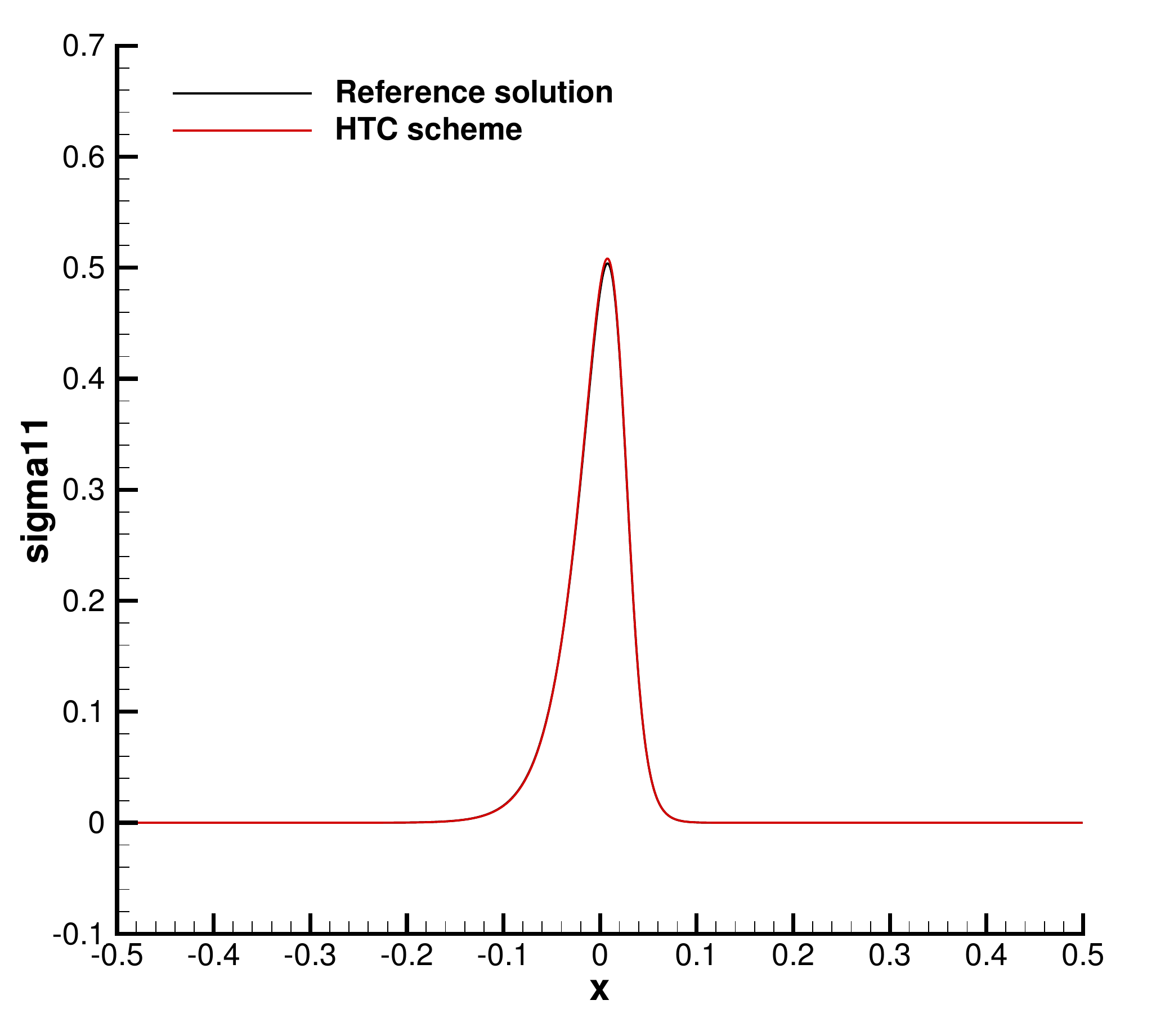}   
		\includegraphics[trim=5 10 10 10,clip,width=0.32\textwidth]{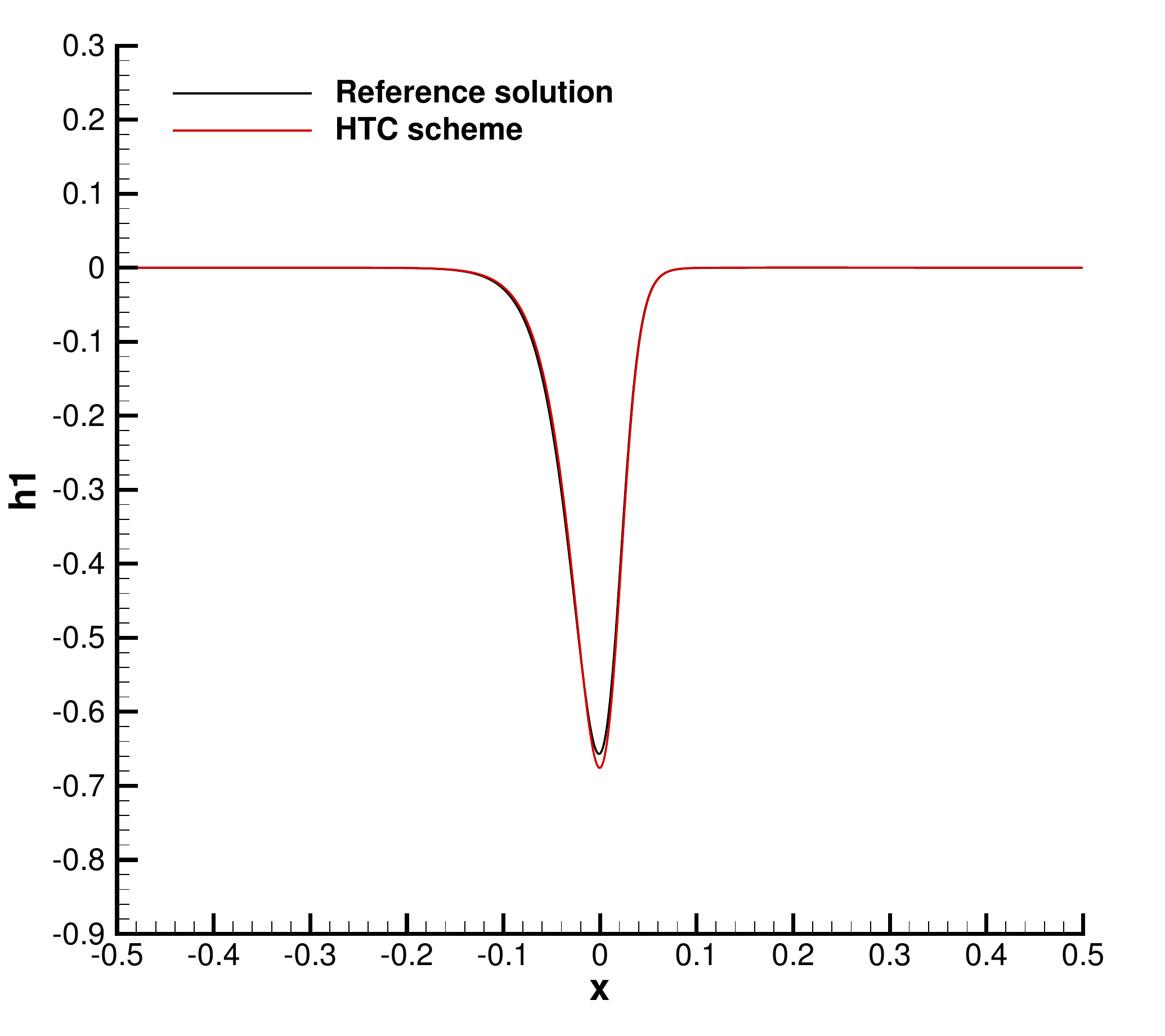} 
		\vspace*{-2mm}  
		\caption{Exact solution of the compressible Navier-Stokes equations and numerical solution 
		obtained with the HTC scheme applied to the GPR model for a viscous shock at $M_s=2$, 
		$Re_s=100$ and $Pr=0.75$. Density (left), stress $\sigma_{11}$ (center) and heat flux $h_1$ 
		(right) at time $t=0.25$. } 
		\vspace*{-4mm}
		\label{fig.vshock}
	\end{center}
\end{figure}

\subsection{Solid rotor problem} 
\label{sec.rotor} 

In this section we solve the solid rotor problem proposed in \cite{SIGPR}.  
By setting $\tau_1 = \tau_2 = 10^{20}$ the model \eqref{eqn.GPR} describes a nonlinear hyperelastic solid. The computational domain is $\Omega = [-1,+1]^2$ with periodic 
boundary conditions everywhere. The initial data for density, pressure, $\AAA$ and $\mathbf{J}$ 
is set to $\rho = 1$, $p = 1$, $\mathbf{A} = \mathbf{I}$ and $\mathbf{J}=\mathbf{0}$, while the 
initial 
condition for the velocity field is $v_1 = -y/R$, $v_2 = +x/R$ and $v_3=0$ within the circular 
region $r \leq R$, where $r = \left\| \mathbf{x} \right\|$ and $R=0.2$, while $\mathbf{v}=0$ for 
$r > R$. The parameters of the GPR model are set to $\gamma = 1.4$, $c_s = 1.0$ and 
$c_h = 1.0$. We run the test problem until a final time of $t=0.3$ using the two-dimensional 
semi-discrete HTC scheme for the GPR model on a uniform Cartesian grid composed of $512 \times 512$ 
elements. The artificial viscosity in the HTC scheme is set to a constant value of $\epsilon = 5 
\cdot 10^{-4}$. To obtain a reference solution, on the same mesh of $512 \times 512$ elements we 
solve the same problem again but using a classical second order MUSCL-Hancock scheme, see 
\cite{toro-book} for details. We emphasize that in the MUSCL scheme, which is not thermodynamically 
compatible, we solve the total energy conservation law \eqref{eqn.energy} rather than the entropy 
inequality \eqref{eqn.entropy}, as already suggested in \cite{GPRmodel}. The obtained results are 
compared with each other in Fig.~\ref{fig.solidrotor}, 
where the contour colors of the velocity component $v_1$ are shown. The agreement between the numerical solution obtained with the new HTC scheme and the reference solution is very good. Since the applied 
HTC scheme for this test problem is only compatible with the \textit{semi-discrete} total energy conservation law, we have explicitly monitored the total energy conservation error during the 
entire simulation, finding a maximum relative energy conservation error of $4.02 \cdot 10^{-7}$.

\begin{figure}[!htbp]
	\begin{center}
		\begin{tabular}{cc} 
			\includegraphics[trim=40 40 40 50,clip,width=0.4\textwidth]{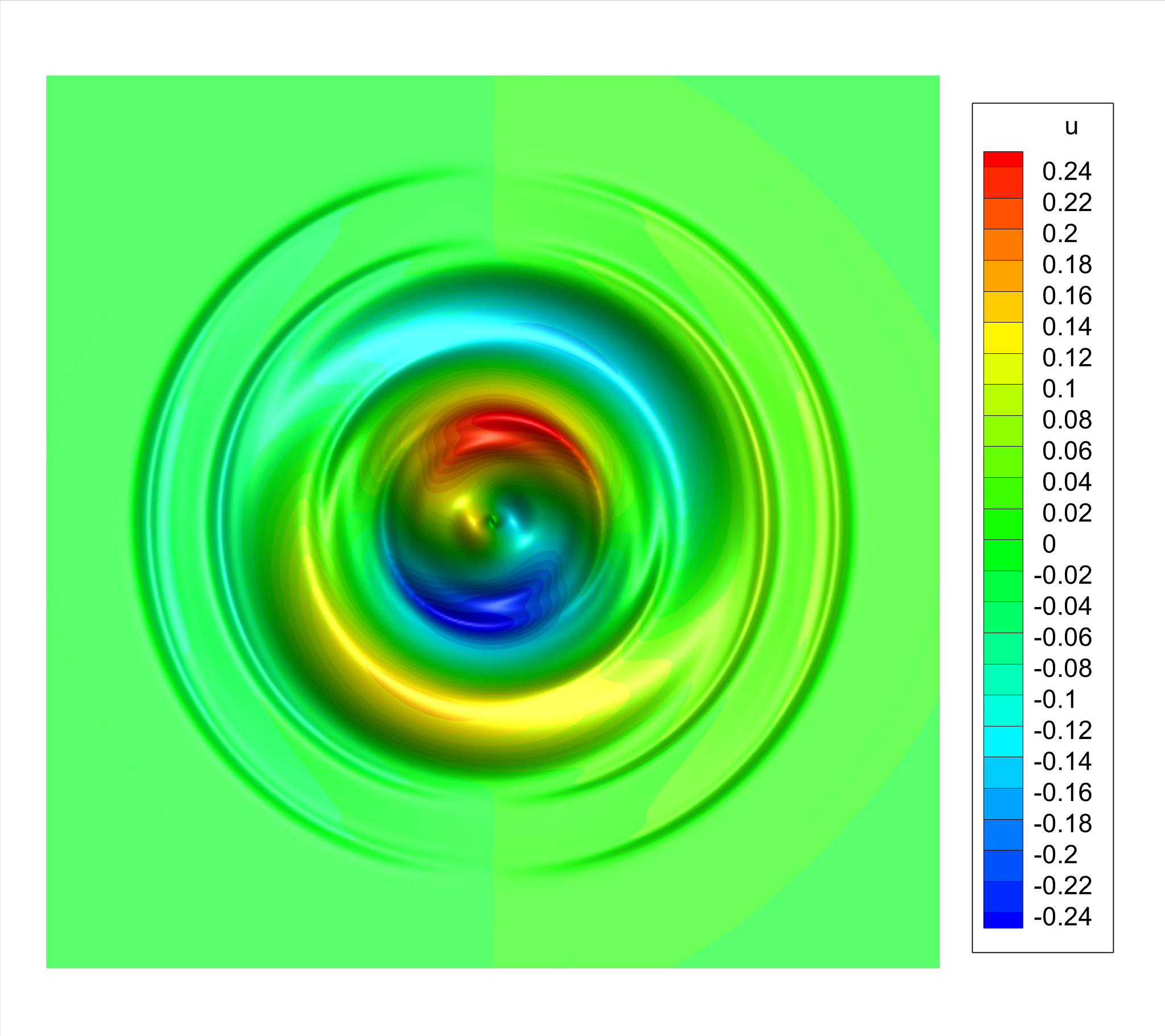}    
			& 
			\includegraphics[trim=40 40 40 50,clip,width=0.4\textwidth]{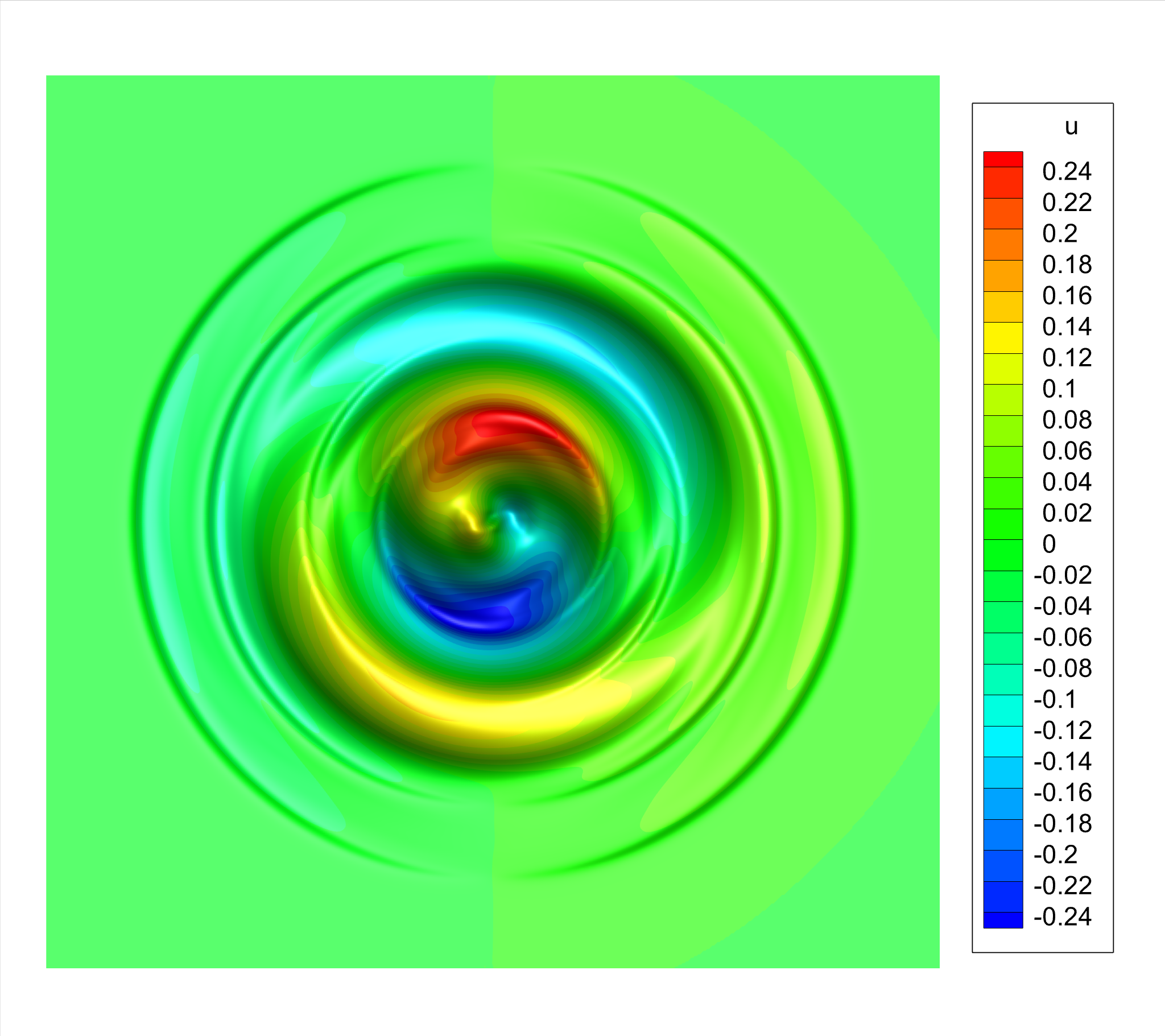}   
		\end{tabular}
		\vspace*{-3mm}
		\caption{Velocity component $v_1$ for the solid rotor test problem at time $t=0.3$ obtained by solving \eqref{eqn.GPR} with the new HTC scheme (left) and by using a classical 
			MUSCL scheme (right). }  
		\vspace*{-4mm}
		\label{fig.solidrotor}
	\end{center}
\end{figure}

\subsection{Double shear layer} 
\label{sec.dsl} 

In this section we present numerical results for the double shear layer test, see \cite{Bell1989,GPRmodel,Hybrid1,Hybrid2}. 
The computational domain is $\Omega=[0,1]^2$ with periodic boundary conditions everywhere. The initial condition is given by 
%\begin{equation*}
%v_1=\left\{
%\begin{array}{l}
%\tanh\left( \tilde{\rho} (y-0.25) \right), \qquad \textnormal{ if } y \leq 0.5, \\
%\tanh\left( \tilde{\rho} (0.75-y) \right), \qquad \textnormal{ if } y > 0.5,
%\end{array}
%\right.
%\end{equation*} 
$v_1 = \tanh\left( \tilde{\rho} (y-0.25) \right)$ for $y \leq 0.5$ and $v_1 = \tanh\left( \tilde{\rho} (0.75-y) \right)$ if $ y > 0.5$, $v_2= \delta \sin(2\pi x), v_3 = 0, \rho  = \rho_0 = 1,   p = 10^2 / \gamma,  
\AAA=\mathbf{I}, \mathbf{J}=\mathbf{0}$, 
with $\delta=0.05$ and  $\tilde{\rho}=30$. The remaining parameters of the GPR model are set to $\nu= \mu / 
\rho_0 = 2 \cdot 10^{-3}$, $\gamma = 1.4$, 
$\rho_0=1$, $c_v=1$, $c_s=8$, $c_h=2$ and $\tau_2 = 4 \cdot 10^{-3}$. \textcolor{black}{The characteristic Mach number of the flow resulting from this setup is $M=0.1$.} 
Calculations are performed 
with the new HTC scheme up to a final time of $t=1.8$. The computational grid is composed of 
$4000 \times 4000$ control volumes and the numerical viscosity is chosen as 
$\epsilon= 1 \cdot 10^{-6}$, hence three orders of magnitude lower than the physical one.
In Fig. \ref{fig.dslrot} the results obtained with the new HTC scheme 
are compared with a numerical reference solution that is based on the solution of the 
incompressible Navier-Stokes equations using a hybrid FV/FE method on a triangular grid made of 
$2097152$ elements ($N_{x} =1000$ divisions along each boundary), see 
\cite{busto2018projection,Hybrid1,Hybrid2} for  details. 
The flow dynamics has already been described in \cite{Bell1989,Hybrid1,Hybrid2,GPRmodel,SIGPR} and 
can be summarized by the development of several vortices from the initially perturbed shear layers. 
The agreement between the Navier-Stokes reference solution and the numerical solution of the GPR model computed with the new HTC schemes is rather good. 
In Fig. \ref{fig.dslA} we present the temporal evolution of the distortion field component $A_{12}$, which is qualitatively similar to the results shown in \cite{GPRmodel}, but for a lower physical viscosity $\mu$.  The maximum relative conservation error of the total energy monitored during the simulation for the semi-discrete HTC scheme was $7.12 \cdot 10^{-7}$. 
\textcolor{black}{Due to the low numerical viscosity of $\epsilon = 10^{-6}$ and fine mesh, one 
can observe small structures developing in the distortion field $\mathbf{A}$, which we would like 
to demonstrate in Fig. \ref{fig.dslA}.  }

\begin{figure}[!htbp]
	\begin{center}
		\begin{tabular}{cc} 
			\includegraphics[trim=20 40 40 50, clip,width=0.45\textwidth]{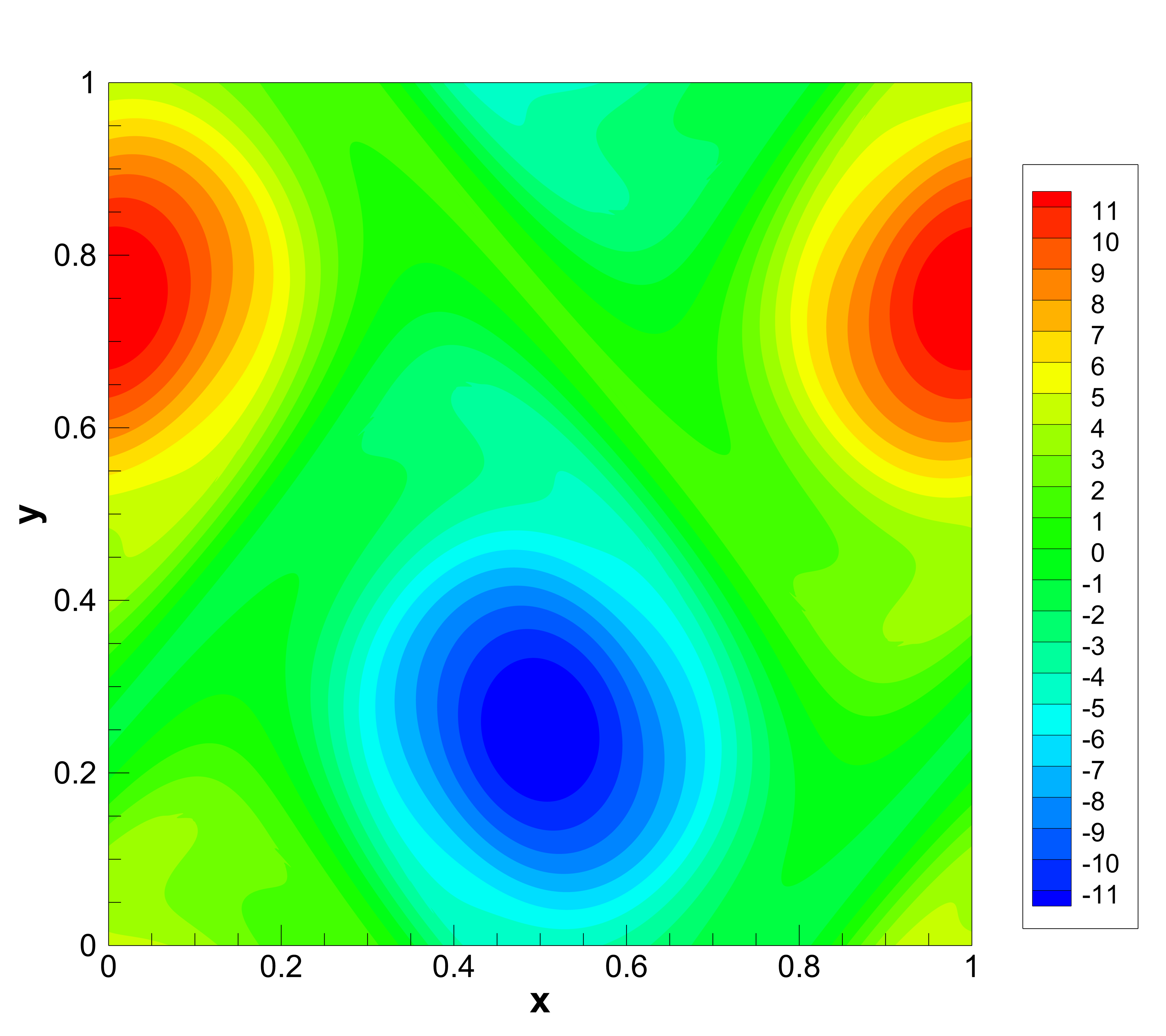}  & 
			\includegraphics[trim=20 40 40 50, 
			clip,width=0.45\textwidth]{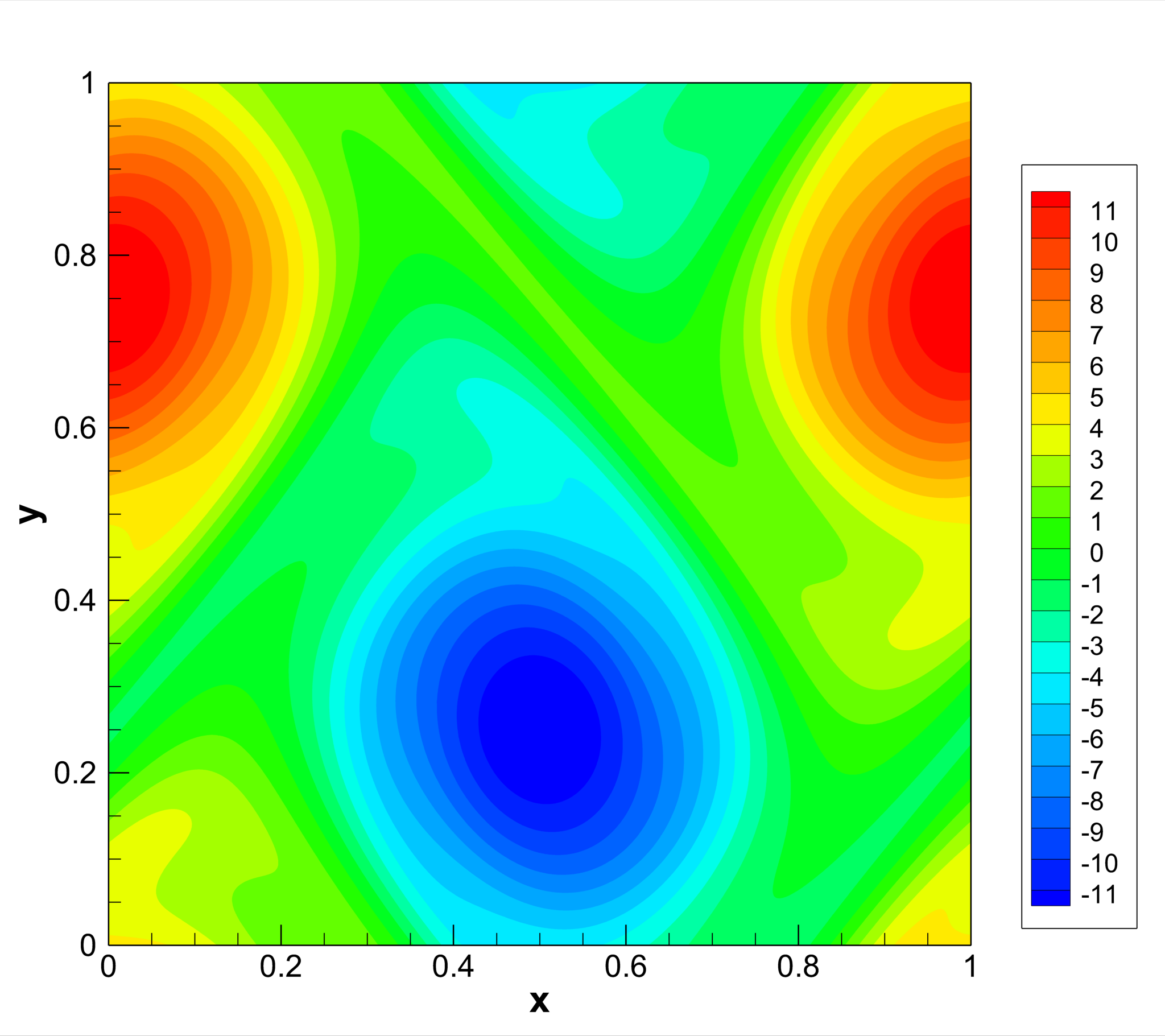}   
		\end{tabular}
		\vspace*{-3mm}
		\caption{Vorticity contours for the double shear layer with a viscosity of $\mu=2 \cdot 10^{-3}$ 
			at time $t=1.8$. 
			Left: numerical solution of the GPR model obtained with the new thermodynamically compatible finite volume scheme.  
			Right: reference solution obtained by solving the incompressible Navier-Stokes equations with the 
			staggered semi-implicit hybrid FV/FE scheme \cite{busto2018projection,Hybrid1,Hybrid2}. 
			} 
		\vspace*{-2mm}
		\label{fig.dslrot}
	\end{center}
\end{figure}

\begin{figure}[!htbp]
	\begin{center}
		\begin{tabular}{cc}
			\includegraphics[trim=20 40 40 50,clip,,width=0.45\textwidth]{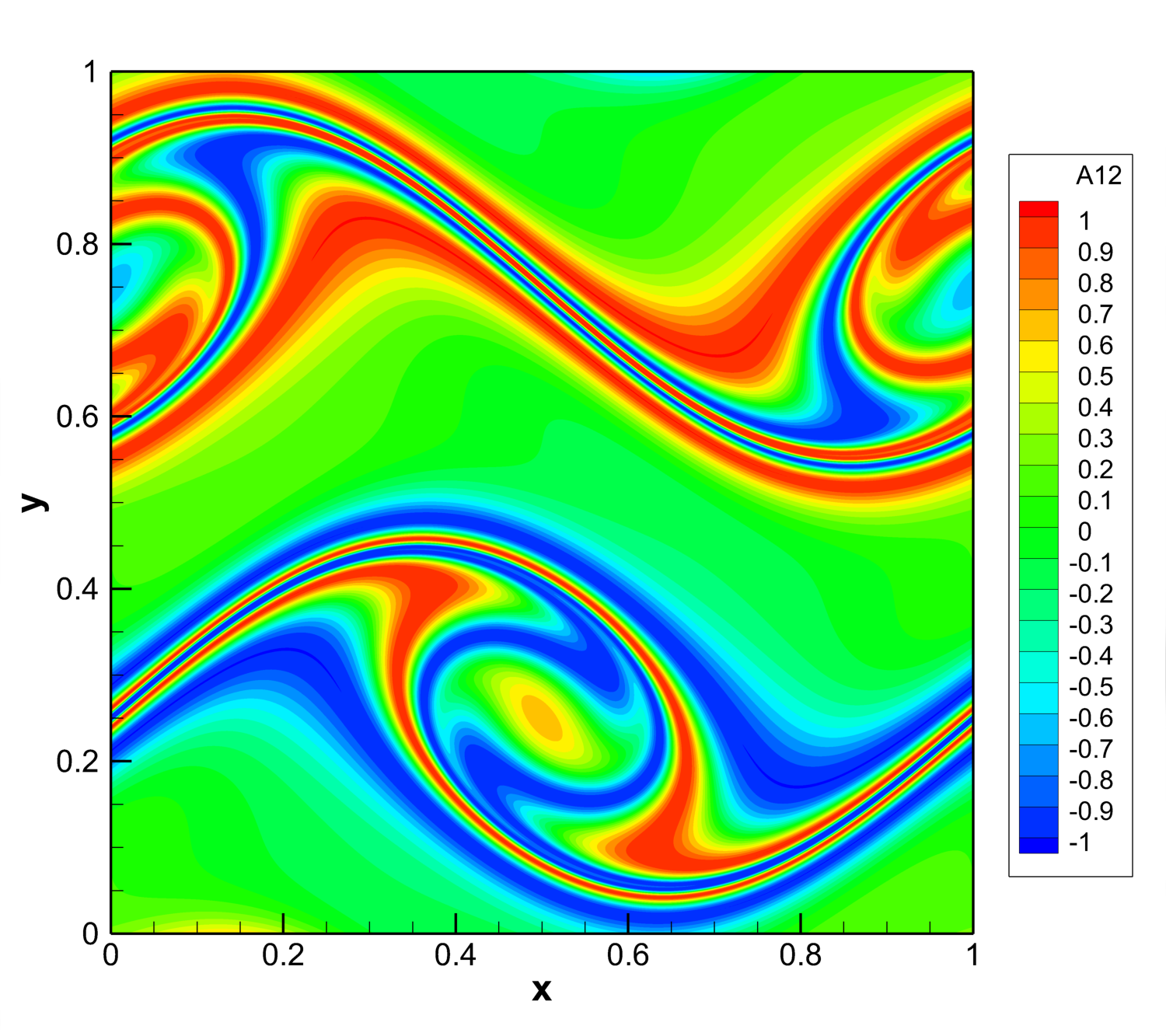} 
			& 
			\includegraphics[trim=20 40 40 
			50,clip,,width=0.45\textwidth]{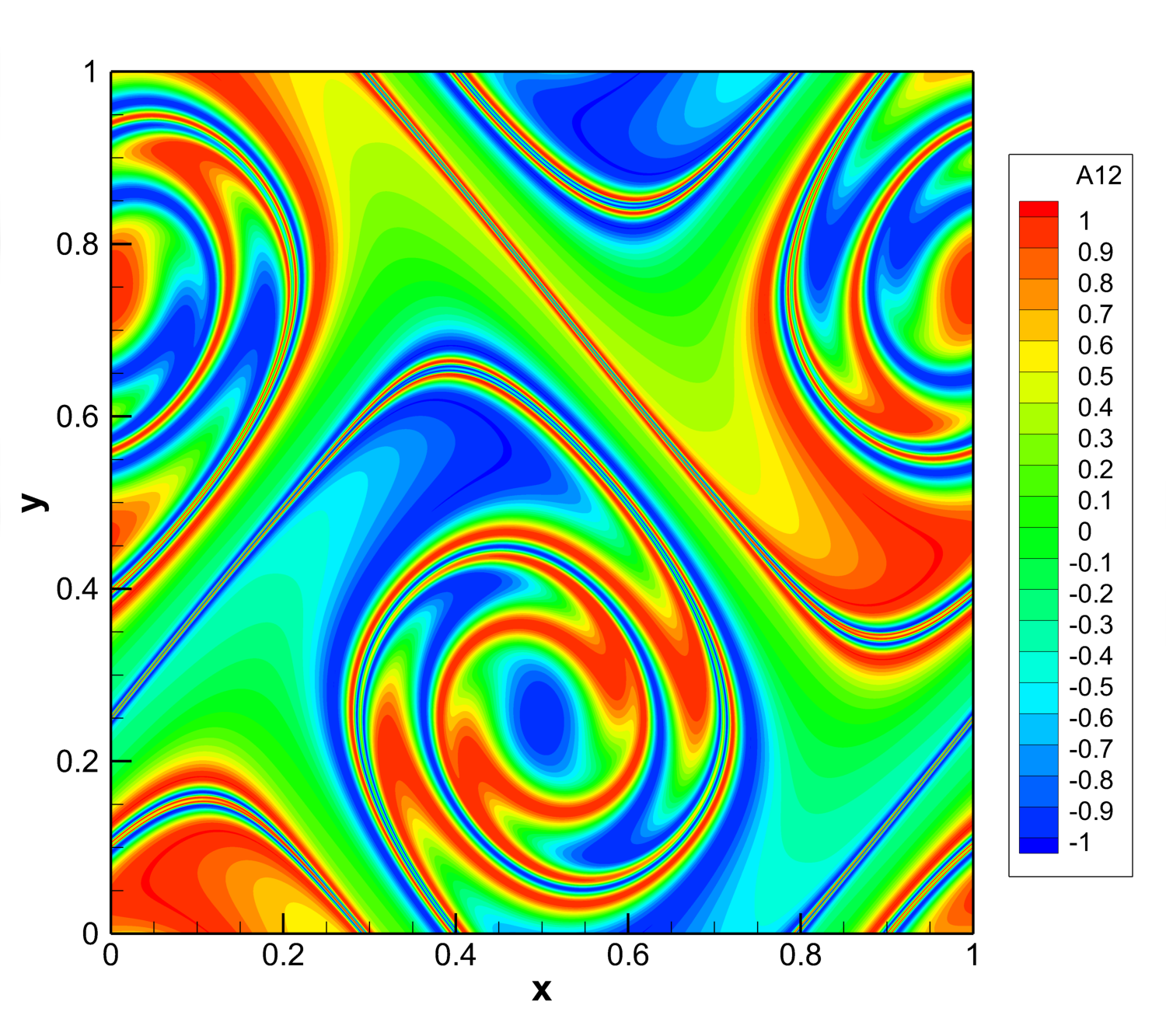}   
		\end{tabular}
		\vspace*{-2mm}
		\caption{Distortion field component $A_{12}$ for the double shear layer problem at times 
		$t=1.2$ and $t=1.8$ obtained by solving the GPR model ($\mu = 2 \cdot 10^{-3}$) with the 
		HTC scheme. }
		\vspace*{-2mm}  
		\label{fig.dslA}
	\end{center}
\end{figure}

\subsection{Lid-driven cavity} 

As last numerical test case for the fluid limit of the model \eqref{eqn.GPR} we present the lid-driven cavity problem, see \cite{Ghia1982}, which can be used to validate compressible flow solvers in the low Mach number regime, see e.g. \cite{TavelliDumbser2017,Hybrid1,Hybrid2} and which was already successfully solved with the GPR model in  \cite{GPRmodel,SIGPR}, but the schemes used in \cite{GPRmodel,SIGPR} were not thermodynamically compatible.  
The computational domain is $\Omega = [0,1] \times [0,1]$ and the initial condition is 
set to $\rho=1$, $\mathbf{v}=\mathbf{0}$, $p=10^2$,  $\AAA=\mathbf{I}$ and $\mathbf{J}=\mathbf{0}$. 
We 
furthermore 
set $\gamma=1.4$, $c_v = 1$, $c_s = 8$, $\rho_0=1$ and $c_h=2$,  $\tau_2=10^{-2}$ and $\mu=10^{-2}$ 
so that the Reynolds number of the test problem is $Re=100$. 
The lid velocity on the upper boundary is set to $\mathbf{v}=(1,0,0)$, while on all other boundaries $\mathbf{v}=\mathbf{0}$ is imposed. \textcolor{black}{The Mach number of this test is about $M=0.08$.} 
The new semi-discrete HTC scheme is run until $t=10$ using $256 \times 256$ elements and a constant artificial viscosity of $\epsilon=10^{-3}$. The numerical results are shown in Fig. \ref{fig.cavity}, where also a comparison with 
the Navier-Stokes reference solution of Ghia \textit{et al.} \cite{Ghia1982} is provided. We note an excellent agreement 
between the numerical solution of the GPR model and the incompressible Navier-Stokes reference solution. 
\begin{figure}[!htbp]
	\begin{center}
		\begin{tabular}{cc} 
			\includegraphics[trim=20 50 60 50,clip,width=0.45\textwidth]{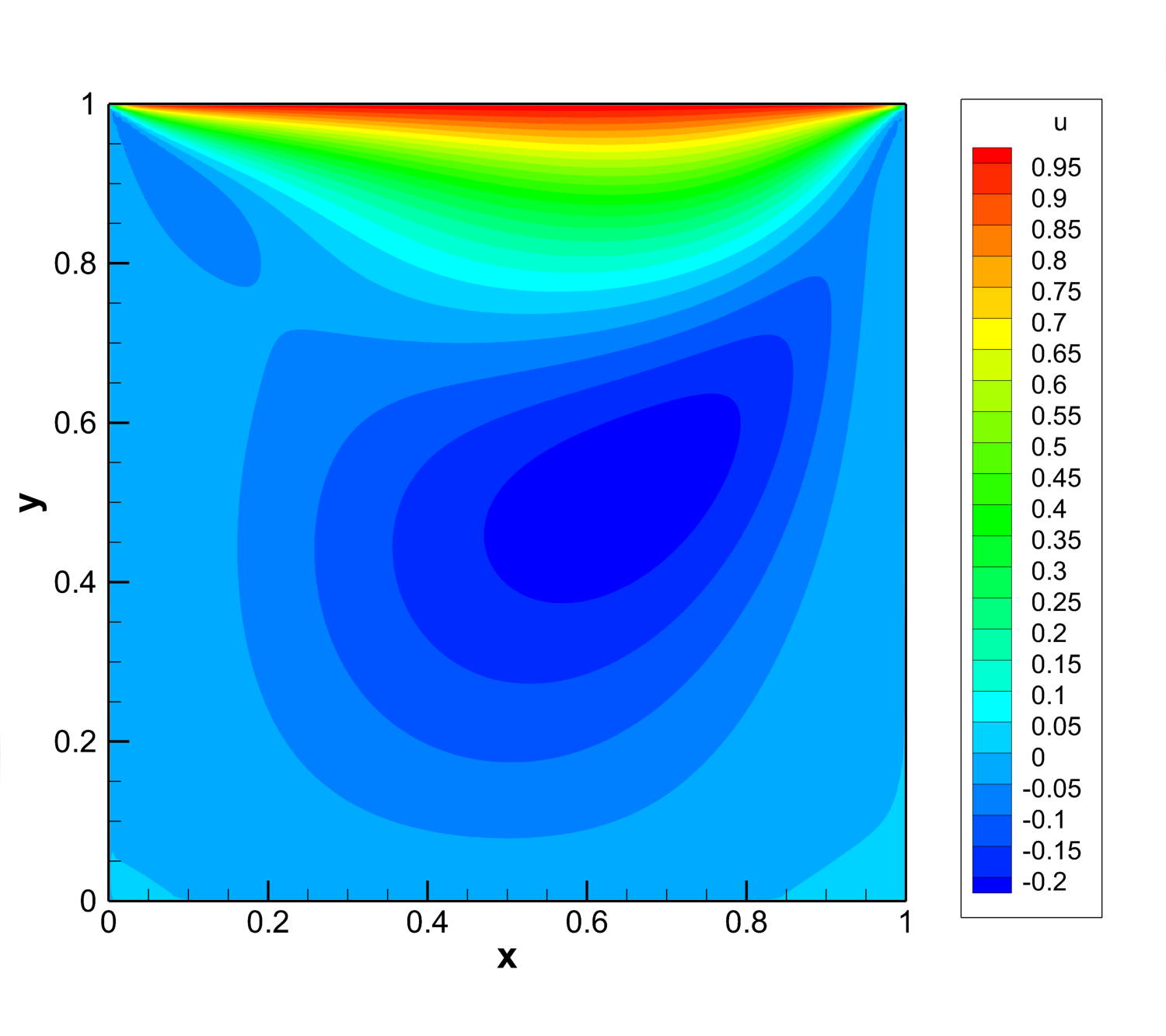}  & 
			\includegraphics[trim=20 20 20 
			20,clip,width=0.45\textwidth]{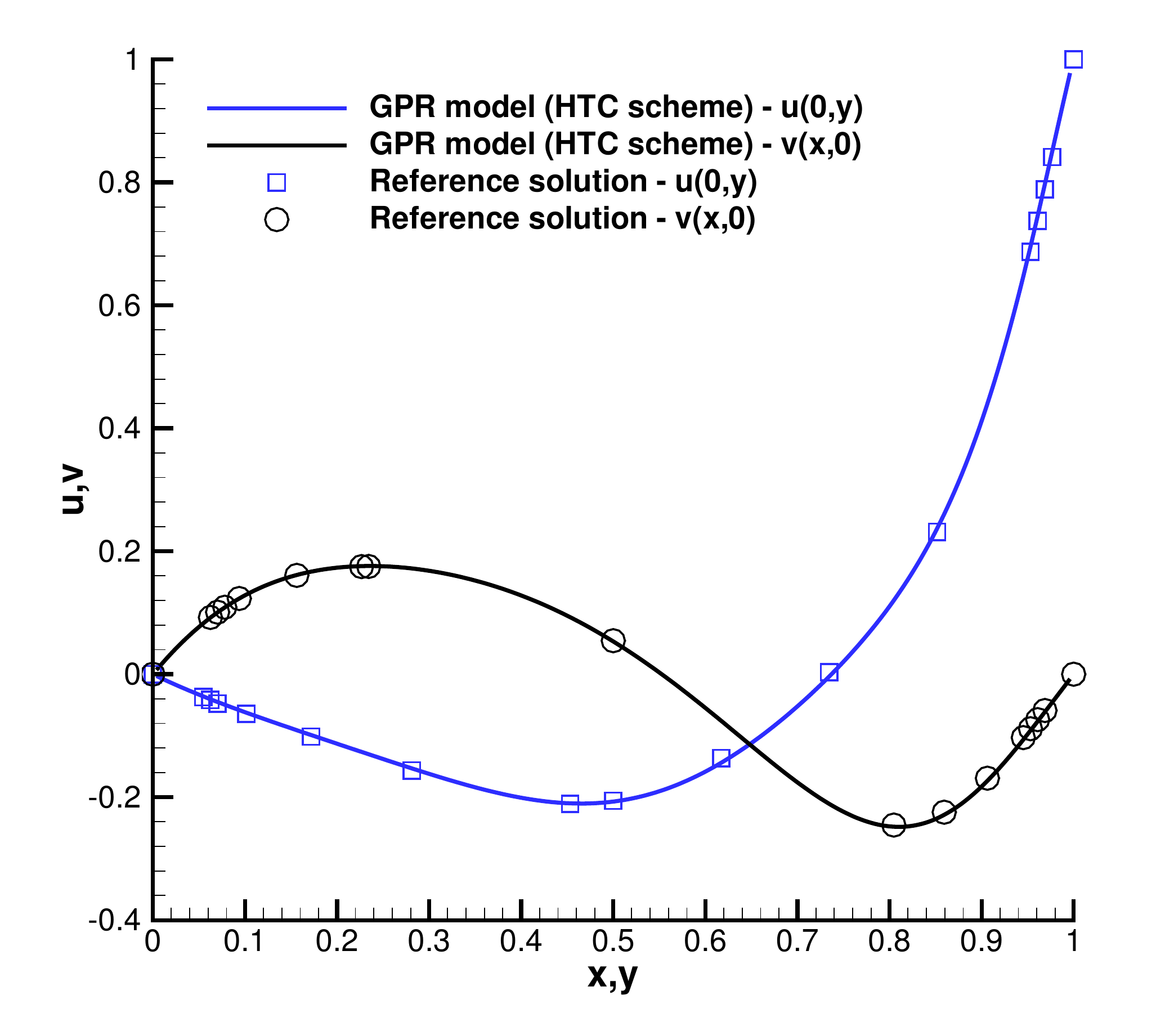}   
		\end{tabular} 
		\vspace*{-2mm}
		\caption{Lid-driven cavity at Reynolds number $Re=100$. Results obtained at time $t=10$ with the new HTC scheme applied to the GPR model. Color contours of the velocity component $v_1$ (left) and comparison of the velocity components $v_1$ and $v_2$ on 1D cuts along the $x$ and $y$ axis with the reference solution of Ghia \textit{et al.} \cite{Ghia1982} (right).} 
		\label{fig.cavity}
		\vspace*{-2mm}
	\end{center}
\end{figure}

\section{Conclusions}
\label{sec.Conclusions}

In this paper, we have presented two novel thermodynamically compatible finite volume schemes for first order hyperbolic PDE systems (HTC schemes). The first method is a semi-discrete finite volume scheme for the unified first order hyperbolic model of solid and fluid mechanics that goes back to the work of Godunov, Peshkov and Romenski on symmetric hyperbolic and thermodynamically compatible (SHTC) systems, see \cite{God1961,GodunovRomenski72,Rom1998,GodRom2003,PeshRom2014}. 
We have furthermore introduced a new fully-discrete HTC scheme for the compressible Euler 
equations, establishing a fully discrete analogy of the continuous framework introduced by Godunov 
in \cite{God1961}.  All schemes under consideration in this paper have in common that they directly 
discretize the \textit{entropy inequality} rather than the usual total energy conservation law. 
Instead, total energy conservation is obtained at the discrete level as a mere \textit{consequence} 
of a suitable and  thermodynamically compatible discretization of all the other equations. As such, 
the new schemes can be proven to be nonlinearly marginally stable in the energy norm and they 
furthermore satisfy a discrete entropy inequality \textit{by construction}. The new HTC schemes 
have been applied to several test  problems for fluid and solid mechanics, obtaining an excellent 
agreement with available reference solutions. 
\textcolor{black}{In future work, we will investigate the possible use of \textit{symplectic} time integrators in order to preserve exact total energy conservation of our new semi-discrete thermodynamically compatible scheme also on the fully discrete level.}
We also plan an extension to higher order in space at the aid of thermodynamically compatible discontinuous Galerkin (DG) finite element schemes, similar to entropy compatible DG schemes introduced in \cite{GassnerEntropyGLM,ShuEntropyMHD2,GassnerSWE} for the shallow water equations and magnetohydrodynamics (MHD), \textcolor{black}{as well as an extension to general unstructured meshes}.   
Another major challenge left to future work is the development of HTC schemes that are not only thermodynamically compatible, but which are also able to preserve the curl \textit{involution constraints} of the governing PDE system \textit{exactly} at the semi-discrete level and that are also consistent with the low Mach number limit of the equations. In this context we will consider staggered semi-implicit finite volume schemes \cite{SIGPR}, as well as staggered semi-implicit hybrid finite volume / finite element methods \cite{busto2018projection,Hybrid1,Hybrid2} and staggered DG schemes \cite{TavelliDumbser2017,SIDGConv}, which are not yet thermodynamically compatible in the sense of the HTC schemes presented in this paper. 

\section*{Acknowledgments}

S.B., M.D. and I.P. are members of the INdAM GNCS group and acknowledge the financial support received from  
the Italian Ministry of Education, University and Research (MIUR) in the frame of the Departments of Excellence  Initiative 2018--2022 attributed to DICAM of the University of Trento (grant L. 232/2016) and in the frame of the 
PRIN 2017 project \textit{Innovative numerical methods for evolutionary partial differential equations and  applications}. S.B. was also funded by INdAM via a GNCS grant for young researchers and by an \textit{UniTN starting grant} of the University of Trento. E.R., M.D. and I.P. were supported by the Mathematical Center in Akademgorodok under agreement No. 075-15-2019-1613 with the Ministry of Science and Higher Education of the Russian Federation.
\textcolor{black}{The authors would like to acknowledge support from the Leibniz Rechenzentrum (LRZ) in Garching, Germany, for granting access to the SuperMUC-NG supercomputer under project number pr63qo.}
\textcolor{black}{The authors are very grateful to the two anonymous referees for their constructive and insightful comments, which helped to improve the clarity and quality of this paper. }

\bibliographystyle{plain}
\bibliography{biblio}
\end{document}